\documentclass[a4paper, 12pt]{article}
\usepackage{amsmath}
\usepackage{amssymb}
\usepackage{amsthm}
\usepackage{mathtools}
\usepackage[final]{showkeys}
\usepackage[utf8]{inputenc}
\usepackage[T1]{fontenc}
\usepackage{lmodern}
\usepackage[a4paper, margin=1in]{geometry}

\usepackage[colorlinks=true, citecolor =violet, linkcolor = violet, urlcolor  = violet]{hyperref}

\usepackage{graphicx}
\usepackage{tikz}
\usepackage{pgfplots}
\pgfplotsset{compat=1.17}

\usepackage{changes}

\newcommand{\R}{\mathbb{R}}

\newcommand{\fa}{\forall}

\newcommand{\commentout}[1]{}

\usepackage{enumitem}
\usepackage{xcolor}
\usepackage{microtype}
\usepackage{physics}
\usepackage{array}
\usepackage{booktabs}
\newtheorem{theorem}{Theorem}
\newtheorem{remark}[theorem]{Remark}
\newtheorem{definition}{Definition}
\newtheorem{proposition}{Proposition}
\newtheorem{corollary}[theorem]{Corollary}
\title{Asymptotic dynamics of inhibitory networks for the NNLIF Model in the large-delay limit}
\date{\today}

\author{
Cl\'ement Rieutord%
\:\thanks{Corresponding author, Sorbonne Universit\'e, CNRS, Universit\'e de Paris Cit\'e,
Laboratoire Jacques-Louis Lions (LJLL), F-75005 Paris, France.
\newline\texttt{clement.rieutord@sorbonne-universite.fr}}
\and
Delphine Salort%
\thanks{Sorbonne Universit\'e, CNRS, Universit\'e de Paris Cit\'e,
Laboratoire Jacques-Louis Lions (LJLL), F-75005 Paris, France.
\newline\texttt{dsalort@gmail.com}}
}

\begin{document}

\maketitle

\begin{abstract}
 We investigate the impact of large synaptic delays on the emergence of periodic dynamics in inhibitory neuronal networks, within the framework of the NNLIF model. 
\newline 
Inspired by the work of \cite{CaceresCanizo2024}  where the notion of pseudo-equilibria was introduced and developed, and by our earlier analysis in \cite{zbMATH08113917},  we show that, as the delay tends to infinity, solutions of sufficitently inhibitory networks oscillate between distinct pseudo-equilibria over any finite time interval. Employing the Doeblin–Harris method, we rigorously establish a local convergence in the Cesàro mean toward a limit function determined solely by these pseudo-equilibria.
\end{abstract}

\noindent{\makebox[1in]\hrulefill}\newline

\textit{Keywords:}Integrate-and-Fire, Fokker-Planck equation, Periodic solutions, Doeblin-Harris Method, Mathematical neuroscience\\
\textit{Mathematics Subject Classification.}  35B40, 35Q84, 35Q92 
,35B10
\newpage

\tableofcontents

\section{Introduction}

\subsection{Context}
To investigate the role of synaptic delay in the dynamics of certain inhibitory neuronal networks, we focus on the NNLIF model introduced in \cite{BrHa}. In recent years, the analysis of this model has experienced significant development, leading to a well-established mathematical theory combining both deterministic and probabilistic techniques, even if several fundamental questions remain open. We refer to \cite{RouxCar} for a comprehensive survey on this class of models and, more generally, on PDE models arising in neuroscience.
In its formulation with instantaneous transmission of neuronal activity, the model reads as follows:
\begin{equation}\label{eq:FP général}
\left\{
\begin{aligned}
& \frac{\partial u}{\partial t}(t,v) + \frac{\partial}{\partial v}\big((bN(t) - v)u(t,v)\big) - \frac{\partial^2 u}{\partial v^2}(t,v) = \delta_{V_R}(v)N(t), \quad t \geq 0, , v \leq V_F, \\
& u(t,V_F) = 0, \\
& N(t) = -\frac{\partial u}{\partial v}(t,V_F).
\end{aligned}
\right.
\end{equation}
The function $u(t,v)$ denotes the probability density of finding a neuron in a  network at time $t$ with membrane potential $v \in (-\infty,V_F)$. The constant $V_F$ represents the firing threshold, while $V_R$ denotes the reset potential.
The Dirichlet boundary condition at $v = V_F$ models the firing mechanism: neurons reaching the threshold instantaneously emit a spike and are immediately reset to the potential $V_R$. In particular, no mass accumulates at the threshold.
\newline 
The firing rate $N(t)$, describing the flux of neurons discharging at time $t$, is defined by the outgoing boundary flux
\[
N(t) = -\partial_v u(t,V_F).
\]
This flux is reinjected at the reset potential through the source term $\delta_{V_R}(v)N(t)$, which formally ensures conservation of the total mass:
\[
\frac{d}{dt} \int_{-\infty}^{V_F} u(t,v)\, dv = 0.
\]
Finally, the parameter $b \in \mathbb{R}$ measures the strength of synaptic connectivity in the network: $b>0$ corresponds (on average) to excitatory interactions, whereas $b<0$ describes inhibitory networks; $b=0$ corresponds to the case without interconnections. 

 \noindent The linear case $b=0$ is well understood: one can show exponential convergence toward a unique stationary state \cite{CCP}. In contrast, the nonlinearity induced by the term $bN(t)$ in the excitatory case ($b>0$) seems to preclude the emergence of periodic solutions. Instead, only two distinct scenarios  seem to occur numerically: either asymptotic convergence toward a stationary state or finite-time blow-up, but never periodic behavior \cite{caceres2018global}.
\newline 
In order to understand the more complex dynamics that may arise in the excitatory regime, several works have focused on asymptotic convergence toward a unique stationary state in the case of weak nonlinearity, on the local stability analysis of stationary states, and on the possible extension of solutions beyond blow-up when it occurs. It is important to emphasize that blow-up may prevent global-in-time well-posedness, so that solutions are not always globally defined \cite{CCP}, \cite{CaPe}.

 \noindent The inhibitory case is  more tractable, as the problem is globally well-posed for all $b<0$ \cite{CGGS}. Numerical simulations suggest that, for any $b \le 0$, the solution converges toward its unique stationary state, with no other types of dynamics observed. However, a rigorous proof of this convergence remains an open problem in the general case.  Indeed, global asymptotic convergence toward a unique stationary state has only been established for small $|b|$ (weak interconnections) \cite{CPSS}, whereas it has been shown that, in the general inhibitory case, the stationary state is locally stable \cite{zbMATH08030978}.

 \noindent 
One possible explanation for the lack of variety in the asymptotic dynamics, even though some richness could in principle be hidden in the blow-up behavior, is the assumption of instantaneous transmission of neuronal activity. We then introduce  a synaptic delay as follows
\begin{equation}\label{eq:FP delay}
\left\{
\begin{aligned}
& \frac{\partial u}{\partial t}(t,v) + \frac{\partial}{\partial v}\big((bN(t-d) - v)u(t,v)\big) - \frac{\partial^2 u}{\partial v^2}(t,v) = \delta_{V_R}(v)N(t), \quad t \geq 0,  v \leq V_F, \\
& u(t,V_F) = 0 \quad t\geq0, \\
& N(t) = -\frac{\partial u}{\partial v}(t,V_F)\quad  t\geq0,\\
&  u(0,v) \in \mathcal{M}^+_v((-\infty, V_F)), \ N(t-d) = N_{\text{ini}}, \quad t\in [0,d],
\end{aligned}
\right.
\end{equation}
 where \[
\mathcal{M}^+_v((-\infty, V_F)) = \left\{ p \in L_v^1((-\infty, V_F)), \ p \geq 0 \text{ and } \int_{-\infty}^{V_F} p(v) \, dv = 1 \right\}
\] is
the set of probability densities on the membrane potential variable and 
\[
L_v^1((-\infty, V_F)) = \left\{ p \in L^1((-\infty, V_F]), \  \int_{-\infty}^{V_F} (1 + v^2) |p(v)| \, dv < \infty \right\}
\]
is the set of integrable functions with finite second moment.
This formalism not only resolves the well-posedness issues that arise in the excitatory case, where blow-up can occur with instantaneous transmission, but also significantly broadens the spectrum of possible dynamics. However, numerical studies have suggested that, in the excitatory regime, delay alone 
is insufficient to generate periodic solutions \cite{caceres2018global}. In this setting, 
periodic solutions were numerically observed in \cite{zbMATH07047481} only upon 
adding a refractory period. In contrast, in inhibitory networks, simulations indicate that once both the delay and the connectivity parameter $b$ are sufficiently large, more complex dynamics emerge, including convergence toward periodic profiles.
A first mathematical analysis considered the case of fixed delay $d>0$ and very strong inhibition ($b \to -\infty$) \cite{IRSS2022}. Another study focused on the opposite asymptotic regime, where the delay tends to infinity while $b<0$ remains fixed \cite{CaceresCanizo2024}. In this latter work, the notion of pseudo-equilibria was introduced, allowing the extraction of the main heuristic mechanisms underlying the emergence of these periodic solutions, via oscillations between multiple pseudo-equilibria.
The objective of the present work is to provide a rigorous framework showing that, in the large-delay limit, the solution of our equation can indeed oscillate between multiple pseudo-equilibria, on arbitrary finite time intervals.

\subsection{Main results and strategy.}

 \noindent Before stating our main result and outlining the strategy of the proof, we first introduce the notion of pseudo-equilibria, following \cite{CaceresCanizo2024}, and recall the heuristic strategy they developed to describe how periodic solutions emerge from the interplay between these states. 
\begin{definition}\label{def:equilibria}[Pseudo-equilibria]
Let $J \geq 0$ and $ b \in \R$. The pseudo-equilibria associated with Equation~\eqref{eq:FP delay} and the parameter $J$ is the normalized nonnegative function on $(-\infty,V_F)$ defined by
\begin{equation}\label{eq:stat J}
   \overline{u}_b(J,v) = \Phi_b(J) \exp\Big(-\frac{(-bJ+v)^2}{2}\Big)
    \left[ \int_{\max(V_R, v)}^{V_F} \exp\Big(\frac{(-bJ+z)^2}{2}\Big) dz \right],
\end{equation}
where the normalization constant $\Phi_b(J)$ is given by
\[
\Phi_b(J) = \left( \int_{-\infty}^{V_F} \exp\Big(-\frac{(-bJ+v)^2}{2}\Big) 
\left[ \int_{\max(V_R, v)}^{V_F} \exp\Big(\frac{(-bJ+z)^2}{2}\Big) dz \right] dv \right)^{-1},
\]
and where $b$ denotes the connectivity parameter in Equation~\eqref{eq:FP delay}. 
\end{definition}
We recognize that this corresponds to the nonnegative normalized ($\int_{-\infty}^{V_F} \overline{u}(bJ,v)\, dv = 1$) stationary solution of the linearized system
\begin{equation}\label{eq:linearized}
\left\{
\begin{aligned}
& \frac{\partial u}{\partial t}(t,v) + \frac{\partial}{\partial v}\big((bJ - v) u(t,v)\big) - \frac{\partial^2 u}{\partial v^2}(t,v) = \delta_{V_R}(v) N(t), \quad t \geq 0, \ v \leq V_F, \\
& u(t,V_F) = 0, \quad t \geq 0, \\
&   u(0,v) \in \mathcal{M}^+_v((-\infty, V_F)) \\
& N(t) = -\frac{\partial u}{\partial v}(t,V_F), \quad t \geq 0, \\
\end{aligned}
\right.
\end{equation}
and that $\Phi_b(J)$ is the flux of neurons associated to the stationary state. From a heuristic point of view, if we formally assume $d = +\infty$ in Equation \eqref{eq:FP delay}, the term $N(t-d)$ is replaced by a constant input $N_{\text{ini}} = J$, and the system reduces to the linear equation \eqref{eq:linearized}. Focusing on the asymptotic dynamics of the solution to Equation \eqref{eq:linearized}, its linear structure allows one to prove, using entropy methods or the Doeblin approach, that the solution and the flux of neurons $N(t)$ converge exponentially fast toward its unique stationary state, which coincides with the pseudo-equilibrium defined above. However, there is a priori no reason for the relation $\Phi_b(J) = J$ to hold. If this heuristic argument is iterated, we find, in the limit $d \to +\infty$, with the linear equation \eqref{eq:linearized} where $J$ has been replaced by $\Phi_b(J)$. By repeating this process, one obtains a recurrence sequence $(\Phi_b^n(J))_{n \in \mathbb{N}}$. If this sequence possesses a stable $2$-cycle, it implies that in the $d \to +\infty$ limit (in the sense defined above), the sequence  $(\Phi_b^n(J))_{n \in \mathbb{N}}$ can oscillates asymptotically between two distinct values. In this regard,  the following proposition proved in \cite{CaceresCanizo2024}, holds:
\begin{proposition}
\label{prop:phi_decreasing}
Let $b < 0$ and let $\overline{N}$ be the unique solution to $\Phi_b(N)=N$. 
There exists a critical value $b^* < 0$ such that:
\begin{enumerate}
    \item If $b^* \leq b < 0$, then for any $N_{\mathrm{ini}} \ge 0$, the sequence 
    $\big(\Phi_b^n(N_{\mathrm{ini}})\big)_{n \in \mathbb{N}}$ 
    converges to $\overline{N}$.
\item If $b < b^*$, there exist two distinct values $N^-$ and $N^+$ such that 
    $N^- < \overline{N} < N^+$, and for any $N_{\mathrm{ini}} \in [0, +\infty) \setminus \{\overline{N}\}$, 
    the sequence $\big(\Phi_b^n(N_{\mathrm{ini}})\big)_{n \in \mathbb{N}}$ 
    converges to the 2-cycle $(N^-, N^+)$.
\end{enumerate}
\end{proposition}
This proposition constitutes the core mechanism underlying the formal emergence of periodic solutions when $b < b^*$. 
In contrast to the above heuristic approach, our strategy consists in fixing a finite number of iterations $M \in \mathbb{N}$ and performing a rigorous analysis of the solution over the time interval $[0, Md]$. Within each block $[nd, (n+1)d]$ for $0 \le n < M$, we establish estimates that are uniform with respect to the delay $d$, although they may depend on the total number of iterations $M$. By rescaling time as $\tau = t/d$, we show that as $d \to +\infty$, and the number of iterations $M$ goes to $+\infty$, the solution $u_d$ and the flux $N_d$ jump from one to the next at each integer value of $\tau$. This allows us to recover the heuristic asymptotic behavior  in \cite{CaceresCanizo2024} by proving that the rescaled dynamics converge toward a step function governed by the discrete iteration of $\Phi_b$. This result is formalized in the following theorem:
\begin{theorem}[Convergence in the large-delay regime]
\label{th:cv_long_delay}
Assume
\begin{equation}\label{eq: condition V_R, V_F}
    0<V_R<V_F\quad \text{and}\quad b<0.
\end{equation}
Let $u^0 \in \mathcal{M}^+_v((-\infty,V_F]) \cap \mathcal{C}^1((-\infty,V_F])$,  and let $N_{\mathrm{ini}} \ge 0$ be initial data for Equation~\eqref{eq:FP delay}.
Let $u_d(t)$ denote the corresponding solution and define the associated
activity function by
\[
N_d(t) := -\partial_v u_d(t,V_F).
\]
Define the time-rescaled functions
\[
\tilde{u}_d(\tau) := u_d(d\tau),
\qquad
\tilde{N}_d(\tau) := N_d(d\tau).
\]
For $\tau \ge 0$, define the limit profiles
\[
\tilde{u}_\infty(v,\tau,N_{\mathrm{ini}})
=
\sum_{n=0}^{\infty}
\overline{u}\!\left(\Phi_b^{\,n}(N_{\mathrm{ini}}),v\right)
\,\mathbf{1}_{[n,n+1)}(\tau),
\]
\[
\tilde{N}_\infty(\tau,N_{\mathrm{ini}})
=
\sum_{n=0}^{\infty}
\Phi_b^{\,n}(N_{\mathrm{ini}})
\,\mathbf{1}_{[n,n+1)}(\tau),
\]
with the convention that $\Phi_b^{\,0}(N)=N$, so that the first
time-interval corresponds to the initial constant input. Then, for every $b \le 0$ and every $M \in \mathbb{N}$, the following convergences hold:
\begin{equation}
\label{eq:convergence_large_delay_u}
\lim_{d\to\infty}
\int_0^M
\big\|
\tilde{u}_d(\tau)
-
\tilde{u}_\infty(\tau,N_{\mathrm{ini}})
\big\|_{L^1_v}
\, d\tau
= 0,
\end{equation}
and
\begin{equation}
\label{eq:convergence_large_delay_N}
\lim_{d\to\infty}
\int_0^M
\big|
\tilde{N}_d(\tau)
-
\tilde{N}_\infty(\tau,N_{\mathrm{ini}})
\big|
\, d\tau
= 0.
\end{equation}
\end{theorem}
Combining Theorem~\ref{th:cv_long_delay} with Proposition~\ref{prop:phi_decreasing}, we deduce the following corollary. It provides a criterion ensuring that the solution, in the limit $n \to +\infty$, approaches a periodic solution whenever $b < b^*$, where $b^*$ is the critical value defined in Proposition \ref{prop:phi_decreasing}.
\begin{corollary}[Long-time dynamics in the large-delay regime]
\label{cor:double_limit_correct}
Let $b<0$ and let $(\tilde N_d)$ be as in
Theorem~\ref{th:cv_long_delay}.
Then the following holds.
\begin{enumerate}
\item If $b^* \le b < 0$, let $\overline N$ denote the unique fixed point of $\Phi_b$.
Then
\[
\lim_{M\to\infty}
\;
\lim_{d\to\infty}
\int_M^{M+1}
\big|
\tilde N_d(\tau) - \overline N
\big|
\, d\tau
=0.
\]
\item If $b<b^*$ and $N_{\mathrm{ini}}\neq \overline N$,
let $(N^-,N^+)$ denote the attracting $2$-cycle of $\Phi_b$.
Define
\[
N_{\mathrm{per}}(\tau)
:=
\lim_{n\to\infty}
\tilde N_\infty(\tau+2n,N_{\mathrm{ini}}),
\]
which is a $2$-periodic, piecewise constant function (depending on $N_{\mathrm{ini}}$) taking the values
$N^-$ and $N^+$.
Then
\[
\lim_{M\to\infty}
\;
\lim_{d\to\infty}
\int_M^{M+2}
\big|
\tilde N_d(\tau) - N_{\mathrm{per}}(\tau)
\big|
\, d\tau
=0.
\]
\end{enumerate}
\end{corollary}
The primary difficulty in proving Theorem \ref{th:cv_long_delay} arises in the subsequent iterations. Starting from the interval $[d, 2d]$, the delayed neuronal flux $N(t-d)$ is no longer constant, hence one cannot directly apply the standard proof of exponential convergence toward the stationary state for the model \eqref{eq:linearized}. However, we may expect that it quickly approaches the value $\Phi_b(J)$ exponentially fast in the sense of Theorem \ref{th:convergence exp u et N}. Heuristically, after a phase during which the input relaxes toward a constant value, the system is effectively governed by a nearly constant input. In other words, provided that $d$ is sufficiently large, the solution has enough time to relax almost to the stationary state associated with a fixed flux.
To make these heuristic arguments rigorous, we decompose the proof into three distinct parts. First, in the spirit of what was done in \cite{Ambrogi_2026}, we generalize the convergence theory by extending the proof of convergence toward a stationary state, initially established for a constant input $J$, to the case where the external input $J(t) \in L^\infty(0, +\infty)$ is time-dependent (see Theorem \ref{th:convergence doeblin-harris}). In this non-autonomous setting, one cannot expect convergence toward a fixed stationary state in general, but we show instead that any solution converges toward a common profile governed by the asymptotic dynamics of $J(t)$. Specifically, Theorem \ref{th:convergence exp u et N} establishes that if $J(t)$ converges sufficiently fast toward a constant $J$, then the solution $u(t), N(t)$ converges toward the pseudo-equilibrium associated with $J$, which is $\overline{u}(J), \Phi_1(J)$.
Second, to apply the previous results to our specific context, we establish $L^\infty$ estimates on the neuronal flux $N(t)$ that are uniform with respect to $d$, although they may depend on the iteration index $n$ on each interval $[nd, (n+1)d]$. At this stage of the proof, we require the initial data to be sufficiently regular and assumption \eqref{eq: condition V_R, V_F} to allow the construction of supersolutions as in \cite{CPSS}, which provides the necessary control on the flux $N(t)$ independently of $d$. Such method was also used in the two-dimensional case with partial diffusion (with an adaptation current) in \cite{Ambrogi_2026} and seems to be, up to our knowledge, the only available tool to get $L^\infty$ bound on the activity function $N(t)$. It should be noted that the constraints on $V_R$ and $V_F$, as well as the regularity requirements on the initial data, appear to be technical conditions inherent to our proof strategy rather than necessary requirements for the result to hold.  Finally, the last part of the proof consists in applying these results through a recursive procedure over the successive intervals to conclude the proof of Theorem \ref{th:cv_long_delay}. 

\vspace{0.5cm}

\noindent The remainder of this article is organized as follows. In Section \ref{sec:convinput}, following the approach of \cite{zbMATH07903423}, we establish a general result regarding exponential stability toward a unique profile for a non-autonomous equation with a given bounded input (see Theorem \ref{th:convergence doeblin-harris}). In section \ref{sec:lienu}, we subsequently apply this result to prove that if the input $J(t)$ converges exponentially fast toward a constant value $\bar{J}$ (up to time averaging), then the unique asymptotic profile for $J(t)$ coincides with that of the constant input $\bar{J}$ (see Theorem \ref{th:convergence exp u et N}). In Section \ref{sec:bornesunif}, in order to apply the theory developed in Section \ref{sec:convinput} to equation \eqref{eq:FP delay}, we use the method of supersolutions to show that, provided the initial data is sufficiently regular and $0 < V_R < V_F$, the neuronal flux remains bounded on each interval $[kd, (k+1)d]$ (see Theorem \ref{th:bound on N_u(t)}). Section \ref{sec:preuvethm} is devoted to the proof of Theorem \ref{th:cv_long_delay}, which relies on applying the results of the two previous sections via a recursive procedure. Furthermore, we provide numerical simulations to illustrate the convergence behavior established in Theorem \ref{th:cv_long_delay}. Finally, the last section is devoted to concluding remarks and perspectives for future work.


\section{Exponential stability non-autonomous system with bounded inputs}\label{sec:convinput}
Let $J $ be a given input defined on $\R^+$ satisfying the following assumption
\begin{equation}\label{eq:propriété J}
    J(t) \geq 0 \quad \text{and} \quad \sup_{t\geq 0} J(t) \coloneqq J_\infty < \infty.
\end{equation}
We consider the associated  linear Fokker-Planck equation given by
\begin{equation}\label{eq:FP linéaire instat}
\begin{cases}
    \displaystyle \frac{\partial u}{\partial t}(t,v) - \frac{\partial}{\partial v}\big((J(t) + v)u(t,v)\big) - \frac{\partial^2 u}{\partial v^2}(t,v) = \delta_{V_R}(v)N_u(t), & t \geq 0, \, v \leq V_F, \\[8pt]
    u(t,V_F) = 0, \\[4pt]
    \displaystyle N_u(t) = -\frac{\partial u}{\partial v}(t,V_F)\\
    \displaystyle u(0,v) = u^0(v).
\end{cases}
\end{equation}

\noindent  Let $(\Gamma(t,s))_{0 \leq s \leq t}$ be the evolution operator associated with the non-autonomous Fokker-Planck equation \eqref{eq:FP linéaire instat}. For any initial condition $u(s, \cdot) = u^s$ at time $s \geq 0$, the solution at time $t \geq s$ is given by $u(t, \cdot) = \Gamma(t, s) u^s$. This operator maps the initial distribution at time $s$ to the distribution at time $t$, effectively accounting for the influence of the input $J$ restricted to the interval $[s, t]$. The following Theorem holds

\begin{theorem}\label{th:convergence doeblin-harris}
    Suppose that $J(t)$ satisfies \eqref{eq:propriété J}. Then, there exist constants $\alpha > 0$ and $M > 0$, depending only on $J_\infty$ such that for all $t \geq s \geq 0$ and for all $f \in L^1_v((-\infty, V_F))$ satisfying $\int_{-\infty}^{V_F}f(v)\,dv = 0$, the following estimate holds
    \[
        \|\Gamma(t, s) f\|_{L^1_v} \leq M e^{-\alpha(t - s)} \|f\|_{L^1_v}
    \]
with
\[
\|f\|_{L^1_v} = \int_{-\infty}^{V_F}\big(1+v^2\big)|f(v)|dv.
\]
\end{theorem}
The proof of Theorem \ref{th:convergence doeblin-harris} relies on the Doeblin-Harris method (see for instance \cite{arXiv:2601.19282}) and is structured into three main steps.  First, we show that the linear drift $v \mapsto -(v + J(t))$ ensures that, if the initial data $u_0 \in \mathcal{M}^+_v((-\infty, V_F))$ has its first two moments bounded by a constant $C > 0$, the probability mass concentrates within a compact set $K \subset (-\infty, V_F)$ after a sufficiently long time $T$, where both $K$ and $T$ depend only on $C$.  Second, by employing a change of variables, we leverage the regularizing properties of the diffusion term to derive a minorization condition. This provides, for all $t \geq T$, a uniform pointwise lower bound for any solution $u(t) = \Gamma(t, s)u(s)$  on a compact subset of $(-\infty, V_F)$.  Finally, to conclude the Doeblin-Harris argument, we distinguish two complementary steps. First, by applying the minorization condition under the constraint that the first two moments of the initial data are bounded, we obtain a contraction for the evolution operator $\Gamma(t,s)$ with bounded moments. 
Second, to extend the contraction to a norm controlling the first two moments for general initial data, we introduce the second moment as a Lyapunov function
\[
\widetilde{M}(t) = \int_{-\infty}^{V_F} v^2 |f(t,v)| \, dv,
\]
where $\int f(t,v)dv=0$, and derive to obtain
\[
\widetilde{M}'(t) \leq -\alpha \widetilde{M}(t) + C \|f(t)\|_{L^1}.
\]
This ensures the exponential decay of the second moment, which allows us to conclude the contraction for general initial data in the corresponding weighted norm, thereby completing the proof.

\subsection{Localization of Mass on a compact of $(-\infty,V_F[$}

 For a given initial datum $u^0 \in \mathcal{M}^+_v((-\infty, V_F))$ and an initial time $s \geq 0$, we consider the unique solution $(u(t, \cdot))_{t \geq s}$ starting from $u^0$. 

\begin{proposition}\label{prop:Localisation of Mass}[Localization of Mass]\\
For any $C>0$, there exist constants \( \eta, A,\varepsilon>0\) and $t_C\geq 1$ only depending on $J_\infty$ and $C$ such that for every \( u^0 \in \mathcal{M}^+_v((-\infty, V_F)) \) with \( M_s(0) \leq C \), the following estimate holds
\[
\forall t \geq t_C, \, s \geq 0, \quad \int_{-A}^{V_F-\varepsilon} \Gamma(t+s, s) u^0(v) \, dv \geq \eta.
\]
\end{proposition}
\begin{proof}

First, we observe that it suffices to treat the case $s=0$, provided that every estimate depends on $t \mapsto J(t)$ only through its upper bound $J_\infty$.
Let $u(t) = \Gamma(t,0)u^0$ and consider 
\begin{equation}
    M(t) = \int_{-\infty}^{V_F} (v -V_R)^2 u(t,v) \, dv \quad \text{and} \quad X(t) = \int_{-\infty}^{V_F}(V_F-v)u(t,v) \, dv.
\end{equation}
Then, assuming that $M(0) \leq C$, it is sufficient to show that there exist $t_C > 0$ and $D', c > 0$ such that for all $u^0 \in \mathcal{M}^+_v((-\infty, V_F))$
\begin{equation}\label{eq:estimation_M_X}
    \forall t \geq t_C, \quad M(t) \leq D' \quad \text{and} \quad \forall t \geq \log(2), \quad X(t) \geq c.
\end{equation}
Indeed, if \eqref{eq:estimation_M_X} holds, then for any $t \geq \max(\log(2), t_C)$ and any $\varepsilon > 0, A > 0$, we have
\begin{align*}
    c \leq \int_{-\infty}^{V_F}(V_F-v)u(t,v) \, dv &= \int_{-\infty}^{-A}(V_F-v)u(t,v) \, dv + \int_{-A}^{V_F-\varepsilon}(V_F-v)u(t,v) \, dv \\
    &\quad + \int_{V_F-\varepsilon}^{V_F} \underbrace{(V_F-v)}_{\leq \varepsilon} u(t,v) \, dv \\
    &\leq \int_{-\infty}^{-A} \underbrace{\frac{V_F-v}{(v-V_R)^2}}_{\leq \frac{V_F+A}{(A+V_R)^2}} (v-V_R)^2 u(t,v) \, dv + (V_F+A) \int_{-A}^{V_F-\varepsilon} u(t,v) \, dv \\
    &\quad + \varepsilon \underbrace{\int_{V_F-\varepsilon}^{V_F} u(t,v) \, dv}_{\leq 1} \\
    &\leq \frac{V_F+A}{(A+V_R)^2} M(t) + (V_F+A) \int_{-A}^{V_F-\varepsilon} u(t,v) \, dv + \varepsilon \\
    &\leq \frac{V_F+A}{(A+V_R)^2} D' + (V_F+A) \int_{-A}^{V_F-\varepsilon} u(t,v) \, dv + \varepsilon.
\end{align*}
Therefore, we obtain
\[
\int_{-A}^{V_F-\varepsilon} u(t,v) \, dv \geq \frac{1}{V_F+A} \left( c - \varepsilon - \frac{V_F+A}{(A+V_R)^2} D' \right).
\]
Choosing $A$ large enough such that $\frac{V_F+A}{(A+V_R)^2} D' \leq \frac{c}{4}$ and setting $\varepsilon = \min(\frac{c}{4}, \frac{V_F-V_R}{2})$, we get
\[
\int_{-A}^{V_F-\varepsilon} u(t,v) \, dv \geq \frac{1}{V_F+A} \cdot \frac{c}{4} := \eta > 0,
\]
which proves Proposition \ref{prop:Localisation of Mass}.\\
Now, we establish \eqref{eq:estimation_M_X}. Differentiating $M(t)$, we obtain
\[
\forall t \geq 0, \quad M'(t) = \int_{-\infty}^{V_F} (v-V_R)^2 \Big( \partial_v \big( (v+J(t))u \big) + \partial_v^2 u + \delta_{V_R} N(t) \Big) \, dv.
\]
By integration by parts and using the boundary conditions $\partial_v u(t, V_F) = -N(t)$ and $u(t, V_F) = 0$, we have
\begin{align*}
M'(t) &= -2 \int_{-\infty}^{V_F} (v-V_R)(v+J(t)) u(t,v) \, dv - \underbrace{(V_F-V_R)^2 N(t)}_{\geq 0} + 2 \underbrace{\int_{-\infty}^{V_F} u(t,v) \, dv}_{=1} \\
    &\leq -2M(t) + 2(J(t)+V_R) \int_{-\infty}^{V_F} (v-V_R) u(t,v) \, dv + 2.
\end{align*}
Using the inequality $2(J(t)+V_R)(v-V_R) \leq (v-V_R)^2 + (J(t)+V_R)^2 \leq (v-V_R)^2 + 2(J_\infty^2 + V_R^2)$, we obtain
\[
M'(t) \leq -M(t) + 2(1 + J_\infty^2 + V_R^2).
\]
We set
\begin{equation}\label{eq:definition_D_tC}
    D := 2(1 + J_\infty^2 + V_R^2) \quad \text{and} \quad t_C := \max\left(1, \log\left(\frac{C}{D}\right)\right).
\end{equation}
By Gronwall's Lemma, we get
\begin{equation}\label{eq:estimation_M_t_final}
    \forall t \geq 0, \quad M(t) \leq M(0) e^{-t} + D \leq C e^{-t} + D.
\end{equation}
For $t \geq t_C$, we obtain
\[
M(t) \leq 2D := D'.
\]
Now consider $X(t)$. By differentiating the integral, we find
\begin{align*}
    X'(t) &= \int_{-\infty}^{V_F} (V_F-v) \Big( \partial_v \big( (v+J(t))u(t,v) \big) + \partial_v^2 u(t,v) \Big) \, dv + (V_F-V_R) N(t) \\
    &= \int_{-\infty}^{V_F} (v+J(t)) u(t,v) \, dv +  (V_F-V_R) N(t) \\
    &= V_F + \underbrace{J(t)}_{\geq 0} - X(t) +\underbrace{(V_F-V_R) N(t)}_{\geq 0} \\
    &\geq V_F - X(t).
\end{align*}
Integrating the above inequality, we have
\[
\forall t \geq 0, \quad X(t) \geq X(0) e^{-t} + (1 - e^{-t}) V_F.
\]
Since $X(0) \geq 0$ and $V_F > 0$ according to \eqref{eq: condition V_R, V_F}, for $t \geq t_0 = \log(2)$, we conclude
\[
X(t) \geq \frac{V_F}{2} := c>0.
\]

\end{proof}

\subsection{Uniform pointwise lower bound of the solution}

\begin{proposition}\label{th:Lower Bound} [Lower Bound]
There exist a constant $\alpha \in (0,1)$ and a time $t'_C \geq t_C$ such that for all $s \geq 0$, there exists a probability measure $\nu_s\in L^1((-\infty,V_F))$ such that for any $u^0 \in \mathcal{M}^+_v((-\infty, V_F))$ satisfying $M(0) \leq C$, the following estimate holds
\[
\Gamma(t'_C + s, s) u^0 \geq \alpha \nu_s.
\]
\end{proposition}

\begin{proof}
As in Proposition~\ref{prop:Localisation of Mass}, we observe that it is enough to establish  the result only for $s=0$. Specifically, we show there exist $\alpha \in (0,1)$, a probability measure $\nu_0\in L^1((-\infty,V_F))$, and a time $t_C'$ such that for any $u^0 \in \mathcal{M}_v^+((-\infty,V_F))$ with $M(0) \leq C$, 
\[
\Gamma(t_C',0)u^0 \geq \alpha \nu_0,
\]
where the constant $\alpha$ and the time $t_C'$ depend on the coupling term $J$ only through its upper bound $J_\infty$ and the constant $C$.  This lower bound extends  then  to the operator $\Gamma(t'_C+s,s)$ for any $s \geq 0$.
\newline
To derive this lower bound, we first perform a change of variables that transforms the system into a Stefan-type problem. We then construct a subsolution, which enables us to define $\nu_0$ by comparison with a heat equation with Dirichlet boundary conditions, using the mass localization result from Proposition \ref{prop:Localisation of Mass}.
\newline
\underline{Change of variables.} We introduce the following change of variables:
\begin{equation}\label{eq: changement variables}
    \left\{
\begin{aligned}
& \tau(t) = \frac{e^{2t} - 1}{2}, \\
& x(v,t) = e^t v + \int_0^t J(z) e^z\, dz.
\end{aligned}
\right.
\end{equation}
We then define the new function \( q(\tau, x) \) by
\begin{equation}\label{eq: def q}
    q(\tau, x) = \frac{1}{\sqrt{2\tau + 1}} \, 
    u\Bigg(\frac{1}{2} \log(2\tau + 1), \frac{x - \int_0^\tau \tilde{J}(z) \, dz}{\sqrt{2\tau + 1}}\Bigg),
\end{equation}
where
\[
\tilde{J}(z) = \frac{J\Big(\frac{1}{2} \log(2z + 1)\Big)}{\sqrt{2z + 1}}.
\]
The function \( q \) satisfies a Stefan-type problem with a source term:
\begin{equation}\label{eq: equation sur q}
    \left\{
\begin{aligned}
& \partial_\tau q - \partial_x^2 q = M_q(\tau) \, \delta_{x = s_1(\tau)}, \quad \tau \geq 0, \; x \leq s(\tau), \\
& M_q(\tau) = -\partial_x q(\tau, s(\tau)), \\
& q(\tau, s(\tau)) = 0, \\
& s(\tau) = V_F \sqrt{2\tau + 1} + \int_0^\tau \tilde{J}(z) \, dz, \\
& s_1(\tau) = s(\tau) - (V_F - V_R) \sqrt{2\tau + 1}.
\end{aligned}
\right.
\end{equation}
The initial condition is given by
\[
\forall x \in (-\infty, V_F], \quad q(0, x) = u^0(x).
\]
\newline
\underline{Construction of a sub-solution.} Assume that \( u^0 \in \mathcal{M}^+_v((-\infty, V_F)) \) and that \( M_0(0) \leq C \). By Proposition \ref{prop:Localisation of Mass}, we may assume, without loss of generality, that there exists \(\varepsilon > 0\), depending only on the constant \(C\), such that \( u^0 \) satisfies
\[
\int_{-A}^{V_F - \varepsilon} u^0(v) \, dv \geq \eta.
\]
In the new variables, we note that for all $\tau \geq 0$, 
\[
s_1(\tau) = s(\tau) - (V_F - V_R) \sqrt{2\tau + 1} \leq s(\tau) \quad \text{and} \quad s(\tau) \geq V_F.
\] 
Let $\phi \in \mathcal{C}^\infty((-\infty, V_F])$ be a cut-off function such that $0 \leq \phi \leq 1$, with 
\[
\phi \equiv 1 \text{ on } (-A, V_F-\varepsilon), \quad 
\phi(x) = 0 \text{ for } x \leq -A-1 \text{ and } x \in \left(V_F-\frac{\varepsilon}{2}, V_F\right).
\]

Let $\rho$ denote the solution of the heat equation with Dirichlet boundary conditions:
\[
\left\{
\begin{aligned}
& \partial_\tau \rho - \partial_x^2 \rho = 0, && \tau \geq 0, \, x \in (-\infty, V_F), \\
& \rho(\tau, V_F) = 0, \\
& \rho(0, x) = \phi(x) u^0(x), && x \in (-\infty, V_F).
\end{aligned}
\right.
\]
Then $\rho$ is a subsolution of $q$, where $q$ is the solution of \eqref{eq: equation sur q} with initial data $u^0$.  
By spectral theory of the heat equation, there exist constants $\tau_1 > 0$, $c > 0$, and $\delta > 0$ such that
\[
\forall \tau \geq \tau_1, \, \forall x \in (-A, V_F-\varepsilon), \quad
\rho(\tau, x) \geq c \, e^{-\delta \tau} \int_{-A}^{V_F-\varepsilon} \rho(0, y) \, dy
= c e^{-\delta \tau} \int_{-A}^{V_F-\varepsilon} u^0(y) \, dy \geq c \eta e^{-\delta \tau}.
\]
Hence, we obtain
\[
\forall \tau \geq \tau_1, \, \forall x \in (-A, V_F-\varepsilon), \quad q(\tau, x) \geq c \eta e^{-\delta \tau}.
\]
Reverting to the original variables, let $t_1 = \frac{1}{2}\log(2\tau_1 + 1)$. We define the bounds at time $t_1$ by
\[
A' = e^{-t_1}\Big(-A - \int_0^{t_1} J(z)e^z \, dz\Big), \quad 
B' = e^{-t_1}\Big(V_F - \varepsilon - \int_0^{t_1} J(z)e^z \, dz\Big).
\]
Consequently, we have
\[
\forall x \in (A', B'), \quad u(t_1, x) \geq \frac{c \eta}{\sqrt{2\tau_1+1}} \, e^{-\delta \tau_1}.
\]
By setting $t'_C = t_C + t_1$ and defining
\[
\alpha = \min\left(\frac{3}{4}, \frac{(B'-A')c \eta}{\sqrt{2\tau_1+1}} \, e^{-\delta \tau_1}\right) \in (0,1),
\]
we obtain the Doeblin condition
\[
\Gamma(t_C',0) u^0 \geq \alpha \, \nu_0,
\]
where 
\[
\nu_0 = \frac{1}{B'-A'} \mathbf{1}_{(A', B')}\in L^1((-\infty,V_F))
\]
is the minorizing probability measure. Observe that $B'-A' = e^{-t_1}(V_F + A)$; hence, $\alpha$ depends solely on $A$ and, by extension, on $J_\infty$. For the general case where $s \geq 0$, the corresponding measure $\nu_s$ is given by
\[
\nu_s = \frac{1}{B'_s - A'_s} \mathbf{1}_{(A'_s, B'_s)}\in L^1((-\infty,V_F)),
\]
with the time-dependent boundaries defined as
\[
A'_s = e^{-t_1}\left(-A - \int_0^{t_1} J(z+s)e^{z} \, dz\right) \quad \text{and} \quad B'_s = e^{-t_1}\left(V_F -\varepsilon- \int_0^{t_1} J(z+s)e^{z} \, dz\right).
\]
Since the constant $\alpha$ depends only on the uniform bound $J_\infty$, it remains valid for the lower bound at any time $s$, yielding
\[
\Gamma(t_C'+s,s)u^0 \geq \alpha \nu_s.
\]
\end{proof}

\subsection{Lyapunov function and $L^1$ contraction for general initial data}
We now state a contraction result for the evolution operator $\Gamma$, under the assumption that the first two moments of the initial condition are bounded. 
Here, the initial data is understood as the difference of two initial datas in $\mathcal{M}^+_v((-\infty, V_F))$, so that its mean is zero.  This result will be instrumental in applying the Doeblin-Harris method. 
\begin{proposition}\label{prop:contraction-Doeblin}
Let $C \geq 0$, and let $\alpha, t'_C$ be the constants from Proposition \ref{th:Lower Bound}. 
Then, for any $f \in L^1_v((-\infty,V_F])$ satisfying
\[
\int_{-\infty}^{V_F} f(v)\, dv = 0 \quad \text{and} \quad \|(v-V_R)^2 f\|_1 \leq \frac{C}{2} \|f\|_1,
\]
the following contraction estimate holds:
\[
\forall s \geq 0, \quad \|\Gamma(t'_C + s, s) f\|_1 \leq (1-\alpha) \|f\|_1.
\]
\end{proposition}

\begin{proof}
Let $f \in L^1_v((-\infty,V_F])$ satisfy the assumptions of Proposition \ref{prop:contraction-Doeblin}. 
Define 
\[
\mu_+ = \max\Big(0, \frac{2f}{\|f\|_1}\Big), \qquad 
\mu_- = \max\Big(0, \frac{-2f}{\|f\|_1}\Big),
\]
so that
\[
f = \frac{\|f\|_1}{2}(\mu_+ - \mu_-), \qquad 
\|(v-V_R)^2 \mu_\pm\|_1 \leq \Big\|(v-V_R)^2 \frac{2f}{\|f\|_1}\Big\|_1 \leq C.
\]
Since $\int_{-\infty}^{V_F} f(v)\, dv = 0$, we have $\mu_+, \mu_- \in \mathcal{M}_v^+((-\infty,V_F))$. 
Applying Proposition \ref{th:Lower Bound} then gives the lower bound
\[
\Gamma(s+t'_C,s)\mu_\pm \geq \alpha \nu_s.
\]
Now, we compute the $L^1$ norm of $\Gamma(t'_C+s,s) f$:
\begin{align*}
\|\Gamma(t'_C+s,s) f\|_1
&= \frac{\|f\|_1}{2} \,\| \Gamma(t'_C+s,s)\mu_+ - \Gamma(t'_C+s,s)\mu_- \|_1 \\
&= \frac{\|f\|_1}{2} \,\| (\Gamma(t'_C+s,s)\mu_+ - \alpha \nu_s) - (\Gamma(t'_C+s,s)\mu_- - \alpha \nu_s) \|_1 \\
&\le \frac{\|f\|_1}{2} \Big( \|\underbrace{\Gamma(t'_C+s,s)\mu_+ - \alpha \nu_s}_{\geq 0}\|_1 + \|\underbrace{\Gamma(t'_C+s,s)\mu_- - \alpha \nu_s}_{\geq 0}\|_1 \Big) \\
&= \frac{\|f\|_1}{2} \Bigg( \int_{-\infty}^{V_F} \big(\Gamma(t'_C+s,s)\mu_+ - \alpha \nu_s\big) \, dv 
       + \int_{-\infty}^{V_F} \big(\Gamma(t'_C+s,s)\mu_- - \alpha \nu_s\big) \, dv \Bigg) \\
&\le (1-\alpha) \, \|f\|_1.
\end{align*}
\end{proof}

\noindent {\em Proof of Theorem \ref{th:convergence doeblin-harris} : Lyapunov estimate and global control.}\\
In order to complement the Doeblin-type contraction argument and obtain a global control, we introduce a Lyapunov estimate that propagates the second  moment of the semigroup. This estimate will be a key ingredient to construct a weighted norm in which a global contraction holds.
\begin{proposition}\label{prop:Lyap}[Lyapunov control]
Let $f \in L^1_v((-\infty,V_F))$. Then, for all $t,s \geq 0$, the following estimate holds:
\begin{equation}\label{eq:controle Ms(f)_prop}
\|(v-V_R)^2\Gamma(t+s,s)f\|_1
\leq \|(v-V_R)^2f\|_1 e^{-t} + D \|f\|_1.
\end{equation}
\end{proposition}
\begin{proof}
We first recall that for any probability density $\mu$, the moment estimate obtained via Proposition \ref{prop:Localisation of Mass} reads
\begin{equation}\label{eq:lyap}
M_s(t) = \int_{-\infty}^{V_F} (v -V_R)^2 \Gamma(t+s, s)\mu(v) \, dv
\leq M_s(0)e^{-t}+D.
\end{equation}
Let $f \in L^1_v((-\infty,V_F))$. We decompose it into its positive and negative parts:
\[
f = \|f_+\|_1 \mu_+ - \|f_-\|_1 \mu_-,
\]
where $f_+ = \max(f,0)$, $f_- = \max(-f,0)$ and $\mu_\pm = \frac{f_\pm}{\|f_\pm\|_1}$ are probability measures.
Using the linearity of $\Gamma(t+s,s)$ and the triangle inequality, we obtain
\begin{align*}
\|(v-V_R)^2\Gamma(t+s,s)f\|_1 
&= \|(v-V_R)^2\Gamma(t+s,s)(\|f_+\|_1\mu_+ - \|f_-\|_1\mu_-)\|_1 \\
&\leq \|f_+\|_1 \|(v-V_R)^2\Gamma(t+s,s)\mu_+\|_1 
+ \|f_-\|_1 \|(v-V_R)^2\Gamma(t+s,s)\mu_-\|_1.
\end{align*}
Applying \eqref{eq:lyap} estimate to $\mu_\pm$, we get
\[
\|(v-V_R)^2\Gamma(t+s,s)\mu_\pm\|_1
\leq \|(v-V_R)^2\mu_\pm\|_1 e^{-t} + D.
\]
Therefore,
\begin{align*}
\|(v-V_R)^2\Gamma(t+s,s)f\|_1
&\leq \|f_+\|_1 \|(v-V_R)^2\mu_+\|_1 e^{-t}
+ \|f_-\|_1 \|(v-V_R)^2\mu_-\|_1 e^{-t} \\
&\quad + D(\|f_+\|_1 + \|f_-\|_1).
\end{align*}
Finally, we observe that
\[
\|f_+\|_1 \mu_+ + \|f_-\|_1 \mu_- = |f|,
\qquad
\|f_+\|_1 + \|f_-\|_1 = \|f\|_1,
\]
which implies
\[
\|f_+\|_1 \|(v-V_R)^2\mu_+\|_1 
+ \|f_-\|_1 \|(v-V_R)^2\mu_-\|_1
= \|(v-V_R)^2 f\|_1.
\]
This yields the desired estimate
\[
\|(v-V_R)^2\Gamma(t+s,s)f\|_1
\leq \|(v-V_R)^2 f\|_1 e^{-t} + D \|f\|_1,
\]
and concludes the proof of Proposition \ref{prop:Lyap}.
\end{proof}

\noindent In order to combine the local Doeblin-type contraction in Proposition \ref{prop:contraction-Doeblin} with the global Lyapunov control in Proposition \ref{prop:Lyap}, we introduce the following suitable weighted norm in which the semigroup will be shown to be contractive. Let $\beta > 0$. For any $f \in L^1_v$, we define the weighted norm
\[
N_\beta(f) = \|f\|_1 + \beta \|(v-V_R)^2 f\|_1.
\]
Let us observe that the norm $N_\beta$ is equivalent to the standard weighted $L^1_v$ norm, for all $\beta >0$.
We now fix $s \geq 0$ and aim to prove that there exist $\beta > 0$ and $C > 0$ such that the operator
\[
f \mapsto \Gamma(t_C' + s, s)f
\]
is a contraction in the norm $N_\beta$. The argument relies on a decomposition according to the distribution of the mass. Roughly speaking, we distinguish between configurations where a non-negligible fraction of the mass lies in a compact region allowing us to apply the Doeblin condition and those where the mass is spread in the tails, in which case the Lyapunov estimate yields a decay of the weighted moment. We therefore split the proof into two complementary cases.

\noindent{\em First case:} $\|(v-V_R)^2 f\|_1 \leq \frac{C}{2} \|f\|_1$.
Using Proposition~\ref{prop:contraction-Doeblin} together with Proposition \ref{prop:Lyap}, we obtain
\[
\|\Gamma(t_C' + s, s)f\|_1 \leq (1 - \alpha)\|f\|_1,
\]
and
\[
\|(v-V_R)^2 \Gamma(t_C' + s, s)f\|_1 
\leq e^{-t'_C} \|(v-V_R)^2 f\|_1 + D \|f\|_1.
\]
Combining these two inequalities yields
\begin{align*}
N_\beta\big(\Gamma(t_C' + s, s)f\big)
&\leq (1 - \alpha + \beta D)\|f\|_1 
+ \beta e^{-t'_C} \|(v-V_R)^2 f\|_1 \\
&= (1 - \alpha + \beta D)
\left(
\|f\|_1 
+ \frac{\beta e^{-t'_C}}{1 - \alpha + \beta D} \|(v-V_R)^2 f\|_1
\right).
\end{align*}
Choosing $\beta \in \left(0, \frac{\alpha}{2D}\right)$ ensures that
\[
1 - \alpha + \beta D \leq 1 - \frac{\alpha}{2}.
\]
Moreover, since $t_C' \geq t_C \geq 1$, we have
\[
\frac{e^{-t'_C}}{1 - \alpha + \beta D}
\leq \frac{e^{-t'_C}}{1 - \frac{\alpha}{2}}
\leq 1.
\]
Therefore, we obtain the contraction estimate
\[
N_\beta\big(\Gamma(t_C' + s, s)f\big)
\leq \left(1 - \frac{\alpha}{2}\right) N_\beta(f)
=: \gamma^{(1)} N_\beta(f).
\]

\noindent{\em Second case:} $\|(v-V_R)^2 f\|_1 \geq \frac{C}{2} \|f\|_1$.
Let $\varepsilon > 0$. Using again Proposition \ref{prop:Lyap}, we have
\begin{align*}
N_\beta\big(\Gamma(t_C' + s, s)f\big)
&= \|\Gamma(t_C' + s, s)f\|_1 + \beta \|(v-V_R)^2 \Gamma(t_C' + s, s)f\|_1 \\
&\leq (1 + \beta D)\|f\|_1 + \beta e^{-t_C'} \|(v-V_R)^2 f\|_1 \\
&\leq (1 - \varepsilon + \beta D)\|f\|_1 + \varepsilon \|f\|_1 + \beta e^{-t_C'} \|(v-V_R)^2 f\|_1.
\end{align*}
Using the assumption of this case, $\|f\|_1 \leq \frac{2}{C} \|(v-V_R)^2 f\|_1$, we obtain
\[
N_\beta\big(\Gamma(t_C' + s, s)f\big)
\leq (1 - \varepsilon + \beta D)\|f\|_1
+ \left(\frac{2\varepsilon}{C} + \beta e^{-t_C'}\right)\|(v-V_R)^2 f\|_1.
\]
We now choose $\varepsilon = 2\beta D$ and assume that $\beta$ is small enough so that $\gamma^{(2)} := 1 - \beta D \in (0,1)$. Then
\[
N_\beta\big(\Gamma(t_C' + s, s)f\big)
\leq \gamma^{(2)}\left[\|f\|_1 + \beta\left(\frac{4D}{C} + e^{-t_C'}\right)\|(v-V_R)^2 f\|_1\right].
\]
Choosing $C = 8D$ and using that $t_C' \geq \log(2)$, we get
\[
\frac{4D}{C} + e^{-t_C'} = \frac{1}{2} + e^{-t_C'} \leq 1.
\]
Therefore,
\[
N_\beta\big(\Gamma(t_C' + s, s)f\big) \leq \gamma^{(2)} N_\beta(f).
\]
We can now conclude the proof of Theorem \ref{th:convergence doeblin-harris}.
Setting $\gamma = \max(\gamma^{(1)}, \gamma^{(2)}) \in (0,1)$, we have shown that for any $f \in L^1_v$ with zero mean and any $s \geq 0$,
\[
N_\beta\big(\Gamma(t_C' + s, s)f\big) \leq \gamma N_\beta(f).
\]
By induction, for any $k \in \mathbb{N}$ and $s \geq 0$,
\[
N_\beta\big(\Gamma(k t_C' + s, s)f\big) \leq \gamma^k N_\beta(f).
\]
For an arbitrary $t \geq 0$, we write $t = n t_C' + r$ with $n = \lfloor t/t_C' \rfloor$ and $r \in [0, t_C')$. Using again \eqref{eq:controle Ms(f)_prop}, we obtain for short times
\[
\forall r \geq 0, \qquad N_\beta\big(\Gamma(r+s, s)f\big) \leq (1 + \beta D) N_\beta(f).
\]
Using the norm equivalence, we conclude
\begin{align*}
\|\Gamma(t+s, s)f\|_{L^1_v}
&\leq (m')^{-1} N_\beta(\Gamma(t+s, s)f) \\
&\leq (m')^{-1} \gamma^n N_\beta(\Gamma(r+s, s)f) \\
&\leq (m')^{-1} \gamma^{t/t_C' - 1} (1 + \beta D) M' \|f\|_{L^1_v} \\
&\leq M e^{-\lambda t} \|f\|_{L^1_v},
\end{align*}
where $\lambda = -\frac{\log \gamma}{t_C'} > 0$ and $M = (m' \gamma)^{-1} M' (1 + \beta D)$.
This completes the proof of Theorem \ref{th:convergence doeblin-harris}. \hfill $\square$

\section{Exponential convergence towards the stationary state under perturbations of the input current}\label{sec:lienu}
In this section, we assume that the input current $J(t)$ converges sufficiently fast towards a constant value $J' \geq 0$, in the sense that
\begin{equation}\label{eq: conv J}
    \exists \beta > 0, \quad \int_0^{\infty} e^{\beta t} |J(t) - J'| \, dt < \infty.
\end{equation}
\noindent 
Under this assumption, we show that the solution converges exponentially fast towards the stationary state associated with $J'$. In other words, although the dynamics is driven by a time-dependent input, the system asymptotically behaves as if the current were constant and equal to $J'$. In this sense, the system progressively ``forgets'' its past, and the long-time behavior is entirely determined by the limiting value $J'$. More precisely, the following Theorem holds : 
\begin{theorem}\label{th:convergence exp u et N}
Let $u^0 \in \mathcal{M}^+_v((-\infty,V_F))$ and define $u(t)=\Gamma(t,0)u^0$ for all $t \geq 0$. Assume that $J(t)$ satisfies \eqref{eq:propriété J} and \eqref{eq: conv J} for some $J' \geq 0$ and $\beta > 0$. Let $M, \alpha > 0$ be the constants given by Theorem~\ref{th:convergence doeblin-harris}, and denote by $\overline{u}(J')$ the stationary state defined in \eqref{eq:stat J}. Then, for all $t \geq 0$, the following estimate holds:
\begin{equation}\label{eq:inégalité exp sur u}
    \| u(t) - \overline{u}(J') \|_{L^1}
    \leq C_t \, e^{-\min(\alpha,\beta) t},
\end{equation}
where
\[
C_t = M \left( 
\| u^0 - \overline{u}(J') \|_{L^1}
+ \| \partial_v \overline{u}(J') \|_{L^1}
\int_0^{t} |J(s) - J'| e^{\beta s} \, ds
\right).
\]
Moreover, there exist constants $c(J_\infty), K(J') > 0$ such that for any $\gamma \in (0, \min(2,\alpha,\beta))$ and any $T \in \mathbb{R}^+ \cup \{+\infty\}$,
\begin{equation}\label{eq:borne int N_u}
\begin{aligned}
(V_F - V_R)^2 \int_0^{T} |N_u(t) - \Phi_1(J')| \, e^{\gamma t} \, dt 
&\leq \| (v - V_R)^2 (u^0 - \overline{u}(J')) \|_{L^1} \\
&\quad + c(J_\infty)\, \frac{C_T}{\min(\alpha,\beta)-\gamma} \\
&\quad + K(J') \int_0^{T} |J(t) - J'| \, e^{\gamma t} \, dt.
\end{aligned}
\end{equation}
\end{theorem}
\begin{proof}
The main idea of the proof of Theorem~\ref{th:convergence exp u et N} is to compare the solution $u(t)$ associated with the time-dependent input $J(t)$ to the stationary solution $\overline{u}(J')$ corresponding to the limiting input $J'$. This allows us to express $u(t)$ as a perturbation around a well-understood reference state, with a remaining source term that depends only on the stationary solution $\overline{u}(J')$ and its derivative, as well as on the deviation of $J(t)$ from $J'$. In this way, the evolution of the difference $u(t)-\overline{u}(J')$ can be controlled using the linear exponential convergence established in Theorem~\ref{th:convergence doeblin-harris}, the assumption~\eqref{eq:propriété J} on $J(t)$, and the smoothness properties of the stationary solution.  To  this, we define the time-dependent linear operator $\mathcal{L}(t)$ by
\[
\forall t \geq 0, \quad \mathcal{L}(t) f = \partial_v\big((J(t)+v) f\big) + \partial_v^2 f + N_f \, \delta_{V_R},
\]
where $N_f = -f'(V_F)$ represents the flux at the boundary for a function $f$. Then, $u$ solves the evolution problem
\begin{equation}\label{eq:evolution u}
\begin{cases}
\displaystyle \frac{d u}{d t} = \mathcal{L}(t) u, \quad t \geq 0, \\
u(0) = u^0.
\end{cases}
\end{equation}
   
  \noindent{\bf Step 1. Proof of inequality \eqref{eq:inégalité exp sur u}.}  Recall that $u$ satisfies
\[
\frac{d u}{d t} = \mathcal{L}(t) u, \quad \text{with } \mathcal{L}(t) f = \partial_v\big((J(t)+v) f\big) + \partial_v^2 f + N_f \delta_{V_R}, \quad N_f = -f'(V_F).
\]
Define the difference
\[
w(t) = u(t) - \overline{u}(J'),
\]
where $\overline{u}(J')$ is the stationary solution associated with the limiting input $J'$. Then $w$ satisfies
\begin{equation}\label{eq: equation w final}
\frac{d w}{d t} = \mathcal{L}(t) w + h(t), \quad 
h(t) = (\mathcal{L}(t) - \mathcal{L}_{J'}) \overline{u}(J') = (J(t) - J') \partial_v \overline{u}(J').
\end{equation}
By Duhamel's formula, we have
\[
w(t) = \Gamma(t,0) w(0) + \int_0^t \Gamma(t,s) h(s) \, ds,
\]
where $\Gamma(t,s)$ is the evolution operator of the linear problem. Since
\[
\int_{-\infty}^{V_F} h(s,v)\, dv = (J(s) - J') \int_{-\infty}^{V_F} \partial_v \overline{u}(J',v)\, dv = 0,
\]
and $\int_{-\infty}^{V_F} w(0,v)\, dv = 0$, we can apply Theorem~\ref{th:convergence doeblin-harris} to obtain
\begin{align*}
\|w(t)\|_{L^1} &\leq \|\Gamma(t,0) w(0)\|_{L^1} + \int_0^t \|\Gamma(t,s) h(s)\|_{L^1} \, ds \\
&\leq M e^{-\alpha t} \|w(0)\|_{L^1} + M \|\partial_v \overline{u}(J')\|_{L^1} \int_0^t |J(s) - J'| e^{-\alpha (t-s)} \, ds.
\end{align*}
Distinguishing cases depending on the relative size of $\alpha$ and $\beta$, and setting $\gamma = \min(\alpha, \beta)$, we obtain
\begin{equation}\label{eq: w final bound}
\|w(t)\|_{L^1} \leq C_t e^{-\gamma t}, \quad 
C_t = M \left( \|w(0)\|_{L^1} + \|\partial_v \overline{u}(J')\|_{L^1} \int_0^t |J(s) - J'| e^{\beta s} \, ds \right),
\end{equation}
which proves \eqref{eq:inégalité exp sur u}.

\noindent{\bf Step 2. Proof of inequality \eqref{eq:borne int N_u}.}  Set
\[
N_w(t) = N_u(t) - \Phi_1(J') = -\partial_v w(t,V_F), \quad M(t) = \|(v-V_R)^2 w(t)\|_{L^1}.
\]
Multiplying the differential inequality for $|w|$ by $(v-V_R)^2$ and integrating over $(-\infty,V_F)$, we obtain after integration by parts and using boundary conditions:
\[
M'(t) + 2 M(t) + (V_F-V_R)^2 |N_w(t)| \leq c(J_\infty) \|w(t)\|_{L^1} + K(J') |J(t) - J'|,
\]
with
\[
c(J_\infty) = 2 + (V_R+J_\infty)(1+V_R), \quad 
K(J') = \|(v-V_R)^2 \partial_v \overline{u}(J')\|_{L^1}.
\]
Multiplying by $e^{\gamma t}$ for $\gamma \in (0, \min(2, \alpha, \beta))$ and integrating over $[0,T]$, we get
\begin{align*}
(V_F-V_R)^2 \int_0^T |N_w(t)| e^{\gamma t} \, dt
&\leq \| (v-V_R)^2 (u^0 - \overline{u}(J')) \|_{L^1} \\
&\quad + c(J_\infty) \int_0^T \|w(t)\|_{L^1} e^{\gamma t} \, dt \\
&\quad + K(J') \int_0^T |J(t) - J'| e^{\gamma t} \, dt.
\end{align*}
Finally, applying \eqref{eq: w final bound} to bound $\|w(t)\|_{L^1} \leq C_T e^{-\gamma t}$, we obtain for any $T>0$:
\begin{align*}
(V_F-V_R)^2 \int_0^T |N_w(t)| e^{\gamma t} \, dt 
&\leq \| (v-V_R)^2 (u^0 - \overline{u}(J')) \|_{L^1} \\
&\quad + c(J_\infty) \frac{C_T}{\min(\alpha, \beta) - \gamma} \\
&\quad + K(J') \int_0^T |J(t) - J'| e^{\gamma t} \, dt,
\end{align*}
which completes the proof of Theorem~\ref{th:convergence exp u et N}.
\end{proof}

\section{$L^\infty$ bounds on the activity $N(t)$}\label{sec:bornesunif} 

In the context of the study of equation~\eqref{eq:FP delay}, the input $J$ in Theorem~\ref{th:convergence exp u et N} will be replaced by the delayed neuron firing rate $N_u^d(\cdotp-d)$. To apply Theorem~\ref{th:convergence exp u et N} in this context, we must therefore verify that the equation with $N_u^d$ satisfies the hypotheses of Theorem~\ref{th:convergence exp u et N}, and in particular that $N_u^d$ remains bounded in the sense required by the following theorem:
\begin{theorem}\label{th:bound on N_u(t)}
Let  $Q$ be the function defined by 
\begin{equation}\label{eq:Q}
     Q(v) =
 \begin{cases}
   e^{- \frac{v^{2}}{2}}
   \displaystyle \int_{\max(v,\, V_R)}^{V_F} 
      e^{\frac{z^{2}}{2}} \, dz,
      & \text{for } v \in (0, V_F), \\[10pt]
   \displaystyle \int_{V_R}^{V_F} 
      e^{\frac{z^{2}}{2}} \, dz,
      & \text{for } v \le 0.
 \end{cases}
\end{equation}
Let $u(t) = \Gamma(t,0)u^0$ be the solution of \eqref{eq:FP linéaire instat}, with $J(t)$ satisfying \eqref{eq:propriété J}. Assume also that $\frac{u^0}{Q}\in L^\infty((-\infty,V_F])$. Then, there exists $\beta_1>0$ depending only on $J_\infty$ and $\|\frac{u^0}{Q}\|_\infty$ such that 
\[
\forall t \geq 0, \quad u(t)\leq \beta_1 Q,
\]
and hence,
\[
\sup_{t \geq 0} N_u(t)\leq \beta_1 < \infty.
\]
\end{theorem}

To prove this result, we follow the arguments of \cite{CPSS} by constructing suitable supersolutions of \eqref{eq:FP linéaire instat} and applying a parabolic regularization argument. The construction of these supersolutions requires that condition~\eqref{eq: condition V_R, V_F} holds. Before entering the core of the proof of Theorem \ref{th:bound on N_u(t)}, we first introduce the notion of supersolution and a main result in comparaison principle. 
\begin{definition}
Let $V_0 \in [-\infty, V_F)$. We say that $\overline{u} \in \mathcal{D}'([0,T]\times [V_0,V_F])$ is a \emph{supersolution} to \eqref{eq:FP linéaire instat} on $[0,T]\times [V_0,V_F]$ if 
\[
\forall (t,v)\in [0,T]\times [V_0,V_F],\qquad  
\frac{\partial \overline{u}}{\partial t} - \frac{\partial}{\partial v}\big((J(t)+v)\overline{u}\big)-\frac{\partial^2 \overline{u}}{\partial v^2} \geq \overline{N_u}(t)\delta_{V_R}(v)
\]
in the distributional sense, with $\overline{N_u}(t) = -\partial_v \overline{u}(V_F,t)$.
\end{definition}
\begin{proposition}[Comparison principle]\label{prop:carac-super-sol}
Let $u(t)$ be a solution of \eqref{eq:FP linéaire instat} and $\overline{u}$ a supersolution on $[0,T]\times [V_0,V_F]$ such that $\overline{u}(V_0,t) \ge u(V_0,t)$ for $t \in [0,T]$ and $\overline{u}(0,v) \ge u^0(v)$ for $v \in (V_0, V_F]$. Then 
\[
\forall (t,v) \in [0,T]\times (V_0,V_F], \quad \overline{u}(t,v) \ge u(t,v).
\]
Moreover, if $\overline{u}(0,\cdot) - u^0$ is not identically zero, then
\[
\forall (t,v) \in [0,T]\times (V_0,V_F), \quad \overline{u}(t,v) > u(t,v).
\]
\end{proposition}
\noindent
\noindent
We now introduce two explicit supersolutions. The first one is defined on the entire domain, so we avoid comparison issues at the boundary. Using it, comparison gives at best an upper bound on the neuronal flux that grows exponentially in time, which does not provide a global $L^\infty$ bound for the solution. The second supersolution is time-independent and potentially provides a global $L^\infty$ bound,  but it is only a supersolution on $(0,V_F)$, so to use it effectively, we must first control the solution $u(t)$ at the boundary to compare the two solutions. More precisely, the following Proposition holds
\begin{proposition}\label{prop:exemple-sur-sol}
Let 
\[
P(t,v)=
\begin{cases}
e^{t}, & v \le V_R, \\[2mm]
e^t\dfrac{V_F - v}{V_F - V_R}, & v \in (V_R, V_F),
\end{cases}
\]
Then $P$ is an upper solution on $[0,+\infty) \times (-\infty, V_F]$ to \eqref{eq:FP linéaire instat}, and $(t,v)\mapsto Q(v)$ from \eqref{eq:Q} is an upper solution on $[0,+\infty) \times [0, V_F]$ to \eqref{eq:FP linéaire instat}.
\end{proposition}

\begin{proof}
For $P$, we have $-\partial_v^2 P = e^t \delta_{V_R}$ and $\partial_v P(t, V_F) = e^t$. Hence,
\begin{align*}
\partial_t P - \partial_v\big((J(t)+v)P\big) - \partial_v^2 P 
&= \partial_t P - (v+J(t)) \partial_v P - \partial_v^2 P \\
&= e^t - (v+J(t)) \partial_v P + e^t \delta_{V_R} \ge 0,
\end{align*}
so $P$ is indeed an upper solution on $[0,+\infty) \times (-\infty, V_F]$. For $Q$, note that in the distributional sense,
\[
-\frac{d}{dv} (v Q(v)) - Q''(v) = \delta_{V_R}(v), \quad v \le V_F.
\]
Thus,
\[
-\frac{d}{dv} ((v+J(t)) Q(v)) - Q''(v) = -J(t) Q'(v) + \delta_{V_R}(v).
\]
Now, for $v \in [0, V_F]$,
\[
Q'(v) = - \mathbf{1}_{[V_R, V_F]} - v Q(v) < 0,
\]
so $Q$ is non-increasing. Therefore,
\[
-\frac{d}{dv} ((v+J(t)) Q(v)) - Q''(v) \ge \delta_{V_R}(v),
\]
and $Q$ is an upper solution on $[0,+\infty) \times [0,V_F]$.
\end{proof}
Let us now introduce a classical result on the smoothing effect for the heat equation, which will allow us to combine the advantages of both supersolutions constructed in Proposition~\ref{prop:exemple-sur-sol}. 
\begin{proposition}\label{prop: reg L^inf eq chaleur}
    There exists a constant $K>0$ such that if $f$ is a solution of the Heat Equation on $D_0 = [t_0-r^2,t_0]\times[x_0-r,x_0+r]$ for a given $t_0,x_0\in \mathbb{R},r>0$ , which is 
    $$\partial_t f-\partial_x^2f = 0\quad \text{on}\quad D_0.$$
    Then,
    \begin{equation}\label{eq:regularite elliptique}
        |f(t_0,x_0)|\leq \frac{K}{r^3}\|f\|_{1, D_0}.
    \end{equation}

\end{proposition}
\begin{proof}
    This result comes from \cite{zbMATH01061253}, Theorem 7.36, as a general form, for parabolic equation, called $L^p$ estimate for local maximum.
\end{proof}

\begin{proposition}\label{prop:borne u elliptique}
    Let $u$ be the solution to \eqref{eq:FP linéaire instat} with $u^0\in \mathcal{M}_v^+((-\infty,V_F))$. Then, for all $t_0 > 0$, there exists a constant $C(t_0, J_\infty)$ such that:
    \begin{equation}\label{eq: estimation u elliptique}
        \forall t \geq t_0, \forall v \leq 0, \quad |u(t,v)| \leq C(t_0, J_\infty).
    \end{equation}
\end{proposition}

\begin{proof}
    Let $t_0 > 0$. Fix $t_1 \geq t_0$ and $v_1 \in (-\infty, 0]$.
    Consider again, the change of variables $(\tau,x)\rightarrow(t,v)$ from \eqref{eq: changement variables} and $q(\tau,x)$ defined from \eqref{eq: def q}.

    Then, $q$ is a solution to the following diffusion problem with free boundaries on the domain $\mathcal{D} = \{(x,\tau) \in \mathbb{R} \times \mathbb{R}^+ \mid x \leq s(\tau)\}$ that we stated in \eqref{eq: equation sur q}.
    In particular, on the set $\Omega = \mathcal{D} \setminus \{(s_1(\tau), \tau) \mid \tau \geq 0\}$, $q$ satisfies
    \[
    \partial_\tau q - \partial_x^2 q = 0.
    \]
    Let $x_1, \tau_1$ be such that $t_1 = \frac{1}{2}\log(2\tau_1+1)$ and
    \[
    \frac{x_1 + \int_0^{\tau_1} \tilde{J}(z) \, dz}{\sqrt{2\tau_1+1}} = v_1.
    \]
    We also define $\tau_0$ such that $t_0 = \frac{1}{2}\log(2\tau_0+1)$.

    For all $r > 0$, we define the parabolic cylinder $\Lambda_0^r = [x_1-r, x_1+r] \times [\tau_1-r^2, \tau_1]$.
    We seek to choose $R > 0$ as large as possible such that $\Lambda_0^R \subset \Omega$. We rely on the fact that $v_1 \leq 0 < V_R$.

    To ensure that $\Lambda_0^R \subset \Omega$, it suffices that the spatial and temporal boundaries satisfy the following condition:
    \begin{equation}\label{eq:condition Lambda_O}
        x_1 + R \leq \min_{0 \leq z \leq R^2} s_1(\tau_1 - z) = \min_{0 \leq z \leq R^2}\left\{V_R\sqrt{2(\tau_1-z)+1}+\int_0^{\tau_1-z}\tilde{J}(x)dx\right\} \quad \text{and} \quad R \leq \sqrt{\tau_1}.
    \end{equation}
    After computation, Eq.\eqref{eq:condition Lambda_O} is satisfied for \[
    R = \min \big( \alpha \sqrt{2\tau_1+1}, \sqrt{\tau_1} \big)\quad \text{with} \quad \alpha = \frac{1}{2} \min \left( \frac{V_R}{\sqrt{2}}, \sqrt{\frac{V_R}{2J_\infty}}, 1 \right).
    \]
    
    Using the regularity estimate \eqref{eq:regularite elliptique}, $q(\tau_1, x_1)$ satisfies the bound
    \[
    |q(\tau_1, x_1)| \leq \frac{K}{R^3} \|q\|_{L^1(\Lambda_0^R)}.
    \]
    Since $q(\tau, \cdot)$ is a probability density for any $\tau \in (\tau_1 - R^2, \tau_1)$, $\|q(\tau)\|_{L^1(x_1-R,x_1+R)} \leq 1$. This leads to
    \[
    |q(\tau_1, x_1)| \leq \frac{K}{R}.
    \]
    Substituting back $R$ and the relationship for $u$, we get
    \[
    \frac{|u(t_1, v_1)|}{\sqrt{2\tau_1+1}} \leq \frac{K}{\min \big( \alpha \sqrt{2\tau_1+1}, \sqrt{\tau_1} \big)}.
    \]
    Therefore
    \[
    |u(t_1, v_1)| \leq \frac{K}{\min \left( \alpha, \sqrt{\frac{\tau_1}{2\tau_1+1}} \right)} \leq K \left[ \frac{1}{\alpha} + \sqrt{\frac{2\tau_1+1}{\tau_1}} \right].
    \]
    Since $\tau_1 \geq \tau_0$ (implying $\frac{2\tau_1+1}{\tau_1} \leq \frac{2\tau_0+1}{\tau_0}$) and $\sqrt{\frac{2\tau_0+1}{\tau_0}} = \frac{2e^{t_0}}{e^{2t_0}-1}$, we can define the constant dependent only on $t_0$ and $J_\infty$.
    By writing
    \begin{equation}\label{eq:constante C}
        C(t_0, J_\infty) = K \left( \frac{1}{\alpha} + \frac{2e^{t_0}}{e^{2t_0}-1} \right),
    \end{equation}
    we obtain the desired estimate \eqref{eq: estimation u elliptique}.
\end{proof}

\begin{proof}(Proof of Theorem \ref{th:bound on N_u(t)}).\\
    The main strategy is to use Proposition \ref{prop:borne u elliptique} with a uniform initial time $t_0$ that depends solely on the initial data.
    
    Let us define the constants:
    \[
        K \coloneqq \left\| \frac{Q}{P} \right\|_\infty \left\| \frac{u(0)}{Q} \right\|_\infty \quad \text{and} \quad \beta_0 \coloneqq \frac{2}{Q(0)} K.
    \]
    We first define the maximal time up to which the solution remains bounded by $\beta_0 Q$
    \[
        t_0 = \sup \left\{ t \geq 0 \mid \forall s \in [0,t], \, u(s,\cdot) < \beta_0 Q(\cdot) \right\}.
    \]
    If $t_0 = \infty$, then $u(t) \leq \beta_0 Q$ for all $t \geq 0$, and the proposition is proved with $\beta_1 = \beta_0$.
    
    Now, assume that $t_0 < \infty$. By continuity, there exists $v_0 \leq V_F$ such that $u(t_0, v_0) = \beta_0 Q(v_0)$.
     According to Proposition \ref{prop:exemple-sur-sol}, $\beta_0 Q$ is a super-solution on $(0,\infty) \times (0, V_F)$. The maximum principle implies that the contact point cannot occur in the domain $(0, V_F)$, hence we must have $v_0 \leq 0$. Note that for $v \leq 0$, $Q(v) = Q(0)$.
    
    Additionally, the function $(t,v) \mapsto K P(t,v)$ is also a super-solution on $(0, t_0) \times (-\infty, V_F)$. Therefore, we have the comparison
    \[
        u(t_0, v_0) \leq K  P(t_0, v_0) \leq K e^{t_0}.
    \]
    Combining the equality $u(t_0, v_0) = \beta_0 Q(v_0)$ with the bound above, we obtain
    \[
        \beta_0 Q(0) \leq \beta_0 Q(v_0) = u(t_0, v_0) \leq K e^{t_0}.
    \]
    Substituting the definition of $\beta_0 = \frac{2K}{Q(0)}$, this yields
    \[
        2K \leq K e^{t_0} \implies e^{t_0} \geq 2 \implies t_0 \geq \log 2.
    \]
    
    Now, let $C(\log 2, J_\infty)$ be the constant from Proposition \ref{prop:borne u elliptique} according to \eqref{eq:constante C}. Note that the constant $C(t_0, J_\infty)$ is generally decreasing with respect to $t_0$, so $C(t_0, J_\infty) \leq C(\log 2, J_\infty)$.
    
    We define our final constant $\beta_1$ as
    \[
        \beta_1 = \max \left( \beta_0, \, \frac{2C(\log 2, J_\infty)}{Q(0)} \right).
    \]
    Similarly, we define $t_1 = \sup \{ t \geq 0 \mid u(t) < \beta_1 Q \}$.
    Suppose $t_1 < \infty$. By the definition of $\beta_1$ (since $\beta_1 \geq \beta_0$), we necessarily have $t_1 \geq t_0 \geq \log 2$.
    There exists $v_1$ such that $u(t_1, v_1) = \beta_1 Q(v_1)$. Since $\beta_1 Q$ is a super-solution on $(0, V_F)$, the contact point must satisfy $v_1 \leq 0$.
    
    However, applying Proposition \ref{prop:borne u elliptique} for $t_1 \geq \log 2$ and $v_1 \leq 0$, we have the uniform bound
    \[
        u(t_1, v_1) \leq C(\log 2, J_\infty).
    \]
    On the other hand, using the definition of $\beta_1$:
    \[
        u(t_1, v_1) = \beta_1 Q(v_1) \geq \beta_1 Q(0) \geq \frac{2C(\log 2, J_\infty)}{Q(0)} Q(0) = 2C(\log 2, J_\infty).
    \]
    This leads to the contradiction
    \[
        2C(\log 2, J_\infty) \leq C(\log 2, J_\infty),
    \]
    which is impossible since $C > 0$. Thus $t_1 = \infty$, and hence, 
    \begin{equation}\label{eq:borne u/Q}
        \fa t \geq 0, \quad u(t)\leq \beta_1Q.
    \end{equation}
    This showes the first part of the Theorem \ref{th:bound on N_u(t)}.
    Using the Hospital rule, we have
\[
\lim_{v\rightarrow V_F}\frac{u(t,v)}{Q(v)} = \frac{\partial_vu(t,V_F)}{-1}=N_u(t).
\]
    Using the bound of Proposition \eqref{eq:borne u/Q}, we obtain 
    \[
    \forall t \geq 0,\quad N_u(t) \leq \beta_1,
    \]
    which finally proves the Theorem \ref{th:bound on N_u(t)}.
\end{proof}

\begin{remark}\label{rk: beta increasing} 
Notice that $\beta_1$ is also an increasing function with respect to $J_\infty$ and $\|\frac{u^0}{Q}\|_\infty$. 
\end{remark}
\section{Proof of Theorem \ref{th:cv_long_delay} : inductive procedure}\label{sec:preuvethm}

\begin{proof}[Proof of Theorem \ref{th:cv_long_delay}]
 The core of the proof of Theorem \ref{th:cv_long_delay} relies on a mathematical induction over successive time intervals of length $d$, showing that the solution $u_d$ relaxes toward a stationary state faster than the drift (induced by the delay) changes.
 Recall that
\[
\quad \tilde{u}_\infty(v,\tau,N_{\text{ini}}) = \sum_{k=0}^{\infty}
\overline{u}\!\left(\Phi_b^{\,k}(N_{\mathrm{ini}}),v\right)
\,\mathbf{1}_{[k,k+1)}(\tau), \quad \text{and }\quad  \tilde{N}_\infty(\tau,N_{\text{ini}}) =  \sum_{k=0}^{\infty}
\Phi_b^{\,k}(N_{\mathrm{ini}})
\,\mathbf{1}_{[k,k+1)}(\tau).
\]
By convention, we take $u(t-d) = u^0$ for $t\in (0,d)$. To formalize the asymptotic behavior, we first introduce three quantities that characterize the stability and convergence for each iteration $k \geq -1$:
\begin{itemize}
    \item Uniform $L^\infty$ control:
    $A_k = \sup_{d > 0} \sup_{0 \leq t \leq d} \left\| \frac{u_d(t+kd)}{Q} \right\|_\infty$.
    \item Cumulative flux deviation:
    $C_k(\alpha) = \sup_{d > 0} \int_0^d |N_d(t + kd) - \Phi_b^{k}(N_{\text{ini}})| e^{\alpha t} dt$.
    \item Exponential relaxation of the profile: For $k \geq 0$, 
    $$D_k(\alpha) = \sup_{d > 0} \sup_{0 \leq t \leq d} \left( \|u_d(t + kd) - \overline{u}( \Phi_b^k(N_{\text{ini}}))\|_{L^1_v} e^{\alpha t} \right).$$
\end{itemize}
The theorem is established if we show that for any $M \in \mathbb{N}$, there exists a sequence of rates $\alpha_k > 0$ such that $A_k, C_k(\alpha_k),$ and $D_k(\alpha_k)$ are finite. Indeed, if these bounds hold uniformly in $d$, then for any fixed $M$:
\begin{align*}
    \int_0^M \big|\tilde{N}_d(\tau) - \tilde{N}_{\infty}(\tau)\big| \, d\tau &= \sum_{k=0}^{M-1} \int_{k}^{k+1} |\tilde{N}_d(\tau) - \Phi_b^k(N_{\text{ini}})| \, d\tau \\
    &= \frac{1}{d} \sum_{k=0}^{M-1} \int_{0}^d |N_d(t+kd) - \Phi_b^k(N_{\text{ini}})| \, dt \\
    &\leq \frac{1}{d} \sum_{k=0}^{M-1} C_{k}(\alpha_{k}) \xrightarrow[d \to \infty]{} 0.
\end{align*}
An analogous argument holds for the convergence of the density $\tilde{u}_d$ in $L^1_v$.
\medskip

\noindent \underline{ Initialization ($k=0$).} 
By convention, the solution $u_d$ is constant on the  interval $(-d, 0)$, which implies $A_{-1} = \|u^0/Q\|_\infty < \infty$. Moreover, the assumption on the initial condition ensures that the delayed flux is constant throughout the first block: 
\[
\forall d > 0, \, \forall t \in [0, d), \quad N_u^d(t - d) = N_{\text{ini}} = \Phi_b^0(N_{\text{ini}}).
\]
Consequently, the cumulative deviation is initially zero, i.e., $C_{-1}(\alpha) = 0$ for any $\alpha > 0$.
On the time interval $(0, d)$, the system thus reduces to a linear Fokker-Planck equation with a time-invariant drift $v \mapsto -v + bN_{\text{ini}}$. According to Theorem \ref{th:convergence exp u et N}, there exist constants $\alpha_0, M_0 > 0$ such that for all $t \in (0, d)$ and all $d > 0$:
\[
\|u_d(t) - \overline{u}(\Phi_b^0(N_{\text{ini}}))\|_{L^1_v} \leq M_0 e^{-\alpha_0 t} \|u^0 - \overline{u}(\Phi_b^0(N_{\text{ini}}))\|_{L^1_v}.
\]
This immediate relaxation of the density towards the stationary profile $\overline{u}(\Phi_b^0)$ ensures that $D_0(\alpha_0) < \infty$, thereby completing the initialization of the recursive procedure.

\medskip
\noindent \underline{ Inductive step.} 
Assume that for some $k \in \mathbb{N}^*$,  for all $i \in \{-1, \dots, k-1\}$, there exists $\alpha_i > 0$ such that $C_i(\alpha_i) < \infty$, $A_i < \infty$, and $D_{i^+}(\alpha_i) < \infty$, where $i^+ = \max(i, 0)$. We now extend these bounds to the interval $[kd, (k+1)d]$.

\medskip
\noindent \textit{Relaxation of the density ($D_k$):} 
The finiteness of $A_{k-1}$ ensures that the delayed neuronal flux from the previous interval is uniformly bounded:
\begin{equation}\label{eq:majoration N delai}
    \sup_{d > 0} \sup_{0 \leq t \leq d} N_u^d(t+(k-1)d) \leq A_{k-1}.
\end{equation}
For any $d > 0$, the function $t \mapsto u_d(kd+t)$ solves the linear Fokker-Planck equation on $[0, d]$ with drift $J^d(t) = bN_u^d(t+(k-1)d)$. To apply the contraction results uniformly, we extend this drift into $\tilde{J}^d$ defined by:
\begin{equation}\label{eq:J tilde}
    \tilde{J}^d(t) = 
\begin{cases} 
    J^d(t), & t < d, \\[4pt] 
    J^d(d), & t \ge d,
\end{cases}
\quad \text{and we denote} \quad J_\infty^d = \|\tilde{J}_d\|_\infty.
\end{equation}
Let $\Gamma_d$ be the evolution operator associated with $\tilde{J}^d$. By the uniqueness of the solution, we have $u_d(t+kd) = \Gamma_d(t, 0) u_d(kd)$ for $t \in [0, d]$. Since the family $(\tilde{J}^d)_{d>0}$ is uniformly bounded by $|b| A_{k-1}$, Theorem \ref{th:convergence exp u et N} ensures that there exist uniform constants $\overline{M} \geq 1$ and $\overline{\alpha} > 0$ such that, comparing with the target equilibrium $\overline{u}(\Phi_b^{k}(N_{\text{ini}})))$, we have for all $d > 0$ and $0 \leq t \leq d$:
\begin{align*}
    \|u_d(t+kd) - & \overline{u}(\Phi_b^{k}(N_{\text{ini}}))\|_{L^1_v}\leq \overline{M} \bigg(  \|u_d(kd) - \overline{u}(\Phi_b^{k}(N_{\text{ini}}))\|_{L_v^1} \\
    & + |b| \|\partial_v \overline{u}_{\Phi_b^{k}(N_{\text{ini}})}\|_{L^1_v} \int_0^t |N_u^d(s+(k-1)d) - \Phi_b^{k}(N_{\text{ini}})| e^{\alpha_{k-1} s} ds \bigg) e^{-\alpha_k' t},
\end{align*}
where we set $\alpha_k' = \min(\overline{\alpha}, \alpha_{k-1})$. The first term on the right-hand side represents the "initial error" at the start of the block, which we decompose as:
\begin{align*}
\|u_d(kd) - \overline{u}(\Phi_b^{k}(N_{\text{ini}}))\|_{L_v^1} &\leq \|u_d(kd) - \overline{u}(\Phi_b^{k-1}(N_{\text{ini}}))\|_{L_v^1} + \| \overline{u}(\Phi_b^{k-1}(N_{\text{ini}})) - \overline{u}(\Phi_b^{k}(N_{\text{ini}}))\|_{L_v^1} \\
&\leq D_{k-1}(\alpha_{k-1}) e^{-\alpha_{k-1} d} + \| \overline{u}(\Phi_b^{k-1}(N_{\text{ini}})) - \overline{u}(\Phi_b^{k}(N_{\text{ini}}))\|_{L_v^1}.
\end{align*}
Since $e^{-\alpha_{k-1} d} \leq 1$, this quantity is uniformly bounded in $d$. Furthermore, the integral term in the sensitivity estimate is bounded by $|b|C_{k-1}(\alpha_{k-1})$. Taking the supremum over $d>0$ and $t \in [0, d]$ of the quantity $\|u_d(t+kd) - \overline{u}(\Phi_b^{k}(N_{\text{ini}}))\|_{L^1} e^{\alpha_k' t}$, we conclude:
\begin{equation*}
    D_k(\alpha_k') \leq \overline{M} \left( D_{k-1}(\alpha_{k-1}) + \| \overline{u}(\Phi_b^{k-1}(N_{\text{ini}})) - \overline{u}(\Phi_b^{k}(N_{\text{ini}}))\|_{L_v^1} + |b| \|\partial_v \overline{u}_{\Phi_b^{k}(N_{\text{ini}})}\|_{L^1_v} C_{k-1}(\alpha_{k-1}) \right) < \infty,
\end{equation*}
which proves that $D_k(\alpha_k') < \infty$.

\medskip
\noindent \textit{Convergence of the flux ($C_k$):} 
To bound the neuronal flux deviation, we apply the stability estimate \eqref{eq:borne int N_u} from Theorem \ref{th:convergence exp u et N} to the solution $u_d(t+kd)$ on the interval $t \in [0, d]$. We set the target drift $J' = b\Phi_b^{k-1}$ and the rate $\delta = \frac{1}{2} \min(\overline{\alpha}, \alpha_{k-1})$. According to the Theorem, we have $\Phi_1(J') = \Phi_b^k$, and the estimate yields:
\begin{align*}
 (V_F - V_R)^2 \int_0^d |N_u^d(t+kd) - \Phi_b^{k}(N_{\text{ini}})| e^{\delta t} dt &\leq \left\|(v-V_R)^2 \big(u_d(kd)-\overline{u}(\Phi_b^{k}(N_{\text{ini}}))\big)\right\|_{L^1} \\
  &\quad + \frac{c(J^d_\infty)}{\min(\overline{\alpha}, \alpha_{k-1})-\delta} C_d \\
  &\quad + K(\Phi_b^{k}(N_{\text{ini}})) \int_0^d |b| |N_u^d(t+(k-1)d) - \Phi_b^{k-1}(N_{\text{ini}})| e^{\delta t} dt,
\end{align*}
where $C_d$ is the constant defined in \eqref{eq:inégalité exp sur u} evaluated at $T=d$. By the induction hypothesis $D_{k-1}(\alpha_{k-1}) < \infty$ and $C_{k-1}(\alpha_{k-1}) < \infty$, the term $C_d$ is uniformly bounded in $d$:
\[
C_d \leq \overline{M} \left( \lVert u_d(kd) - \overline{u}(\Phi_b^{k}(N_{\text{ini}})) \rVert_{L^1_v} + |b| \lVert \partial_v \overline{u}(\Phi_b^{k}(N_{\text{ini}})) \rVert_{L^1_v} C_{k-1}(\alpha_{k-1}) \right) < \infty.
\]
The first term on the right-hand side is bounded via a second-order moment estimate:
\begin{align*}
\|(v-V_R)^2 \big(u_d(kd)-\overline{u}(\Phi_b^k(N_{\text{ini}}))\big)\|_{L^1}
&\leq \kappa' \Big( \|u_d(kd)-\overline{u}(\Phi_b^{k-1}(N_{\text{ini}}))\|_{L_v^1} \\
&+ \|\overline{u}(\Phi_b^{k-1}(N_{\text{ini}})) -\overline{u}(\Phi_b^{k}(N_{\text{ini}}))\|_{L_v^1} \Big) \\
&\leq \kappa' \Big( D_{k-1}(\alpha_{k-1}) + \|\overline{u}(\Phi_b^{k-1}(N_{\text{ini}})) -\overline{u}(\Phi_b^{k}(N_{\text{ini}}))\|_{L_v^1} \Big),
\end{align*}
which is bounded uniformly in $d$. Finally, the last integral is bounded by $|b|K(\Phi_b^{k}(N_{\text{ini}}))C_{k-1}(\alpha_{k-1})$ since $\delta \leq \alpha_{k-1}$. This concludes that:
\[
\sup_{d>0} \int_0^d |N_u^d(t+kd) - \Phi_b^{k}(N_{\text{ini}})| e^{\delta t} dt < \infty,
\]
proving $C_k(\delta) < \infty$.

\medskip
\noindent \textit{Uniform bound ($A_k$):} Finally, we show that $A_k < \infty$. Using the function $\tilde{J}^d$ defined in \eqref{eq:J tilde} and Theorem \ref{th:bound on N_u(t)}, we have
\[
\sup_{0 \leq t \leq d} \bigg\|\frac{u_d(t+kd)}{Q}\bigg\|_\infty = \sup_{0 \leq t \leq d} \bigg\|\frac{\Gamma_d(t, 0) u_d(kd)}{Q}\bigg\|_\infty \leq \beta\left(J_\infty^d, \Big\|\frac{u_d(kd)}{Q}\Big\|_\infty\right).
\]
By the induction hypothesis, $\|u(kd)/Q\|_\infty \leq A_{k-1}$ and $J_\infty^d \leq A_{k-1}$.
From Remark \ref{rk: beta increasing}, we obtain the uniform bound
\[
A_k \leq \sup_{d>0} \beta_1\left(J_\infty^d, \Big\|\frac{u_d(kd)}{Q}\Big\|_\infty\right) \leq \beta_1(A_{k-1}, A_{k-1}) < \infty.
\]
\medskip
\noindent \underline{Conclusion.} 
By induction, the sequences $A_k, C_k, D_k$ are finite for all $k \leq M$. This confirms that on any finite time interval $[0, M]$, the rescaled solution $\tilde{u}_d$ and flux $\tilde{N}_d$ converge toward the step functions governed by the iterations of $\Phi_b$.
\end{proof}

\section{Periodic Limits and numerical simulation}
Corollary \ref{cor:double_limit_correct} enables two distinct dynamics of the limit function $\tilde{N}_\infty(\cdot)=\lim_{d\to\infty} N_d(\cdot \times d)$ for $b<0$:
\begin{enumerate}
    \item Convergence towards a unique stationary state $\overline{N}$ that coincide with the limit of the sequence $(\Phi_b^k(N_{\text{ini}}))_{k\geq 0}$.
    \item Convergence towards the 2-periodic function $N_{per}$ that is characterized by the $2-$cycle limit of the sequence $(\Phi_b^k(N_{\text{ini}}))_{k\geq 0}$.
\end{enumerate}
We want to emphasize these distinct behaviors with numerical simulations. 
We used an upwind scheme as follows.
The discrete version of $u(t,x)$ is $u(n\Delta t, j\Delta x) \equiv u^n_j$ and the discretized and finite domain is
$$
\mathcal{D}=\{0, \Delta t, ..., N_t\Delta t\}\times \{V_{min}, V_{min}+\Delta x,...,V_{min}+N_x\Delta x\}.
$$
With $N_x$ satisfying $N_x= \frac{V_F-V_{min}}{\Delta x}$. The discrete delay $D$ satisfies $D=\big[\frac{d}{\Delta t}\big]$. Notice that $(n\Delta t, V_{min}+ j\Delta x)\in \mathcal{D}\Leftrightarrow (n,j)\in \{0,...,N_t\}\times\{0,...,N_x\}$.
Then, 
\begin{equation}
    \frac{u_j^{n+1}-u_j^n}{\Delta t} =  \frac{(\mu_j^n)_+u_j^n-(\mu_{j-1}^n)_+u_{j-1}^{n-1}}{\Delta x}+ \frac{(\mu_{j+1}^n)_-u_{j+1}^n-(\mu_j^n)_-u_{j}^{n-1}}{\Delta x}+\frac{u^n_{j+1}-2u^n_j+u^n_{j-1}}{(\Delta x)^2}
\end{equation}
For any $\mu \in \R, \mu_+ = \max(\mu,0)$ and $\mu_- = \min(\mu,0)$ and 
\begin{equation}
    \mu_j^n = bN^{n-D}-(V_{min}+j\Delta x)\quad\text{with}\quad N^{n} = \frac{4u^n_{N_x-1}-u^n_{N_x-2}}{2\Delta x}.
\end{equation}.
We chose numerical values $V_F, V_R, V_{min}$, $\Delta t$ and $\Delta x$ for stability as follows
\begin{enumerate}
    \item $V_R = 0.5, V_F = 2.0$
    \item $V_{min} = -10.0$
    \item $\Delta t = 2.10^{-5}$
    \item $\Delta x = 10^{-2}$
\end{enumerate}
Since the convergence shown in Theorem \ref{th:cv_long_delay} is local in time, an increase of the error term $\|\tilde{u}_d(\cdot)-\tilde{u}_\infty\|$ on every interval $[k,k+1]$ is expected. This is naturally caused by the time needed for the solution $\Tilde{u}_d$ to move continuously from close to $\overline{u}(\Phi^k_b(N_{\text{ini}}))$ to close to $\overline{u}(\Phi^{k+1}_b(N_{\text{ini}}))$ on $[k,k+1]$ and initially from $u^0$ to $\overline{u}(N_{\text{ini}}) $ on $[0,1]$. Hence, to highlihght the convergence of Theorem \ref{th:cv_long_delay}, we choose a smooth function with its mass mainly distributed close to $(V_R,V_F)$, like $\overline{u}(N_{\text{ini}}) $ in order to get a rapid convergence towards the first pseudo-equilibria. More concretely, we chose $u^0$ such that 
\[
u^0(x) \approx \sqrt{\frac{\pi}{4}} \exp(-4x^2).
\]
\begin{figure}[h]
    \centering
    \includegraphics[width=0.4\textwidth]{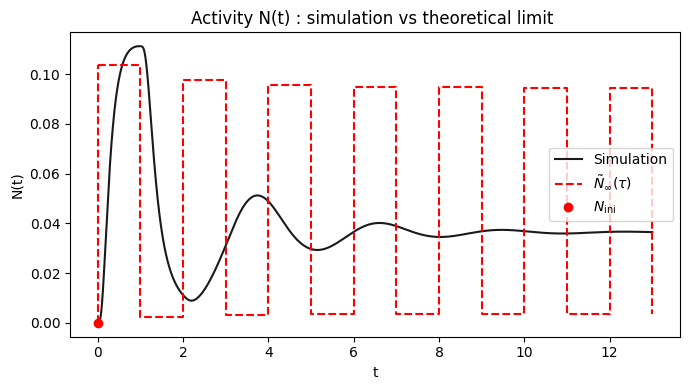}
    \\[-7pt]
    \caption{$N_d(t)$ and $N_\infty(d\,\cdot)$ when $d(=1)$ is small and $b(=-15)<b^*<0$.}
    \label{fig:1}
\end{figure}
Figure~\ref{fig:1} shows the simulation of the delayed Fokker-Planck equation with delay $d=1$ and $b=-15<b^*$. Under this condition, the simulation of $N(t)$ converges towards the stationary state while the sequence $(\Phi_b^n(N_{\text{ini}}))_{n\geq 0}$ converges to the $2$-cycle $(N_+,N_-)$. This emphasizes that a large delay is necessary for the appearance of periodic solutions with very negative connectivity $b<b^*$.\\
\begin{figure}[h]
    \centering
    \includegraphics[width=0.4\textwidth]{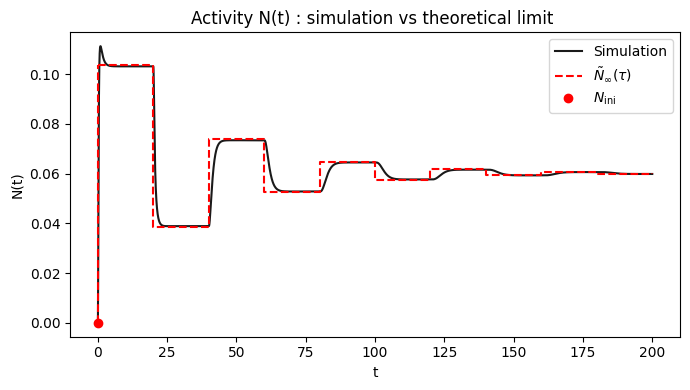}
    \\[-7pt]
    \caption{$N_d(t)$ and $N_\infty(d\,\cdot)$ when $d(=20)$ is large and $b^*<b(=-5)<0$.}
    \label{fig:2}
\end{figure}
Figure~\ref{fig:2} presents numerical evidence of point~1 of Corollary~\ref{cor:double_limit_correct}. By choosing a large delay ($d=20$) and a small connectivity coefficient $b\in (b^*,0)$, both the simulation of $N(t)$ and the sequence $\Phi_b^n(N_{\text{ini}})$ converge to the same value $\overline{N}$.
Numerical simulation shows an even stronger result: since $d<\infty$, we have numerically
\[
\lim_{M\to\infty}
\int_M^{M+1}
\bigl|
\tilde N_d(\tau) - \overline N
\bigr|
\, d\tau
=0.
\]
A rigorous proof of such a result, the global stability of the stationnary state in the delayed NNLIF for weakly non linear regime with exponential rate can be found in Section~5 of \cite{IRSS2022} as well as in \cite{zbMATH08030978}.
\begin{figure}[h]
    \centering
    \includegraphics[width=0.4\textwidth]{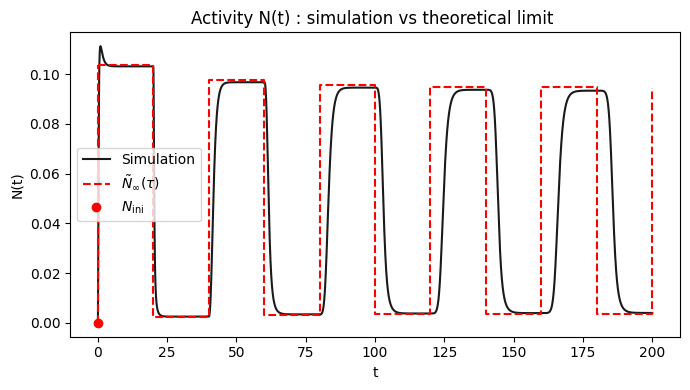}
    \\[-7pt]
    \caption{$N_d(t)$ and $N_\infty(d\,\cdot)$ when $d(=20)$ is large and $b(=-15)<b^*<0$.}
    \label{fig:3}
\end{figure}
\\

Finally, Figure~\ref{fig:3} provides numerical evidence of point~2 of Corollary~\ref{cor:double_limit_correct}, stating that when $b<b^*$ (here $b=-15$) and $d$ is large (here $d=20$), $\tilde{N}_d(t)$ is close to the periodic function $N_{\text{per}}$. Notice that the simulation highlights the fact that the convergence of $N_d$ is local. Indeed, in Figure~\ref{fig:3}, the phase transition of the function $N(t)$, going from $\Phi_b^k(N_{\text{ini}})$ to $\Phi_b^{k+1}(N_{\text{ini}})$, is amplified on every interval  $[kd,(k+1)d]$.

\section{Conclusion}
Our study covers the local convergence of the solution to the integrate-and-fire model when the delay $d$ goes to $+\infty$.  
The Doeblin–Harris method was used to obtain an exponential contraction of the unstationary evolution operator $\Gamma(t,s)$, associated with the drift $v\mapsto -v-J(t)$. Then, using a perturbation argument, we establish that if the term $J(t)$ satisfies some type of convergence stated as \eqref{eq: conv J}, then it is also the case for $N(t)$.
Using an upper solution of \eqref{eq:FP linéaire instat} and parabolic regularization, we propagate the bound on $t\mapsto J(t)$ to $t\mapsto N(t)$, which finally enables us to propagate the exponential convergence of the solution $(u_d,N_d)$ from $[kd,(k+1)d]$ to $[(k+1)d,(k+2)d]$ and conclude the proof of Theorem~\ref{th:cv_long_delay}.

One may wonder whether the local convergence still holds if $b>0$. This would require proving Theorem~\ref{th:convergence doeblin-harris}, \ref{th:convergence exp u et N} and \ref{th:bound on N_u(t)} in the case where $J(t)\leq 0$ and $|J(t)|\leq J_\infty$. Even if Theorem~\ref{th:convergence doeblin-harris} and Theorem~\ref{th:convergence exp u et N} can be proved by the same method, due to the lack of an upper solution in the excitatory case, proving Theorem~\ref{th:bound on N_u(t)} is still an open problem.

In any case, the asymptotic dynamics would be different; \cite{CaceresCanizo2024} 
showed that in the excitatory regime, the sequence $(\Phi^k(N_{\text{ini}}))_{k\geq 0}$ 
exhibits two distinct behaviors: either it is monotone and convergent, or it satisfies 
$\lim_{k\rightarrow\infty}\Phi^k(N_{\text{ini}})=+\infty$. 
We may also ask whether the local Cesàro mean convergence still holds when the space is endowed with a stronger norm, such as $\|\cdot\|_{L^2(\overline{p}(0)^{-1})}$ or $\|\cdot\|_\infty$ instead of $\|\cdot\|_{L^1_v}$. This question remains open. Addressing it would require exploiting the regularizing effects studied in \cite{zbMATH08030978}, together with exponential convergence results based on entropy dissipation.

However, in this stronger framework, Theorem~\ref{th:convergence exp u et N} would require a completely different proof. Indeed, the available exponential convergence results with $\|\cdot\|_{L^2(\overline{p}(0)^{-1})}$ or $\|\cdot\|_\infty$ apply only to semigroups of the form $e^{tL}$, where the operator $L$ is time-independent, whereas the operator $\Gamma(t,s)$ considered here is time-dependent.\\

\textbf{Acknowledgment.}This work was supported by the Fondation Simone et Cino Del Duca, Institut de France

\bibliographystyle{siam}  
\bibliography{BibDPSZ}

@article{zbMATH08023785,
 author = {Carrillo, Jos{\'e} A. and Roux, Pierre},
 title = {Nonlinear partial differential equations in neuroscience: from modeling to mathematical theory},
 fjournal = {M\(^3\)AS. Mathematical Models \& Methods in Applied Sciences},
 journal = {Math. Models Methods Appl. Sci.},
 issn = {0218-2025},
 volume = {35},
 number = {2},
 pages = {403--584},
 year = {2025},
 language = {English},
 doi = {10.1142/S0218202525400044},
 keywords = {35-02,92-02,35Q49,35Q84,35B40,65C30},
 zbMATH = {8023785},
 Zbl = {1562.35002}
}

@misc{arXiv:2601.19282,
 author = {Perthame, Beno{\^{\i}}t and Rieutord, Cl{\'e}ment and Salort, Delphine},
 title = {A {Fokker}-{Planck} equation with superlinear drift at infinity for {Integrate}-and-{Fire} model},
 year = {2026},
 howpublished = {Preprint, {arXiv}:2601.19282 [math.{AP}] (2026)},
 url = {https://arxiv.org/abs/2601.19282},
 arXiv = {arXiv:2601.19282}
}

@article{zbMATH08030978,
 author = {C{\'a}ceres, Mar{\'{\i}}a J. and Ca{\~n}izo, Jos{\'e} A. and Ramos-Lora, Alejandro},
 title = {On the asymptotic behavior of the {NNLIF} neuron model for general connectivity strength},
 fjournal = {Communications in Mathematical Physics},
 journal = {Commun. Math. Phys.},
 issn = {0010-3616},
 volume = {406},
 number = {5},
 pages = {55},
 note = {Id/No 115},
 year = {2025},
 language = {English},
 doi = {10.1007/s00220-025-05287-5},
 keywords = {35Q84,92C20,35B40,35B65,45D05,35A01,35A02,35R60},
 zbMATH = {8030978},
 Zbl = {1565.35328}
}


@article{zbMATH08113917,
 author = {Perthame, Beno{\^{\i}}t and Rieutord, Cl{\'e}ment and Salort, Delphine},
 title = {Strongly nonlinear age-structured equation, time-elapsed model and large delays},
 fjournal = {Journal of Mathematical Biology},
 journal = {J. Math. Biol.},
 issn = {0303-6812},
 volume = {91},
 number = {5},
 pages = {29},
 note = {Id/No 65},
 year = {2025},
 language = {English},
 doi = {10.1007/s00285-025-02294-x},
 keywords = {35B10,35B40,35Q49,92D25},
 zbMATH = {8113917}
}


@book{zbMATH01061253,
 author = {Lieberman, Gary M.},
 title = {Second order parabolic differential equations},
 isbn = {981-02-2883-X},
 year = {1996},
 publisher = {Singapore: World Scientific},
 language = {English},
 keywords = {35-02,35K20,35Bxx,35K55,35K10},
 zbMATH = {1061253}, 
 Zbl = {0884.35001}
}

@article{BretteG2005,
author = {Brette, Romain and Gerstner, Wulfram},
title = {Adaptive Exponential Integrate-and-Fire Model as an Effective Description of Neuronal Activity},
journal = {Journal of Neurophysiology},
volume = {94},
number = {5},
pages = {3637-3642},
year = {2005},
doi = {10.1152/jn.00686.2005},
    note ={PMID: 16014787},
URL = {https://doi.org/10.1152/jn.00686.2005},
 abstract = { We introduce a two-dimensional integrate-and-fire model that combines an exponential spike mechanism with an adaptation equation, based on recent theoreticl findings. We describe a systematic method to estimate its parameters with simple electrophysiological protocols (current-clamp injection of pulses and ramps) and apply it to a detailed conductance-based model of a regular spiking neuron. Our simple model predicts correctly the timing of 96\% of the spikes (Â±2 ms) of the detailed model in response to injection of noisy synaptic conductances. The model is especially reliable in high-conductance states, typical of cortical activity in vivo, in which intrinsic conductances were found to have a reduced role in shaping spike trains. These results are promising because this simple model has enough expressive power to reproduce qualitatively several electrophysiological classes described in vitro. }
}

@misc{CR2025,
      title={Nonlinear partial differential equations in neuroscience: from modelling to mathematical theory}, 
      author={José A Carrillo and Pierre Roux},
      year={2025},
      eprint={2501.06015},
      archivePrefix={arXiv},
      primaryClass={math.AP},
      url={https://arxiv.org/abs/2501.06015}, 
}


@article{BL_2003,
    author = {Brunel, Nicolas and Latham, Peter E.},
    title = {Firing Rate of the Noisy Quadratic Integrate-and-Fire Neuron},
    journal = {Neural Computation},
    volume = {15},
    number = {10},
    pages = {2281-2306},
    year = {2003},
    month = {10},
    abstract = {We calculate the firing rate of the quadratic integrate-and-fire neuron in response to a colored noise input current. Such an input current is a good approximation to the noise due to the random bombardment of spikes, with the correlation time of the noise corresponding to the decay time of the synapses. The key parameter that determines the firing rate is the ratio of the correlation time of the colored noise, τs, to the neuronal time constant, τm. We calculate the firing rate exactly in two limits: when the ratio, τs/τm, goes to zero (white noise) and when it goes to infinity. The correction to the short correlation time limit is O(τs/τm), which is qualitatively different from that of the leaky integrate-and-fire neuron, where the correction is O(√τs/τm). The difference is due to the different boundary conditions of the probability density function of the membrane potential of the neuron at firing threshold. The correction to the long correlation time limit is O(τm/τs). By combining the short and long correlation time limits, we derive an expression that provides a good approximation to the firing rate over the whole range of τs/τm in the suprathreshold regime—that is, in a regime in which the average current is sufficient to make the cell fire. In the subthreshold regime, the expression breaks down somewhat when τs becomes large compared to τm.},
    issn = {0899-7667},
    doi = {10.1162/089976603322362365},
    url = {https://doi.org/10.1162/089976603322362365},
    eprint = {https://direct.mit.edu/neco/article-pdf/15/10/2281/815765/089976603322362365.pdf},
}

@unpublished{PRS3,
author = {Perthame, Beno{\^{\i}}t and Rieutord, Cl{\'e}ment and Salort, Delphine},
title = {Nonlinear {F}okker-{P}lanck equation with superlinear drift},
year      = {In preparation},
} 

@misc{DPSZ2024,
      title={Noisy integrate-and-fire equation: continuation after blow-up}, 
      author={Xu'An Dou and Benoît Perthame and Delphine Salort and Zhennan Zhou},
      year={2024},
      eprint={2409.14749},
      archivePrefix={arXiv},
      primaryClass={math.AP},
      url={https://arxiv.org/abs/2409.14749}, 
}

@incollection{ledoux,
  author    = {Ledoux, M.},
  title     = {The concentration of measure phenomenon},
  booktitle = {AMS math. surveys and monographs},
  editor    = {},
  publisher = {AMS},
  pages     = {},
  year      = {2001},
    address = {},
    volume  = {89},
    edition = {}
}

@misc{Harris_1956,
 author = {Harris, T. E.},
 title = {The existence of stationary measures for certain {Markov} processes},
 year = {1956},
 language = {English},
 howpublished = {Proc. 3rd {Berkeley} {Sympos}. {Math}. {Statist}. {Probability} 2, 113-124 (1956).},
 zbMATH = {3121180},
 Zbl = {0072.35201}
}

@incollection{HairerM-Doeblin,
 author = {Hairer, Martin and Mattingly, Jonathan C.},
 title = {Yet another look at {Harris}' ergodic theorem for {Markov} chains},
 booktitle = {Seminar on stochastic analysis, random fields and applications VI. Centro Stefano Franscini, Ascona (Ticino), Switzerland, May 19--23, 2008.},
 isbn = {978-3-0348-0020-4; 978-3-0348-0021-1},
 pages = {109--117},
 year = {2011},
 publisher = {Basel: Birkh{\"a}user},
 language = {English},
 doi = {10.1007/978-3-0348-0021-1_7},
 keywords = {60J05,37A30},
 zbMATH = {6071108},
 Zbl = {1248.60082}
}

@article{Figalli2008,
 author = {Figalli, Alessio},
 title = {Existence and uniqueness of martingale solutions for {SDEs} with rough or degenerate coefficients},
 fjournal = {Journal of Functional Analysis},
 journal = {J. Funct. Anal.},
 issn = {0022-1236},
 volume = {254},
 number = {1},
 pages = {109--153},
 year = {2008},
 language = {English},
 doi = {10.1016/j.jfa.2007.09.020},
 keywords = {60H10,34F05,60G44},
 zbMATH = {5234457},
 Zbl = {1169.60010}
}

@article{DiPL89,
 author = {DiPerna, R. J. and Lions, P. L.},
 title = {Ordinary differential equations, transport theory and {Sobolev} spaces},
 fjournal = {Inventiones Mathematicae},
 journal = {Invent. Math.},
 issn = {0020-9910},
 volume = {98},
 number = {3},
 pages = {511--547},
 year = {1989},
 language = {English},
 doi = {10.1007/BF01393835},
 keywords = {34G20,35D05,35D10},
 url = {https://eudml.org/doc/143741},
 zbMATH = {4140224},
 Zbl = {0696.34049}
}

@article {Lebris_Lions_2008,
    AUTHOR = {Le Bris, C. and Lions, P.-L.},
     TITLE = {Existence and uniqueness of solutions to {F}okker-{P}lanck
              type equations with irregular coefficients},
   JOURNAL = {Comm. Partial Differential Equations},
  FJOURNAL = {Communications in Partial Differential Equations},
    VOLUME = {33},
      YEAR = {2008},
    NUMBER = {7-9},
     PAGES = {1272--1317},
      ISSN = {0360-5302},
   MRCLASS = {35K20 (35B35 35B45 60H15 82C70)},
  MRNUMBER = {2450159},
MRREVIEWER = {C\'{e}dric Villani},
       DOI = {10.1080/03605300801970952},
       URL = {https://doi-org.accesdistant.sorbonne-universite.fr/10.1080/03605300801970952},
}

@book{brezis2010functional,
  title={Functional analysis, Sobolev spaces and partial differential equations},
  author={Brezis, Haim},
  year={2010},
  publisher={Springer Science \& Business Media}
}

@book{AmannBook,
 author = {Amann, Herbert},
 title = {Linear and quasilinear parabolic problems. {Vol}. 1: {Abstract} linear theory},
 fseries = {Monographs in Mathematics},
 series = {Monogr. Math., Basel},
 issn = {1017-0480},
 volume = {89},
 isbn = {3-7643-5114-4},
 year = {1995},
 publisher = {Basel: Birkh{\"a}user},
 language = {English},
 keywords = {35-02,35G10,35K25,47D06,58D25},
 zbMATH = {742737},
 Zbl = {0819.35001}
}

@article {NDB_2025,
    AUTHOR = {Nabti, A. and Djilali, S. and Belghit, M.},
     TITLE = {Dynamics of a Double Age-Structured {SEIRI} Epidemic Model},
   JOURNAL = {Acta Appl. Math.},
  FJOURNAL = {Acta Applicandae Mathematicae},
    VOLUME = {In press},
      YEAR = {2025},
}

@book{murray1,
	address = {New York, NY},
	author = {Murray, J D},
	date-added = {2018-11-17 19:27:35 +0100},
	date-modified = {2018-11-17 19:27:35 +0100},
	edition = {3rd},
	publisher = {Springer Verlag},
	title = {Mathematical Biology: I. An Introduction},
	year = {2002}}

@book {murray2,
    AUTHOR = {Murray, J. D.},
     TITLE = {Mathematical biology. {II}},
    SERIES = {Interdisciplinary Applied Mathematics},
    VOLUME = {18},
   EDITION = {Third},
      NOTE = {Spatial models and biomedical applications},
 PUBLISHER = {Springer-Verlag},
   ADDRESS = {New York},
      YEAR = {2003},
     PAGES = {xxvi+811},
      ISBN = {0-387-95228-4},
   MRCLASS = {92-02 (92B05 92C05 92D30)},
  MRNUMBER = {MR1952568 (2004b:92001)},
MRREVIEWER = {Trachette L. Jackson},
}

@article {DH1,
	AUTHOR = {Dumont, Gr\'{e}gory and Henry, Jacques and Tarniceriu, Carmen Oana},
	TITLE = {Noisy threshold in neuronal models: connections with the noisy
	leaky integrate-and-fire model},
	JOURNAL = {J. Math. Biol.},
	FJOURNAL = {Journal of Mathematical Biology},
	VOLUME = {73},
	YEAR = {2016},
	NUMBER = {6-7},
	PAGES = {1413--1436},
	ISSN = {0303-6812},
	MRCLASS = {92C20 (35Q92)},
	MRNUMBER = {3556838},
	DOI = {10.1007/s00285-016-1002-8},
	URL = {https://doi.org/10.1007/s00285-016-1002-8},
}

@article{caceres2018global,
  title={Global-in-time classical solutions and qualitative properties for the NNLIF neuron model with synaptic delay},
  author={C{\'a}ceres, M.J. and Roux, P. and Salort, D. and Schneider, R.},
  journal={arXiv preprint arXiv:1806.01934},
  year={2018},
}

@article {IRSS2022,
    AUTHOR = {Ikeda, Kota and Roux, Pierre and Salort, Delphine and Smets,
              Didier},
     TITLE = {Theoretical study of the emergence of periodic solutions for
              the inhibitory {NNLIF} neuron model with synaptic delay},
   JOURNAL = {Math. Neurosci. Appl.},
  FJOURNAL = {Mathematical Neuroscience and Applications},
    VOLUME = {2},
      YEAR = {2022},
     PAGES = {Art. No. 4, 37},
   MRCLASS = {92B20 (35B10 35K20 92B25)},
  MRNUMBER = {4665143},
MRREVIEWER = {Shangjiang Guo},
}

@Article{FonteS2022,
 Author = {Fonte, Claudia and Schmutz, Valentin},
 Title = {Long time behavior of an age- and leaky memory-structured neuronal population equation},
 FJournal = {SIAM Journal on Mathematical Analysis},
 Journal = {SIAM J. Math. Anal.},
 ISSN = {0036-1410},
 Volume = {54},
 Number = {4},
 Pages = {4721--4756},
 Year = {2022},
 Language = {English},
 DOI = {10.1137/21M1428571},
 Keywords = {35B40,35F15,35F20,35Q49,92B20},
 zbMATH = {7577529},
 Zbl = {1496.35075}
}

@incollection {MetzDiekmann_LN,
    AUTHOR = {Metz, J. A. J. and Diekmann, O.},
     TITLE = {A gentle introduction to structured population models: three
              worked examples},
 BOOKTITLE = {The dynamics of physiologically structured populations
              ({A}msterdam, 1983)},
    SERIES = {Lecture Notes in Biomath.},
    VOLUME = {68},
     PAGES = {3--45},
 PUBLISHER = {Springer, Berlin},
      YEAR = {1986},
   MRCLASS = {92A15},
  MRNUMBER = {860960},
       DOI = {10.1007/978-3-662-13159-6\_1},
       URL = {https://doi-org.accesdistant.sorbonne-universite.fr/10.1007/978-3-662-13159-6_1},
}
@article {GMC79,
    AUTHOR = {Gurtin, Morton E. and MacCamy, Richard C.},
     TITLE = {Some simple models for nonlinear age-dependent population
              dynamics},
   JOURNAL = {Math. Biosci.},
  FJOURNAL = {Mathematical Biosciences},
    VOLUME = {43},
      YEAR = {1979},
    NUMBER = {3-4},
     PAGES = {199--211},
      ISSN = {0025-5564},
   MRCLASS = {92A15},
  MRNUMBER = {527566},
MRREVIEWER = {K. E. Swick},
       DOI = {10.1016/0025-5564(79)90049-X},
       URL = {https://doi-org.accesdistant.sorbonne-universite.fr/10.1016/0025-5564(79)90049-X},
}


@incollection {MichelNL07,
    AUTHOR = {Michel, P.},
     TITLE = {General relative entropy in a nonlinear {M}c{K}endrick model},
 BOOKTITLE = {Stochastic analysis and partial differential equations},
    SERIES = {Contemp. Math.},
    VOLUME = {429},
     PAGES = {205--232},
 PUBLISHER = {Amer. Math. Soc., Providence, RI},
      YEAR = {2007},
   MRCLASS = {35F30 (35Q80 92D25)},
  MRNUMBER = {2391537},
       DOI = {10.1090/conm/429/08238},
}

@article {MMW2010,
    AUTHOR = {Magal, P. and McCluskey, C. C. and Webb, G. F.},
     TITLE = {Lyapunov functional and global asymptotic stability for an
              infection-age model},
   JOURNAL = {Appl. Anal.},
  FJOURNAL = {Applicable Analysis. An International Journal},
    VOLUME = {89},
      YEAR = {2010},
    NUMBER = {7},
     PAGES = {1109--1140},
      ISSN = {0003-6811},
   MRCLASS = {34K20 (34D23 35B25 35Q92 47D06 47N20 92D30)},
  MRNUMBER = {2674945},
MRREVIEWER = {Pauline van den Driessche},
       DOI = {10.1080/00036810903208122},
}

@book{GLP_book,
AUTHOR = {Galves, Antonio and L{\"o}cherbach, Eva and Pouzat, Christophe},
     TITLE = {Probabilistic Spiking Neuronal Nets},
    SERIES = {Lecture Notes on Mathematical Modelling in the Life Sciences},
 PUBLISHER = {Springer, Cham},
      YEAR = {2024},
     PAGES = {xv+199},
      ISBN = {978-3-031-68408-1},
       DOI = {10.1007/978-3-031-6},
}

@book {EvansBook,
    AUTHOR = {Evans, Lawrence C.},
     TITLE = {Partial differential equations},
    SERIES = {Graduate Studies in Mathematics},
    VOLUME = {19},
   EDITION = {Second},
 PUBLISHER = {American Mathematical Society, Providence, RI},
      YEAR = {2010},
     PAGES = {xxii+749},
      ISBN = {978-0-8218-4974-3},
   MRCLASS = {35-01},
  MRNUMBER = {2597943 (2011c:35002)},
MRREVIEWER = {Diego M. Maldonado},
}

@book {IannelliBook,
    AUTHOR = {Iannelli, Mimmo and Milner, Fabio},
     TITLE = {The basic approach to age-structured population dynamics},
    SERIES = {Lecture Notes on Mathematical Modelling in the Life Sciences},
      NOTE = {Models, methods and numerics},
 PUBLISHER = {Springer, Dordrecht},
      YEAR = {2017},
     PAGES = {xii+350},
      ISBN = {978-94-024-1145-4; 978-94-024-1146-1},
   MRCLASS = {92D25 (35Q92 65Mxx)},
  MRNUMBER = {3700352},
       DOI = {10.1007/978-94-024-1146-1},
       URL = {https://doi-org.accesdistant.sorbonne-universite.fr/10.1007/978-94-024-1146-1},
}

@article{SCHWALGER2019,
title = {Mind the last spike - firing rate models for mesoscopic populations of spiking neurons},
journal = {Current Opinion in Neurobiology},
volume = {58},
pages = {155-166},
year = {2019},
note = {Computational Neuroscience},
issn = {0959-4388},
doi = {https://doi.org/10.1016/j.conb.2019.08.003},
url = {https://www.sciencedirect.com/science/article/pii/S095943881930039X},
author = {Tilo Schwalger and Anton V Chizhov},
abstract = {The dominant modeling framework for understanding cortical computations are heuristic firing rate models. Despite their success, these models fall short to capture spike synchronization effects, to link to biophysical parameters and to describe finite-size fluctuations. In this opinion article, we propose that the refractory density method (RDM), also known as age-structured population dynamics or quasi-renewal theory, yields a powerful theoretical framework to build rate-based models for mesoscopic neural populations from realistic neuron dynamics at the microscopic level. We review recent advances achieved by the RDM to obtain efficient population density equations for networks of generalized integrate-and-fire (GIF) neurons â a class of neuron models that has been successfully fitted to various cell types. The theory not only predicts the nonstationary dynamics of large populations of neurons but also permits an extension to finite-size populations and a systematic reduction to low-dimensional rate dynamics. The new types of rate models will allow a re-examination of models of cortical computations under biological constraints.}
}

@article {MPS2019,
    AUTHOR = {Moussa, Ayman and Perthame, Beno\^{\i}t and Salort, Delphine},
     TITLE = {Backward parabolicity, cross-diffusion and {T}uring
              instability},
   JOURNAL = {J. Nonlinear Sci.},
  FJOURNAL = {Journal of Nonlinear Science},
    VOLUME = {29},
      YEAR = {2019},
    NUMBER = {1},
     PAGES = {139--162},
      ISSN = {0938-8974},
   MRCLASS = {35K57 (35B10 35B36 35K51 35Q92)},
  MRNUMBER = {3908864},
       DOI = {10.1007/s00332-018-9480-z},
       URL = {https://doi-org.accesdistant.sorbonne-universite.fr/10.1007/s00332-018-9480-z},
}

@article {alikakos79,
    AUTHOR = {Alikakos, N. D.},
     TITLE = {{$L^{p}$} bounds of solutions of reaction-diffusion
              equations},
   JOURNAL = {Comm. Partial Differential Equations},
  FJOURNAL = {Communications in Partial Differential Equations},
    VOLUME = {4},
      YEAR = {1979},
    NUMBER = {8},
     PAGES = {827--868},
      ISSN = {0360-5302},
   MRCLASS = {35K55 (80A20)},
  MRNUMBER = {537465},
MRREVIEWER = {Mieczyslaw Altman},
       DOI = {10.1080/03605307908820113},
       URL = {https://doi-org.accesdistant.sorbonne-universite.fr/10.1080/03605307908820113},
}

@misc{bernou2021,
      title={Hypocoercivity for kinetic linear equations in bounded domains with general Maxwell boundary condition}, 
      author={Armand Bernou and Kleber Carrapatoso and Stéphane Mischler and Isabelle Tristani},
      year={In press},
      VOLUME = {},
      NUMBER = {},
      JOURNAL={Annales de l'Institut Henri Poincar{\'e} Analyse Non Lin{\'e}aire},
      primaryClass={math.AP}
}

@article {Monmarche2014,
    AUTHOR = {Monmarch\'{e}, Pierre},
     TITLE = {Hypocoercive relaxation to equilibrium for some kinetic
              models},
   JOURNAL = {Kinet. Relat. Models},
  FJOURNAL = {Kinetic and Related Models},
    VOLUME = {7},
      YEAR = {2014},
    NUMBER = {2},
     PAGES = {341--360},
      ISSN = {1937-5093},
   MRCLASS = {60J99 (35B40 35R60 60H10)},
  MRNUMBER = {3195078},
       DOI = {10.3934/krm.2014.7.341},
       URL = {https://doi.org/10.3934/krm.2014.7.341},
}

@article {MR1787105,
    AUTHOR = {Desvillettes, L. and Villani, C.},
     TITLE = {On the trend to global equilibrium in spatially inhomogeneous
              entropy-dissipating systems: the linear {F}okker-{P}lanck
              equation},
   JOURNAL = {Comm. Pure Appl. Math.},
  FJOURNAL = {Communications on Pure and Applied Mathematics},
    VOLUME = {54},
      YEAR = {2001},
    NUMBER = {1},
     PAGES = {1--42},
      ISSN = {0010-3640},
   MRCLASS = {82C40 (35B40 47N55)},
  MRNUMBER = {1787105},
MRREVIEWER = {Carlo Cercignani},
       DOI = {10.1002/1097-0312(200101)54:1<1::AID-CPA1>3.0.CO;2-Q},
       URL =
              {https://doi.org/10.1002/1097-0312(200101)54:1<1::AID-CPA1>3.0.CO;2-Q},
}

@article {MR4071827,
    AUTHOR = {Dujardin, Guillaume and H\'{e}rau, Fr\'{e}d\'{e}ric and Lafitte, Pauline},
     TITLE = {Coercivity, hypocoercivity, exponential time decay and
              simulations for discrete {F}okker-{P}lanck equations},
   JOURNAL = {Numer. Math.},
  FJOURNAL = {Numerische Mathematik},
    VOLUME = {144},
      YEAR = {2020},
    NUMBER = {3},
     PAGES = {615--697},
      ISSN = {0029-599X},
   MRCLASS = {35Q84 (35B40 35Q83)},
  MRNUMBER = {4071827},
MRREVIEWER = {Luigi Nocera},
       DOI = {10.1007/s00211-019-01094-y},
       URL = {https://doi.org/10.1007/s00211-019-01094-y},
}

@article {BDMMS,
    AUTHOR = {Bouin, Emeric and Dolbeault, Jean and Mischler, St\'{e}phane and
              Mouhot, Cl\'{e}ment and Schmeiser, Christian},
     TITLE = {Hypocoercivity without confinement},
   JOURNAL = {Pure Appl. Anal.},
  FJOURNAL = {Pure and Applied Analysis},
    VOLUME = {2},
      YEAR = {2020},
    NUMBER = {2},
     PAGES = {203--232},
      ISSN = {2578-5885},
   MRCLASS = {82C40 (35H10 35K65 35P15 35Q84 76P05)},
  MRNUMBER = {4113786},
MRREVIEWER = {Pierre Monmarch\'{e}},
       DOI = {10.2140/paa.2020.2.203},
       URL = {https://doi.org/10.2140/paa.2020.2.203},
}
		

@article {MR3324910,
    AUTHOR = {Dolbeault, Jean and Mouhot, Cl\'{e}ment and Schmeiser, Christian},
     TITLE = {Hypocoercivity for linear kinetic equations conserving mass},
   JOURNAL = {Trans. Amer. Math. Soc.},
  FJOURNAL = {Transactions of the American Mathematical Society},
    VOLUME = {367},
      YEAR = {2015},
    NUMBER = {6},
     PAGES = {3807--3828},
      ISSN = {0002-9947},
   MRCLASS = {35F10 (35B40 35H10 82C31)},
  MRNUMBER = {3324910},
MRREVIEWER = {Ingrid Alma Belti\c{t}\u{a}},
       DOI = {10.1090/S0002-9947-2015-06012-7},
       URL = {https://doi.org/10.1090/S0002-9947-2015-06012-7},
}
		
@article {CH2019,
  Author = {Ca{\~n}izo, Jos{\'e} A. and Yolda{\c{s}}, Havva},
 Title = {Asymptotic behaviour of neuron population models structured by elapsed-time},
 FJournal = {Nonlinearity},
 Journal = {Nonlinearity},
 ISSN = {0951-7715},
 Volume = {32},
 Number = {2},
 Pages = {464--495},
 Year = {2019},
 Language = {English},
 DOI = {10.1088/1361-6544/aaea9c},
 Keywords = {35L04,35B10,92B20,35F30,35B35,35L60},
 URL = {hdl.handle.net/20.500.11824/908},
 zbMATH = {7010524},
 Zbl = {1406.35172}
}
     
@incollection {GaESAIM2018,
    AUTHOR = {Gabriel, Pierre},
     TITLE = {Measure solutions to the conservative renewal equation},
 BOOKTITLE = {C{IMPA} {S}chool on {M}athematical {M}odels in {B}iology and
              {M}edicine},
    SERIES = {ESAIM Proc. Surveys},
    VOLUME = {62},
     PAGES = {68--78},
 PUBLISHER = {EDP Sci., Les Ulis},
      YEAR = {2018},
   MRCLASS = {92D25 (35F31)},
  MRNUMBER = {3893207},
}

@unpublished{bansaye:hal-01617071,
  TITLE = {{Ergodic behavior of non-conservative semigroups via generalized Doeblin's conditions}},
  AUTHOR = {Bansaye, Vincent and Cloez, Bertrand and Gabriel, Pierre},
  URL = {https://hal.archives-ouvertes.fr/hal-01617071},
  NOTE = {working paper or preprint},
  YEAR = {2018},
  MONTH = Sep,
  KEYWORDS = {measure solutions ; Floquet theory ; population dynamics ; positive semigroups ; Krein-Rutman theorem ; ergodicity ; non-autonomous linear evolution equations},
  PDF = {https://hal.archives-ouvertes.fr/hal-01617071/file/Doeblin-BCG.pdf},
  HAL_ID = {hal-01617071},
  HAL_VERSION = {v3},
}


@article {DMS_2015,
    AUTHOR = {Dolbeault, Jean and Mouhot, Cl\'ement and Schmeiser, Christian},
     TITLE = {Hypocoercivity for linear kinetic equations conserving mass},
   JOURNAL = {Trans. Amer. Math. Soc.},
  FJOURNAL = {Transactions of the American Mathematical Society},
    VOLUME = {367},
      YEAR = {2015},
    NUMBER = {6},
     PAGES = {3807--3828},
      ISSN = {0002-9947},
   MRCLASS = {35F10 (35B40 35H10 82C31)},
  MRNUMBER = {3324910},
MRREVIEWER = {Ingrid Alma Belti\c t\u a},
       DOI = {10.1090/S0002-9947-2015-06012-7},
       URL = {https://doi.org/10.1090/S0002-9947-2015-06012-7},
}

@article {Mischler_ws_2018,
    AUTHOR = {Mischler, S. and Qui\~{n}inao, C. and Weng, Q.},
     TITLE = {Weak and strong connectivity regimes for a general time
              elapsed neuron network model},
   JOURNAL = {J. Stat. Phys.},
  FJOURNAL = {Journal of Statistical Physics},
    VOLUME = {173},
      YEAR = {2018},
    NUMBER = {1},
     PAGES = {77--98},
      ISSN = {0022-4715},
   MRCLASS = {92B20},
  MRNUMBER = {3859527},
       DOI = {10.1007/s10955-018-2122-x},
       URL = {https://doi-org.accesdistant.sorbonne-universite.fr/10.1007/s10955-018-2122-x},
}
		
@article {Mischler_TENN_2018,
    AUTHOR = {Mischler, S. and Weng, Q.},
     TITLE = {Relaxation in time elapsed neuron network models in the weak
              connectivity regime},
   JOURNAL = {Acta Appl. Math.},
  FJOURNAL = {Acta Applicandae Mathematicae},
    VOLUME = {157},
      YEAR = {2018},
     PAGES = {45--74},
      ISSN = {0167-8019},
   MRCLASS = {35F15 (35B40 35F20 92B20)},
  MRNUMBER = {3850019},
       DOI = {10.1007/s10440-018-0163-4},
       URL = {https://doi-org.accesdistant.sorbonne-universite.fr/10.1007/s10440-018-0163-4},
}

@ARTICLE{2017arXiv171005596D,
   author = {{Dumont}, G. and {Gabriel}, P.},
    title = "{The mean-field equation of a leaky integrate-and-fire neural network: measure solutions and steady states}",
  journal = {ArXiv e-prints},
archivePrefix = "arXiv",
   eprint = {1710.05596},
 primaryClass = "math.AP",
 keywords = {Mathematics - Analysis of PDEs},
     year = 2017,
    month = oct,
   adsurl = {http://adsabs.harvard.edu/abs/2017arXiv171005596D},
  adsnote = {Provided by the SAO/NASA Astrophysics Data System}
}

@article {Villani_hypocoercivity,
    AUTHOR = {Villani, C\'edric},
     TITLE = {Hypocoercivity},
   JOURNAL = {Mem. Amer. Math. Soc.},
  FJOURNAL = {Memoirs of the American Mathematical Society},
    VOLUME = {202},
      YEAR = {2009},
    NUMBER = {950},
     PAGES = {iv+141},
      ISSN = {0065-9266},
      ISBN = {978-0-8218-4498-4},
   MRCLASS = {35Q84 (35H10 76N10 76P05 82C70)},
  MRNUMBER = {2562709},
MRREVIEWER = {Andr\'as Domokos},
       DOI = {10.1090/S0065-9266-09-00567-5},
       URL = {https://doi.org/10.1090/S0065-9266-09-00567-5},
}
@article{rangan2006maximum,
  title={Maximum-entropy closures for kinetic theories of neuronal network dynamics},
  author={Rangan, Aaditya V and Cai, David},
  journal={Physical review letters},
  volume={96},
  number={17},
  pages={178101},
  year={2006},
  publisher={APS}
}

@article {Herda_Rodrigues,
    AUTHOR = {Herda, Maxime and Rodrigues, L. Miguel},
     TITLE = {Large-{T}ime {B}ehavior of {S}olutions to
              {V}lasov-{P}oisson-{F}okker-{P}lanck {E}quations: {F}rom
              {E}vanescent {C}ollisions to {D}iffusive {L}imit},
   JOURNAL = {J. Stat. Phys.},
  FJOURNAL = {Journal of Statistical Physics},
    VOLUME = {170},
      YEAR = {2018},
    NUMBER = {5},
     PAGES = {895--931},
      ISSN = {0022-4715},
   MRCLASS = {35Q83 (35B30 35B35 35B40 35H10 35Q84 82C31)},
  MRNUMBER = {3767000},
       DOI = {10.1007/s10955-018-1963-7},
       URL = {https://doi.org/10.1007/s10955-018-1963-7},
}

@article {Jabin_friction,
    AUTHOR = {Jabin, Pierre-Emmanuel},
     TITLE = {Macroscopic limit of {V}lasov type equations with friction},
   JOURNAL = {Ann. Inst. H. Poincar\'e Anal. Non Lin\'eaire},
  FJOURNAL = {Annales de l'Institut Henri Poincar\'e. Analyse Non Lin\'eaire},
    VOLUME = {17},
      YEAR = {2000},
    NUMBER = {5},
     PAGES = {651--672},
      ISSN = {0294-1449},
   MRCLASS = {82C40 (35B25 45K05)},
  MRNUMBER = {1791881},
       URL = {https://doi.org/10.1016/S0294-1449(00)00118-9},
}

@article {DV2005,
    AUTHOR = {Desvillettes, L. and Villani, C.},
     TITLE = {On the trend to global equilibrium for spatially inhomogeneous
              kinetic systems: the {B}oltzmann equation},
   JOURNAL = {Invent. Math.},
  FJOURNAL = {Inventiones Mathematicae},
    VOLUME = {159},
      YEAR = {2005},
    NUMBER = {2},
     PAGES = {245--316},
      ISSN = {0020-9910},
   MRCLASS = {82C40 (35B40 35F20)},
  MRNUMBER = {2116276},
MRREVIEWER = {Manuel Portilheiro},
       URL = {https://doi.org/10.1007/s00222-004-0389-9},
}

@book {RV_lecture,
    AUTHOR = {Rezakhanlou, Fraydoun and Villani, C\'edric},
     TITLE = {Entropy methods for the {B}oltzmann equation},
    SERIES = {Lecture Notes in Mathematics},
    VOLUME = {1916},
      NOTE = {Lectures from a Special Semester on Hydrodynamic Limits held
              at the Universit\'e de Paris VI, Paris, 2001,
              Edited by Fran\c{c}ois Golse and Stefano Olla},
 PUBLISHER = {Springer, Berlin},
      YEAR = {2008},
     PAGES = {xii+107},
      ISBN = {978-3-540-73704-9},
   MRCLASS = {82C40 (82-02 82-06)},
  MRNUMBER = {2407976},
MRREVIEWER = {Cl\'ement Mouhot},
       URL = {https://doi.org/10.1007/978-3-540-73705-6},
}

@book{Bouchut,
	alteditor = "Birkha{\"u}ser-Verlag",
	author = "Bouchut, F.",
	optaddress = "",
	optedition = "",
	optkey = "",
	optmonth = "",
	optnote = "",
	optnumber = "",
	optseries = "",
	optvolume = "",
	publisher = "Birkha{\"u}ser-Verlag",
	title = "{Non linear stability of finite volume methods for hyperbolic conservation laws and well balanced schemes for sources}",
	year = "2004"
}

@book{Leveque,
	alteditor = "Cambridge university press",
	author = "LeVeque, R. J.",
	optaddress = "",
	optannote = "",
	optedition = "",
	optkey = "",
	optmonth = "",
	optnote = "",
	optnumber = "",
	optseries = "",
	optvolume = "",
	publisher = "Cambridge University Press",
	title = "{Finite volume methods for hyperbolic problems}",
	year = "2002"
}

@incollection{Jungel,
  author    = {J\"ungel, A.},
  title     = {Transport Equations for Semiconductors},
  publisher = {Springer-Verlag},
  booktitle = {Lecture Notes in Physics },
  volume={773},
  year      = {2009},
  address   = {Berlin Heidelberg},
  DOI = {10.1007/978-3-540-89526-8},
}

@incollection{Dafermos,
  author    = {Dafermos, C. },
  title     = {Hyperbolic conservation laws in continuum physics},
  publisher = {Springer-Verlag},
  booktitle = {Grundlehren der Mathematischen Wissenschaften},
  volume={325},
  year      = {2000},
  address   = {Berlin}
}

@book {SerreBook,
       AUTHOR = {Serre, Denis},
       TITLE = {Systems of conservation laws, I},
       PUBLISHER = {Cambridge University Press}, 
       YEAR = {1999},
       ISBN = {0 521 58233 4},
}       
     
@book {DautrayLions6,
    AUTHOR = {Dautray, Robert and Lions, Jacques-Louis},
     TITLE = {Mathematical analysis and numerical methods for science and
              technology. {V}ol. 6},
      NOTE = {Evolution problems. II,
              With the collaboration of Claude Bardos, Michel Cessenat,
              Alain Kavenoky, Patrick Lascaux, Bertrand Mercier, Olivier
              Pironneau, Bruno Scheurer and R{\'e}mi Sentis,
              Translated from the French by Alan Craig},
 PUBLISHER = {Springer-Verlag, Berlin},
      YEAR = {1993},
     PAGES = {xii+485},
      ISBN = {3-540-50206-8; 3-540-66102-6},
   MRCLASS = {00A05 (35-01 47-01 82-01)},
  MRNUMBER = {1295030 (95e:00001)},
       DOI = {10.1007/978-3-642-58004-8},
       URL = {http://dx.doi.org/10.1007/978-3-642-58004-8},
}
		

@book {DautrayLions5,
    AUTHOR = {Dautray, Robert and Lions, Jacques-Louis},
     TITLE = {Mathematical analysis and numerical methods for science and
              technology. {V}ol. 5},
      NOTE = {Evolution problems. I,
              With the collaboration of Michel Artola, Michel Cessenat and
              H{\'e}l{\`e}ne Lanchon,
              Translated from the French by Alan Craig},
 PUBLISHER = {Springer-Verlag, Berlin},
      YEAR = {1992},
     PAGES = {xiv+709},
      ISBN = {3-540-50205-X; 3-540-66101-8},
   MRCLASS = {00A05 (35-01 47-01)},
  MRNUMBER = {1156075 (92k:00006)},
       DOI = {10.1007/978-3-642-58090-1},
       URL = {http://dx.doi.org/10.1007/978-3-642-58090-1},
}	


@misc{CaceresCanizoT2025,
      title={On the local stability of the elapsed-time model in terms of the transmission delay and interconnection strength}, 
      author={María J. Cáceres and José A Cañizo and Nicolas Torres},
      year={2025},
      eprint={2504.17358},
      archivePrefix={arXiv},
      primaryClass={math.DS},
      url={https://arxiv.org/abs/2504.17358}, 
}

@misc{CCTdelay,
      title={Comparison principles and asymptotic behavior of delayed age-structured neuron models}, 
      author={María J Cáceres and José A Cañizo and Nicolas Torres},
      year={2025},
      note={ar{X}iv/2502.13553},
      eprint={2502.13553},
      archivePrefix={arXiv},
      primaryClass={math.AP},
      url={https://arxiv.org/abs/2502.13553}, 
}

@misc{CaceresCanizo2024,
      title={Sequence of pseudoequilibria describes the long-time behavior of the nonlinear noisy leaky integrate-and-fire model with large delay}, 
      author={María J. Cáceres and José A. Cañizo and Alejandro Ramos-Lora},
      year={2024},
      note={ar{X}iv/2403.00971},
      eprint={2403.00971},
      archivePrefix={arXiv},
      primaryClass={math.AP},
      url={https://arxiv.org/abs/2403.00971}, 
}

@article {CaPe,
    AUTHOR = {C{\'a}ceres, Mar{\'{\i}}a J. and Perthame, Beno{\^{\i}}t},
     TITLE = {Beyond blow-up in excitatory integrate and fire neuronal
              networks: refractory period and spontaneous activity},
   JOURNAL = {J. Theoret. Biol.},
  FJOURNAL = {Journal of Theoretical Biology},
    VOLUME = {350},
      YEAR = {2014},
     PAGES = {81--89},
      ISSN = {0022-5193},
   MRCLASS = {92C20 (92B20)},
  MRNUMBER = {3190511},
       DOI = {10.1016/j.jtbi.2014.02.005},
       URL = {http://dx.doi.org/10.1016/j.jtbi.2014.02.005},
}
	
@article {CPSS,
    AUTHOR = {Carrillo, Jos{\'e} Antonio and Perthame, Beno{\^{\i}}t and
              Salort, Delphine and Smets, Didier},
     TITLE = {Qualitative properties of solutions for the noisy integrate
              and fire model in computational neuroscience},
   JOURNAL = {Nonlinearity},
  FJOURNAL = {Nonlinearity},
    VOLUME = {28},
      YEAR = {2015},
    NUMBER = {9},
     PAGES = {3365--3388},
      ISSN = {0951-7715},
   MRCLASS = {35Q84 (35B44 92C20)},
  MRNUMBER = {3403402},
       DOI = {10.1088/0951-7715/28/9/3365},
       URL = {http://dx.doi.org/10.1088/0951-7715/28/9/3365},
}

@article {FoPeSIAM,
    AUTHOR = {Fournier, Nicolas and Perthame, Beno\^{\i}t},
     TITLE = {A nonexpanding transport distance for some structured
              equations},
   JOURNAL = {SIAM J. Math. Anal.},
  FJOURNAL = {SIAM Journal on Mathematical Analysis},
    VOLUME = {53},
      YEAR = {2021},
    NUMBER = {6},
     PAGES = {6847--6872},
      ISSN = {0036-1410},
   MRCLASS = {35K55 (28A33 60J99)},
  MRNUMBER = {4347325},
       DOI = {10.1137/21M1397313},
       URL = {https://doi-org.accesdistant.sorbonne-universite.fr/10.1137/21M1397313},
}

@article {KPS2021,
    AUTHOR = {Kim, Jeongho and Perthame, Beno\^{\i}t and Salort, Delphine},
     TITLE = {Fast voltage dynamics of voltage-conductance models for neural
              networks},
   JOURNAL = {Bull. Braz. Math. Soc. (N.S.)},
  FJOURNAL = {Bulletin of the Brazilian Mathematical Society. New Series.
              Boletim da Sociedade Brasileira de Matem\'{a}tica},
    VOLUME = {52},
      YEAR = {2021},
    NUMBER = {1},
     PAGES = {101--134},
      ISSN = {1678-7544},
   MRCLASS = {35Q92 (35Q84 92B20)},
  MRNUMBER = {4212118},
       DOI = {10.1007/s00574-019-00192-7},
       URL = {https://doi-org.accesdistant.sorbonne-universite.fr/10.1007/s00574-019-00192-7},
}
		
@article {PS2019,
    AUTHOR = {Perthame, Beno\^{\i}t and Salort, Delphine},
     TITLE = {Derivation of a voltage density equation from a
              voltage-conductance kinetic model for networks of
              integrate-and-fire neurons},
   JOURNAL = {Commun. Math. Sci.},
  FJOURNAL = {Communications in Mathematical Sciences},
    VOLUME = {17},
      YEAR = {2019},
    NUMBER = {5},
     PAGES = {1193--1211},
      ISSN = {1539-6746},
   MRCLASS = {35Q84 (35B65 62M45 92B20 92C20)},
  MRNUMBER = {4044187},
MRREVIEWER = {Alessandro Goffi},
       DOI = {10.4310/CMS.2019.v17.n5.a2},
       URL = {https://doi-org.accesdistant.sorbonne-universite.fr/10.4310/CMS.2019.v17.n5.a2},
}

@article{caceres:hal-03172021,
  TITLE = {{An elapsed time model for strongly coupled inhibitory and excitatory neural networks}},
  AUTHOR = {Caceres, Maria J. and Perthame, Beno{\^i}t and Salort, Delphine and Torres, Nicolas},
  URL = {https://hal.science/hal-03172021},
  JOURNAL = {{Physica D: Nonlinear Phenomena}},
  PUBLISHER = {{Elsevier}},
  YEAR = {2021},
  MONTH = Jun,
  DOI = {10.1016/j.physd.2021.132977},
  KEYWORDS = {Neural networks ; Mathematical neuroscience ; Periodic solutions ; Delay differential equations ; Structured equations},
  PDF = {https://hal.science/hal-03172021v1/file/Article_Convergence.pdf},
  HAL_ID = {hal-03172021},
  HAL_VERSION = {v1},
}

@article {TPS2022,
    AUTHOR = {Torres, Nicol\'{a}s and Perthame, Beno\^{i}t and Salort, Delphine},
     TITLE = {A multiple time renewal equation for neural assemblies with
              elapsed time model},
   JOURNAL = {Nonlinearity},
  FJOURNAL = {Nonlinearity},
    VOLUME = {35},
      YEAR = {2022},
    NUMBER = {10},
     PAGES = {5051--5075},
      ISSN = {0951-7715},
   MRCLASS = {35B40 (35F20 35R09 92B20)},
  MRNUMBER = {4500858},
       DOI = {10.1088/1361-6544/ac8714},
       URL = {https://doi-org.accesdistant.sorbonne-universite.fr/10.1088/1361-6544/ac8714},
}
@Article{GrayAndSinger:89,
 Author = "Gray, C. M. and Singer, W.",
  Title = "{Stimulus-specific neuronal oscillations in orientation columns of cat visual cortex}",
Journal = "Proc Natl Acad Sci U S A",
   Year = "1989",
 Volume = "86",
 Number = "5",
  Pages = "1698-1702"}
 } 
		

@article {Henry:12,
    AUTHOR = {Dumont, Gr\'egory and Henry, Jacques},
     TITLE = {Population density models of integrate-and-fire neurons with
              jumps: well-posedness},
   JOURNAL = {J. Math. Biol.},
  FJOURNAL = {Journal of Mathematical Biology},
    VOLUME = {67},
      YEAR = {2013},
    NUMBER = {3},
     PAGES = {453--481},
      ISSN = {0303-6812},
   MRCLASS = {92B25 (35B30 35F30 82C32)},
  MRNUMBER = {3084360},
MRREVIEWER = {Irina Ioana Mohorianu},
       URL = {https://doi.org/10.1007/s00285-012-0554-5},
}
		

@article {Henry:13,
    AUTHOR = {Dumont, Gr\'egory and Henry, Jacques},
     TITLE = {Synchronization of an excitatory integrate-and-fire neural
              network},
   JOURNAL = {Bull. Math. Biol.},
  FJOURNAL = {Bulletin of Mathematical Biology},
    VOLUME = {75},
      YEAR = {2013},
    NUMBER = {4},
     PAGES = {629--648},
      ISSN = {0092-8240},
   MRCLASS = {92C20},
  MRNUMBER = {3039916},
       URL = {https://doi.org/10.1007/s11538-013-9823-8},
       abstract={In this paper, we study the influence of the coupling strength on the synchronization behavior of a population of leaky integrate-and-fire neurons that is self-excitatory with a population density approach. Each neuron of the population is assumed to be stochastically driven by an independent Poisson spike train and the synaptic interaction between neurons is modeled by a potential jump at the reception of an action potential. Neglecting the synaptic delay, we will establish that for a strong enough connectivity between neurons, the solution of the partial differential equation which describes the population density function must blow up in finite time. Furthermore, we will give a mathematical estimate on the average connection per neuron to ensure the occurrence of a burst. Interpreting the blow up of the solution as the presence of a Dirac mass in the firing rate of the population, we will relate the blow up of the solution to the occurrence of the synchronization of neurons. Fully stochastic simulations of a finite size network of leaky integrate-and-fire neurons are performed to illustrate our theoretical results.},
}
		 

@article{AD09,
author = "Albantakis, L. and Deco, G.",
title = {The encoding of alternatives in multiple-choice decision making},
abstract = {During the last decades, research on binary decision making elucidated some of the basic neural mechanisms underlying the decision-making process. Recently, the focus of experimental as well as modeling studies began to shift from simple binary choices to decision making with multiple alternatives. In this article, we address the question how different numbers of choice alternatives might be handled and encoded in the brain. We present a minimal, biophysically realistic spiking neuron model for decision making with multiple alternatives. Our model accounts for the relevant aspects of recent experimental data of a random-dot motion-discrimination task on both the cellular and behavioral level. Notably, all network parameters and inputs in our network are independent of the number of possible alternatives used in the tested experimental paradigms (2 and 4 alternatives and 2 alternatives with an angular separation of 90 degrees ). This avoids the use of extra top-down regulation mechanisms to adapt the network to the number of choices. Furthermore, we show that increasing the number of neurons encoding each choice alternative is positively related to the network's capacity of choice-number-independent decision making. Consequently, our results suggest a physiological advantage of a pooled, multineuron representation of choice alternatives.},
journal = "Proc Natl Acad Sci U S A",
year = "2009",
volume = "106",
number = "25",
pages = "10308-10313",
month = "Jun",
pmid = "19497888",
url = "http://www.hubmed.org/display.cgi?uids=19497888",
doi = "10.1073/pnas.0901621106"
}

@article{PTW,
author = {Pakdaman, K. and Thieullen, M. and Wainrib, G.},
 title =   {Fluid limit theorems for stochastic hybrid systems with application to neuron models},
 journal = {Adv. in Appl. Probab.},
 volume = {42},
 number={3},
 pages = {761--794},
 year =  { 2010}
}



@article{MRR07,
author = {Moreno-Bote, R. and Rinzel, J. and Rubin, N.},
 title =   {Noise-induced alternations in an attractor network model of perceptual bistability},
 journal = {J.\ Neurophysiol.},
 volume =  98,
 pages =   "1125-1139",
 year =    2007
}

@article{sirovich,
  author  = {Sirovich, L. and Omurtag, A. and Lubliner K.},
  title   = {Dynamics of neural populations: Stability and synchrony},
  journal = {Network: Computation in Neural Systems},
  year  = {2006},
    month = {},
    volume= {17},
    number= {},
    pages = {3--29},
    note  = {}
}

@article{omurtag,
  author  = {Omurtag, A. and Knight B. W. and Sirovich, L.},
  title   = {On the Simulation of Large Populations of Neurons},
  journal = {J. Comp. Neurosci.},
  year  = {2000},
    month = {},
    volume= {8},
    number= {},
    pages = {51--63},
    note  = {}
}

@article{brunel,
  author  = {Brunel, N.},
  title   = {Dynamics of sparsely connected networks of excitatory
and inhibitory spiking networks},
  journal = {J. Comp. Neurosci.},
  year  = {2000},
    month = {},
    volume= {8},
    number= {},
    pages = {183--208},
    note  = {}
}

@article{ForucaudBrunel,
  author  = {Fourcaud, N. and Brunel, N.},
  title   = {Dynamics of the firing probability of noisy integrate-and-fire neurons,},
  journal = {Neural Comp.},
  year  = {2002},
    month = {},
    volume= {14},
    number= {},
    pages = {2057--2110},
    note  = {}
}


@article{BDP,
  author  = {Blanchet, A. and  Dolbeault, J. and Perthame, B.},
  title   = {{K}eller-{S}egel model: optimal critical mass and qualitative
properties of the solutions},
  journal = {Electron. J. Differential Equations},
  year  = {2006},
    month = {},
    volume= {44},
    number= {},
    pages = {32 (electronic)},
    note  = {}
}
@article{BrGe,
  author  = {Brette, R. and Gerstner, W.},
  title   = {Adaptive exponential integrate-and-fire model as an effective description of neural activity},
  journal = {Journal of neurophysiology},
  year  = {2005},
    month = {},
    volume= {94},
    number= {},
    pages = {3637--3642},
    note  = {}
}




@article{RouxCar,
author = {Carrillo, Jos\'{e} A. and Roux, Pierre},
title = {Nonlinear partial differential equations in neuroscience: From modeling to mathematical theory},
journal = {Mathematical Models and Methods in Applied Sciences},
volume = {35},
number = {02},
pages = {403-584},
year = {2025},
doi = {10.1142/S0218202525400044},

URL = { 
    
        https://doi.org/10.1142/S0218202525400044
    
    

},
eprint = { 
    
        https://doi.org/10.1142/S0218202525400044
    
    

}
,
    abstract = { Many systems of partial differential equations (PDEs) have been proposed as simplified representations of complex collective behaviors in large networks of neurons. In this survey, we briefly discuss their derivations and then review the mathematical methods developed to handle the unique features of these models, which are often nonlinear and nonlocal. The first part focuses on parabolic Fokker–Planck equations: the Nonlinear Noisy Leaky Integrate and Fire neuron model, stochastic neural fields in PDE form with applications to grid cells and rate-based models for decision-making. The second part concerns hyperbolic transport equations, namely, the model of the Time Elapsed since the last discharge and the jump-based Leaky Integrate and Fire model. The last part covers some kinetic mesoscopic models, with particular attention to the kinetic Voltage-Conductance model and FitzHugh–Nagumo kinetic Fokker–Planck systems. }
}



@article{BrHa,
  author  = {Brunel, N. and Hakim, V.},
  title   = {Fast global oscillations in networks of
integrate-and-fire neurons with long firing rates},
  journal = {Neural Computation},
  year  = {1999},
    month = {},
    volume= {11},
    number= {},
    pages = {1621--1671},
    note  = {}
}


@article{BCD,
  author  = {Buet, C. and Cordier, S. and Dos Santos, V.},
  title   = {A conservative and entropy
scheme for a simplified model of granular media},
  journal = {Transp. Theory Statist. Phys.},
  year  = {2004},
    month = {},
    volume= {33},
    number= {},
    pages = {125--155},
    note  = {}
}
@article{CCTao,
  author  = {C\'aceres, M. J. and  Carrillo,  J. A.  and  Tao, L.},
  title   = {A numerical solver for a nonlinear {F}okker-{P}lanck equation
representation of neuronal network dynamics},
  journal = {J. Comp. Phys.},
  year  = {2011},
    month = {},
    volume= {230},
    number= {},
    pages = {1084-1099},
    note  = {}
}

@article{CBGW,
  author  = {Compte, A. and
Brunel, N. and  Goldman-Rakic,  P. S.  and  Wang, X.-J.},
  title   = {Synaptic mechanisms and network dynamics underlying spatial
working memory in a cortical network model},
  journal = {Cerebral Cortex 10},
  year  = {2000},
    month = {},
    volume= {10},
    number= {},
    pages = {910--923},
    note  = {}
}
@article{CPZ,
  author  = {Corrias, L. and  Perthame, B.  and  Zaag, H.},
  title   = {Global solutions of some
chemotaxis and angiogenesis systems in high space dimensions},
  journal = {Milan J. Math.},
  year  = {2004},
    month = {},
    volume= {72},
    number= {},
    pages = {1--28},
    note  = {}
}


@article{GG09,
  author  = {Gonz\'alez, M. d. M. and  Gualdani, M. P.},
  title   = {Asymptotics for a symmetric equation in price formation},
  journal = {App. Math. Optim.},
  year  = {2009},
    month = {},
    volume= {59},
    number= {},
    pages = {233--246},
    note  = {}
}

@article{G,
  author  = {Guillamon, T.},
  title   = {An introduction to the mathematics of
neural activity},
  journal = {Butl. Soc. Catalana Mat.},
  year  = {2004},
    month = {},
    volume= {19},
    number= {},
    pages = {25--45},
    note  = {}
}

@article{lapicque,
  author  = {Lapicque, L.},
  title   = {Recherches quantitatives sur l'excitation
\'electrique des nerfs trait\'ee comme une polarisation},
  journal = { J. Physiol. Pathol. Gen},
  year  = {1907},
    month = {},
    volume= {9},
    number= {},
    pages = {620--635},
    note  = {}
}

@incollection{ledoux,
  author    = {Ledoux, M.},
  title     = {The concentration of measure phenomenon},
  booktitle = {AMS math. surveys and monographs},
  editor    = {},
  publisher = {AMS},
  pages     = {},
  year      = {2001},
    address = {},
    volume  = {89},
    edition = {}
}

@article{mg,
  author  = {Mattia, M. and Del Giudice, P.},
  title   = {Population dynamics of interacting spiking neurons},
  journal = {Phys. Rev. E},
  year  = {2002},
    month = {},
    volume= {66},
    number= {},
    pages = {051917},
    note  = {}
}

@article{RefMMP,
  author  = {Michel, P. and  Mischler, S. and  Perthame, B.},
  title   = {General relative entropy inequality: an illustration
on growth models},
  journal = {J. Math. Pures Appl.},
  year  = {2005},
    month = {},
    volume= {84},
    number= {},
    pages = {1235--1260},
    note  = {}
}

@article {PPD3,
    AUTHOR = {Pakdaman, Khashayar and Perthame, Beno{\^\i}t and Salort, Delphine},
     TITLE = {Adaptation and fatigue model for neuron networks and large
              time asymptotics in a nonlinear fragmentation equation},
   JOURNAL = {J. Math. Neurosci.},
  FJOURNAL = {Journal of Mathematical Neuroscience},
    VOLUME = {4},
      YEAR = {2014},
     PAGES = {Art. 14, 26},
      ISSN = {2190-8567},
   MRCLASS = {82C21 (35B40 35F31 35R09 92B20)},
  MRNUMBER = {3246939},
       URL = {https://doi.org/10.1186/2190-8567-4-14},
}

@article {PPD2,
    AUTHOR = {Pakdaman, Khashayar and Perthame, Beno{\^\i}t and Salort, Delphine},
     TITLE = {Relaxation and self-sustained oscillations in the time elapsed
              neuron network model},
   JOURNAL = {SIAM J. Appl. Math.},
  FJOURNAL = {SIAM Journal on Applied Mathematics},
    VOLUME = {73},
      YEAR = {2013},
    NUMBER = {3},
     PAGES = {1260--1279},
      ISSN = {0036-1399},
   MRCLASS = {82C32 (92C20)},
  MRNUMBER = {3071416},
       URL = {https://doi.org/10.1137/110847962},
}
		
		
@article{GW,
author = {Mathieu N. Galtier and Gilles Wainrib},
title = {A Biological Gradient Descent for Prediction Through a Combination of STDP and Homeostatic Plasticity},
journal = {Neural Computation},
volume = {25},
number = {11},
pages = {2815-2832},
year = {2013},
}		
		

@Article{Wilson1973,
author="Wilson, H. R.
and Cowan, J. D.",
title="A mathematical theory of the functional dynamics of cortical and thalamic nervous tissue",
journal="Kybernetik",
year="1973",
month="Sep",
day="01",
volume="13",
number="2",
pages="55--80",
abstract="It is proposed that distinct anatomical regions of cerebral cortex and of thalamic nuclei are functionally two-dimensional. On this view, the third (radial) dimension of cortical and thalamic structures is associated with a redundancy of circuits and functions so that reliable signal processing obtains in the presence of noisy or ambiguous stimuli.",
issn="1432-0770",
doi="10.1007/BF00288786",
url="https://doi.org/10.1007/BF00288786"
}

@article{CCDR,
 Author = {Chevallier, Julien and C{\'a}ceres, Mar{\'{\i}}a Jos{\'e} and Doumic, Marie and Reynaud-Bouret, Patricia},
 Title = {Microscopic approach of a time elapsed neural model},
 FJournal = {M\(^3\)AS. Mathematical Models \& Methods in Applied Sciences},
 Journal = {Math. Models Methods Appl. Sci.},
 ISSN = {0218-2025},
 Volume = {25},
 Number = {14},
 Pages = {2669--2719},
 Year = {2015},
 Language = {English},
 DOI = {10.1142/S021820251550058X},
 Keywords = {35Q90,92B20,35F15,35B10,60G57,60K15},
 zbMATH = {6506995},
 Zbl = {1325.35231}
}

@article {PPD,
    AUTHOR = {Pakdaman, Khashayar and Perthame, Beno{\^\i}t and Salort, Delphine},
     TITLE = {Dynamics of a structured neuron population},
   JOURNAL = {Nonlinearity},
  FJOURNAL = {Nonlinearity},
    VOLUME = {23},
      YEAR = {2010},
    NUMBER = {1},
     PAGES = {55--75},
      ISSN = {0951-7715},
   MRCLASS = {35Q92 (35B35 92C20)},
  MRNUMBER = {2576373},
       URL = {https://doi.org/10.1088/0951-7715/23/1/003},
}

@article {PeSa,
    AUTHOR = {Perthame, Beno{\^\i}t and Salort, Delphine},
     TITLE = {On a voltage-conductance kinetic system for integrate \& fire
              neural networks},
   JOURNAL = {Kinet. Relat. Models},
  FJOURNAL = {Kinetic and Related Models},
    VOLUME = {6},
      YEAR = {2013},
    NUMBER = {4},
     PAGES = {841--864},
      ISSN = {1937-5093},
   MRCLASS = {35Q84 (35B65 35Q83 62M45 82C32 92C20)},
  MRNUMBER = {3177631},
MRREVIEWER = {Renjun Duan},
       DOI = {10.3934/krm.2013.6.841},
       URL = {https://doi.org/10.3934/krm.2013.6.841},
}		


@article {PSW,
    AUTHOR = {Perthame, Beno{\^\i}t and Salort, Delphine and Wainrib, Gilles},
     TITLE = {Distributed synaptic weights in a {LIF} neural network and
              learning rules},
   JOURNAL = {Phys. D},
  FJOURNAL = {Physica D. Nonlinear Phenomena},
    VOLUME = {353/354},
      YEAR = {2017},
     PAGES = {20--30},
      ISSN = {0167-2789},
   MRCLASS = {92B20 (92C20)},
  MRNUMBER = {3669355},
       DOI = {10.1016/j.physd.2017.05.005},
       URL = {https://doi.org/10.1016/j.physd.2017.05.005},
}
		

@article {PSK,
    AUTHOR = {Kang, Moon-Jin and Perthame, Beno\^\i t and Salort, Delphine},
     TITLE = {Dynamics of time elapsed inhomogeneous neuron network model},
   JOURNAL = {C. R. Math. Acad. Sci. Paris},
  FJOURNAL = {Comptes Rendus Math\'ematique. Acad\'emie des Sciences. Paris},
    VOLUME = {353},
      YEAR = {2015},
    NUMBER = {12},
     PAGES = {1111--1115},
      ISSN = {1631-073X},
   MRCLASS = {35F31 (35B40)},
  MRNUMBER = {3427917},
       DOI = {10.1016/j.crma.2015.09.029},
       URL = {https://doi.org/10.1016/j.crma.2015.09.029},
}
					
@book{perthame2015parabolic,
  AUTHOR = {Perthame, Beno\^{\i}t},
     TITLE = {Parabolic equations in biology},
    SERIES = {Lecture Notes on Mathematical Modelling in the Life Sciences},
 PUBLISHER = {Springer, Cham},
      YEAR = {2015},
     PAGES = {xii+199},
      ISBN = {978-3-319-19499-8; 978-3-319-19500-1},
   MRCLASS = {92-01 (35K40 35Q92 82C31 92Cxx 92Dxx 92Exx)},
  MRNUMBER = {3408563},
       DOI = {10.1007/978-3-319-19500-1},
       URL = {https://doi-org.accesdistant.sorbonne-universite.fr/10.1007/978-3-319-19500-1},
}

@book{BP07,
title = {Transport equations in biology},
publisher = {Birkh\"auser Verlag},
year = {2007},
author = {Perthame, Beno{\^\i t}},
pages = {x+198},
series = {Frontiers in Mathematics},
address = {Basel},
isbn = {978-3-7643-7841-7; 3-7643-7841-7},
mrclass = {35-02 (35Q80 92C17 92C50 92D25)},
mrnumber = {MR2270822 (2007j:35004)},
mrreviewer = {Reinhard Illner}
}

@article{PPCV,
  author  = {Pham, J. and Pakdaman, K. and Champagnat, J.
and  Vivert, J.-F.},
  title   = {Activity in sparsely
connected excitatory neural networks: effect of connectivity},
  journal = {Neural Networks},
  year  = {1998},
    month = {},
    volume= {11},
    number= {},
    pages = {415--434},
    note  = {}
}

@incollection{RBW,
  author    = {A. Renart and N. Brunel and  X.-J. Wang},
  title     = {Mean-Field Theory of
Irregularly Spiking Neuronal Populations and Working Memory in
Recurrent Cortical Networks},
  booktitle = {Computational
Neuroscience: A comprehensive approach},
  editor    = {Jianfeng Feng},
  publisher = {Chapman \& Hall/CRC Mathematical Biology and Medicine Series},
  pages     = {},
  year      = {2004},
    address = {},
    volume  = {},
    edition = {}
}
@incollection{S,
  author  = { C.-W. Shu},
  title   = {Essentially Non-Oscillatory and Weighted Esentially
Non-Oscillatory
schemes for hyperbolic conservation laws},
  booktitle = {Advanced Numerical
Approximation of Nonlinear Hyperbolic Equations, B. Cockburn, C.
Johnson, C.-W. Shu and E. Tadmor},
  editor    = {A. Quarteroni},
  publisher = {Springer},
  pages     = {325-432},
  year      = {1998},
    address = {},
    volume  = {1697},
    edition = {}
}

@article{FLtoy,
author = {Nicolas Fournier and Eva L{\"o}cherbach},
title = {{On a toy model of interacting neurons}},
volume = {52},
journal = {Annales de l'Institut Henri Poincar\'e, Probabilit\'es et Statistiques},
number = {4},
publisher = {Institut Henri Poincar\'e},
pages = {1844 -- 1876},
keywords = {Biological neural nets, interacting particle systems, Mean-field interaction, Nonlinear stochastic differential equations, Piecewise deterministic Markov processes},
year = {2016},
doi = {10.1214/15-AIHP701},
URL = {https://doi.org/10.1214/15-AIHP701}
}


@article {Touboul_2008,
  AUTHOR = {Touboul, Jonathan},
   TITLE = {Bifurcation analysis of a general class of nonlinear
            integrate-and-fire neurons},
 JOURNAL = {SIAM J. Appl. Math.},
FJOURNAL = {SIAM Journal on Applied Mathematics},
  VOLUME = {68},
    YEAR = {2008},
  NUMBER = {4},
   PAGES = {1045--1079},
    ISSN = {0036-1399},
 MRCLASS = {34C23 (37G10 37N25 92C20)},
MRNUMBER = {2390979 (2009c:34083)},
MRREVIEWER = {Peter Moson},
     DOI = {10.1137/070687268},
     URL = {http://dx.doi.org/10.1137/070687268},
}


@article{Touboul_AQIF,
  author  = {Touboul, J.},
  title   = {Importance of the cutoff value in the
quadratic adaptive integrate-and-fire model},
  journal = {Neural Computation},
  year  = {2009},
    month = {},
    volume= {21},
    number= {},
    pages = {2114-2122},
    note  = {}
}

@book{T,
  author    = {H.C. Tuckwell},
  title     = {Introduction to Theoretical Neurobiology},
  publisher = {Cambridge Univ. Press},
  year=   {1988},
  address = {Cambridge}
}

@article{CCP,
  AUTHOR = {C\'aceres, Mar{\'\i}a J. and Carrillo, Jos\'e A. and Perthame, Beno{\^\i}t},
     TITLE = {Analysis of nonlinear noisy integrate \& fire neuron models:
              blow-up and steady states},
   JOURNAL = {J. Math. Neurosci.},
  FJOURNAL = {Journal of Mathematical Neuroscience},
    VOLUME = {1},
      YEAR = {2011},
     PAGES = {Art. 7, 33},
      ISSN = {2190-8567},
   MRCLASS = {35K10 (35B44 35R60 92C20)},
  MRNUMBER = {2853216},
       URL = {https://doi.org/10.1186/2190-8567-1-7},
}


@article{LyT,
  AUTHOR = { Ly, C. and Tranchina, D.},
  TITLE = {Critical analysis of dimension reduction by a moment closure method in a population density approach to neural network modeling},
JOURNAL = {Neural Computation},
VOLUME = {19},
 YEAR = {2007}, 
 PAGES = {2032--2092},
} 


@article {CTSL,
   AUTHOR = {Cai, D. and Tao, L.  and Shelley,  M. and McLaughlin D. W. },
   TITLE = {An effective kinetic representation of fluctuation-driven neuronal networks with application to simple and complex cells in visual cortex},
   JOURNAL = {PNAS}, 
   VOLUME = {101},
    YEAR = {2004},
   PAGES = {7757--7762}
}

@article {RKC,
    AUTHOR = {Rangan, Aaditya V. and Kova{\v c}i{\v c}, Gregor and Cai, David},
     TITLE = {Kinetic theory for neuronal networks with fast and slow
              excitatory conductances driven by the same spike train},
   JOURNAL = {Phys. Rev. E (3)},
  FJOURNAL = {Physical Review E. Statistical, Nonlinear, and Soft Matter
              Physics},
    VOLUME = {77},
      YEAR = {2008},
    NUMBER = {4},
     PAGES = {041915, 13},
      ISSN = {1539-3755},
   MRCLASS = {92C20 (60G55 62P10 82C32)},
  MRNUMBER = {2495462},
MRREVIEWER = {Chih-Wen Shih},
       URL = {https://doi.org/10.1103/PhysRevE.77.041915},
}
		

@article {RCT,
    AUTHOR = {Rangan, Aaditya V. and Cai, David and Tao, Louis},
     TITLE = {Numerical methods for solving moment equations in kinetic
              theory of neuronal network dynamics},
   JOURNAL = {J. Comput. Phys.},
  FJOURNAL = {Journal of Computational Physics},
    VOLUME = {221},
      YEAR = {2007},
    NUMBER = {2},
     PAGES = {781--798},
      ISSN = {0021-9991},
   MRCLASS = {92C20 (65M99 82C32)},
  MRNUMBER = {2293150},
       URL = {https://doi.org/10.1016/j.jcp.2006.06.036},
}
		

@article {RCnum,
    AUTHOR = {Rangan, Aaditya V. and Cai, David},
     TITLE = {Fast numerical methods for simulating large-scale
              integrate-and-fire neuronal networks},
   JOURNAL = {J. Comput. Neurosci.},
  FJOURNAL = {Journal of Computational Neuroscience},
    VOLUME = {22},
      YEAR = {2007},
    NUMBER = {1},
     PAGES = {81--100},
      ISSN = {0929-5313},
   MRCLASS = {92C20 (65L05)},
  MRNUMBER = {2278487},
       URL = {https://doi.org/10.1007/s10827-006-8526-7},
}
		

@article {CTRM06,
    AUTHOR = {Cai, David and Tao, Louis and Rangan, Aaditya V. and
              McLaughlin, David W.},
     TITLE = {Kinetic theory for neuronal network dynamics},
   JOURNAL = {Commun. Math. Sci.},
  FJOURNAL = {Communications in Mathematical Sciences},
    VOLUME = {4},
      YEAR = {2006},
    NUMBER = {1},
     PAGES = {97--127},
      ISSN = {1539-6746},
   MRCLASS = {82C32 (37N25 82C31 92C20)},
  MRNUMBER = {2204080 (2007a:82053)},
MRREVIEWER = {Chih-Wen Shih},
       URL = {http://projecteuclid.org/euclid.cms/1145905939},
}

@article{cai2004effective,
  title={An effective kinetic representation of fluctuation-driven neuronal networks with application to simple and complex cells in visual cortex},
  author={Cai, David and Tao, Louis and Shelley, Michael and McLaughlin, David W},
  journal={Proceedings of the National Academy of Sciences},
  volume={101},
  number={20},
  pages={7757--7762},
  year={2004},
  publisher={National Acad Sciences}
}

@article{Caikinetic,
  author  = {Rangan, A. V.  and Kova{\u c}i{\u c}, G. and Cai, D.},
  title   = {Kinetic theory for neuronal networks with fast and slow excitatory conductances driven
by the same spike train},
  journal = {Physical Review E},
  year  = {2008},
    month = {},
    volume= {77},
    number= {041915},
    pages = {1-13},
    note  = {}
}


@article{RB,
  author  = {Rossant, C. and Goodman, D. F. M. and Fontaine, B. and Platkiewicz, J. and Magnusson, A. K.  and  Brette, R.},
  title   = {Fitting neuron models to spike trains},
  journal = {Frontiers in Neuroscience},
  year  = {2011},
    month = {},
    volume= {5},
    number= {9},
    pages = {1-8},
    note  = {}
}


@article{CGGS,
 author = {Carrillo, Jos{\'e} A. and Gonz{\'a}lez, Mar{\'{\i}}a D. M. and Gualdani, Maria P. and Schonbek, Maria E.},
 title = {Classical solutions for a nonlinear {Fokker}-{Planck} equation arising in computational neuroscience},
 fjournal = {Communications in Partial Differential Equations},
 journal = {Commun. Partial Differ. Equations},
 issn = {0360-5302},
 volume = {38},
 number = {1-3},
 pages = {385--409},
 year = {2013},
} 

@article{AV93,
  author  = {Abbott, L. F. and Vreeswijk, C. van},
  title   = {Asynchronous states in networks of pulse-coupled oscillators},
  journal = {Phys. Rev. E},
  year  = {1993},
    month = {},
    volume= {48},
    number= {2},
    pages = {1483-1490},
    note  = {}
}

@article {NKKRC10,
   AUTHOR = {Newhall, K. and Kova{\v{c}}i{\v{c}}, G. and Kramer, P. and 
Rangan, A. V. and Cai, D.},
    TITLE = {Cascade-Induced synchrony in stochastically driven neuronal networks},
  JOURNAL = {Phys. Rev. E},
   VOLUME = {82},
     YEAR = {2010},
     PAGES ={041903} 
}

@article {NKKZRC10,
   AUTHOR = {Newhall, K. and Kova{\v{c}}i{\v{c}}, G. and Kramer, P. and 
Zhou, D. and Rangan, A. V. and Cai, D.},
    TITLE = {Dynamics of Current-Based, Poisson Driven, Integrate-and-Fire Neuronal Networks},
  JOURNAL = {Comm. in Math. Sci.},
   VOLUME = {8},
     YEAR = {2010},
    PAGES = {541--600},
}

@article{baladron2012mean,
  title={Mean-field description and propagation of chaos in networks of Hodgkin-Huxley and FitzHugh-Nagumo neurons},
  author={Baladron, Javier and Fasoli, Diego and Faugeras, Olivier and Touboul, Jonathan and others},
  journal={The Journal of Mathematical Neuroscience},
  volume={2},
  number={1},
  pages={10},
  year={2012},
  publisher={Springer}
}

@article {BoFaTa,
    AUTHOR = {Bossy, Mireille and Faugeras, Olivier and Talay, Denis},
     TITLE = {Clarification and complement to ``{M}ean-field description and
              propagation of chaos in networks of {H}odgkin-{H}uxley and
              {F}itz{H}ugh-{N}agumo neurons''},
   JOURNAL = {J. Math. Neurosci.},
  FJOURNAL = {Journal of Mathematical Neuroscience},
    VOLUME = {5},
      YEAR = {2015},
     PAGES = {Art. 19, 23},
      ISSN = {2190-8567},
   MRCLASS = {92C20 (34F05 60F99 60J60 60K35)},
  MRNUMBER = {3392551},
       DOI = {10.1186/s13408-015-0031-8},
       URL = {https://doi-org.accesdistant.sorbonne-universite.fr/10.1186/s13408-015-0031-8},
}

@article{touboul2014propagation,
  title={Propagation of chaos in neural fields},
  author={Touboul, Jonathan},
  journal={The Annals of Applied Probability},
  volume={24},
  number={3},
  pages={1298--1328},
  year={2014},
  publisher={Institute of Mathematical Statistics}
}

@article{delarue2015global,
  title={Global solvability of a networked integrate-and-fire model of McKean--Vlasov type},
  author={Delarue, Fran{\c{c}}ois and Inglis, James and Rubenthaler, Sylvain and Tanr{\'e}, Etienne and others},
  journal={The Annals of Applied Probability},
  volume={25},
  number={4},
  pages={2096--2133},
  year={2015},
  publisher={Institute of Mathematical Statistics}
}

@article{dai1996mckean,
  title={McKean-Vlasov limit for interacting random processes in random media},
  author={Dai Pra, Paolo and den Hollander, Frank},
  journal={Journal of statistical physics},
  volume={84},
  number={3-4},
  pages={735--772},

  year={1996},
  publisher={Springer}
}

@article{bertini2010dynamical,
  title={Dynamical aspects of mean field plane rotators and the Kuramoto model},
  author={Bertini, Lorenzo and Giacomin, Giambattista and Pakdaman, Khashayar},
  journal={Journal of Statistical Physics},
  volume={138},
  number={1-3},
  pages={270--290},
  year={2010},
  publisher={Springer}
}

@article{sompolinsky1988chaos,
  title={Chaos in random neural networks},
  author={Sompolinsky, Haim and Crisanti, A and Sommers, HJ},
  journal={Physical Review Letters},
  volume={61},
  number={3},
  pages={259},
  year={1988},
  publisher={APS}
}

@article{arous1995large,
  title={Large deviations for Langevin spin glass dynamics},
  author={Ben Arous, G. and Guionnet, A.},
  journal={Probability Theory and Related Fields},
  volume={102},
  number={4},
  pages={455--509},
  year={1995},
  publisher={Springer}
}


@article{gerstner2002mathematical,
  title={Mathematical formulations of Hebbian learning},
  author={Gerstner, W. and Kistler, W. M.},
  journal={Biological cybernetics},
  volume={87},
  number={5},
  pages={404--415},
  year={2002},
  publisher={Springer}
}


@article{morrison2008phenomenological,
  title={Phenomenological models of synaptic plasticity based on spike timing},
  author={Morrison, Abigail and Diesmann, Markus and Gerstner, Wulfram},
  journal={Biological cybernetics},
  volume={98},
  number={6},
  pages={459--478},
  year={2008},
  publisher={Springer}
}

@article{burkitt2007spike,
  title={Spike-timing-dependent plasticity for neurons with recurrent connections},
  author={Burkitt, Anthony N and Gilson, Matthieu and van Hemmen, J Leo},
  journal={Biological cybernetics},
  volume={96},
  number={5},
  pages={533--546},
  year={2007},
  publisher={Springer}
}


@article{gilson2009emergence,
  title={Emergence of network structure due to spike-timing-dependent plasticity in recurrent neuronal networks iv},
  author={Gilson, Matthieu and Burkitt, Anthony N and Grayden, David B and Thomas, Doreen A and van Hemmen, J Leo},
  journal={Biological cybernetics},
  volume={101},
  number={5-6},
  pages={427--444},
  year={2009},
  publisher={Springer}
}

@book{hebb1949organization,
  title={The organization of behavior: A neuropsychological approach},
  author={Hebb, Donald Olding},
  year={1949},
  publisher={John Wiley \& Sons}
}

@book{dayan2001theoretical,
  title={Theoretical neuroscience},
  author={Dayan, Peter and Abbott, Laurence F},
  volume={806},
  year={2001},
  publisher={Cambridge, MA: MIT Press}
}

@book{gerstner2002spiking,
  title={Spiking neuron models: Single neurons, populations, plasticity},
  author={Gerstner, Wulfram and Kistler, Werner M.},
  year={2002},
  publisher={Cambridge university press}
}

@article{hopfield1982neural,
  title={Neural networks and physical systems with emergent collective computational abilities},
  author={Hopfield, John J},
  journal={Proceedings of the national academy of sciences},
  volume={79},
  number={8},
  pages={2554--2558},
  year={1982},
  publisher={National Acad Sciences}
}

@article{martin2000synaptic,
  title={Synaptic plasticity and memory: an evaluation of the hypothesis},
  author={Martin, SJ and Grimwood, PD and Morris, RGM},
  journal={Annual review of neuroscience},
  volume={23},
  number={1},
  pages={649--711},
  year={2000},
  publisher={Annual Reviews 4139 El Camino Way, PO Box 10139, Palo Alto, CA 94303-0139, USA}
}

@article{song2000competitive,
  title={Competitive Hebbian learning through spike-timing-dependent synaptic plasticity},
  author={Song, Sen and Miller, Kenneth D and Abbott, Larry F},
  journal={Nature neuroscience},
  volume={3},
  number={9},
  pages={919--926},
  year={2000},
  publisher={Nature Publishing Group}
}

@article{galtier2013biological,
  title={A biological gradient descent for prediction through a combination of stdp and homeostatic plasticity},
  author={Galtier, Mathieu N and Wainrib, Gilles},
  journal={Neural computation},
  volume={25},
  number={11},
  pages={2815--2832},
  year={2013},
  publisher={MIT Press}
}



@article {PTW1,
    AUTHOR = {Wainrib, Gilles and Thieullen, Mich\`ele and Pakdaman,
              Khashayar},
     TITLE = {Reduction of stochastic conductance-based neuron models with
              time-scales separation},
   JOURNAL = {J. Comput. Neurosci.},
  FJOURNAL = {Journal of Computational Neuroscience},
    VOLUME = {32},
      YEAR = {2012},
    NUMBER = {2},
     PAGES = {327--346},
      ISSN = {0929-5313},
   MRCLASS = {92C20},
  MRNUMBER = {2904337},
       URL = {https://doi.org/10.1007/s10827-011-0355-7},
}

@article{galtier2012multiscale,
  title={Multiscale analysis of slow-fast neuronal learning models with noise},
  author={Galtier, Mathieu and Wainrib, Gilles},
  journal={The Journal of Mathematical Neuroscience},
  volume={2},
  number={1},
  pages={1--64},
  year={2012},
  publisher={Springer}
}

@article {BossyFaugerasTalay,
    AUTHOR = {Bossy, Mireille and Faugeras, Olivier and Talay, Denis},
     TITLE = {Clarification and complement to ``{M}ean-field description and
              propagation of chaos in networks of {H}odgkin-{H}uxley and
              {F}itz{H}ugh-{N}agumo neurons''},
   JOURNAL = {J. Math. Neurosci.},
  FJOURNAL = {Journal of Mathematical Neuroscience},
    VOLUME = {5},
      YEAR = {2015},
     PAGES = {Art. 19, 23},
      ISSN = {2190-8567},
   MRCLASS = {92C20 (34F05 60F99 60J60 60K35)},
  MRNUMBER = {3392551},
       DOI = {10.1186/s13408-015-0031-8},
       URL = {http://dx.doi.org/10.1186/s13408-015-0031-8},
}

@article {Demasietal,
    AUTHOR = {De Masi, A. and Galves, A. and L{\"o}cherbach, E. and
              Presutti, E.},
     TITLE = {Hydrodynamic limit for interacting neurons},
   JOURNAL = {J. Stat. Phys.},
  FJOURNAL = {Journal of Statistical Physics},
    VOLUME = {158},
      YEAR = {2015},
    NUMBER = {4},
     PAGES = {866--902},
      ISSN = {0022-4715},
   MRCLASS = {92C20 (60F17 60J25 60K35 92B20)},
  MRNUMBER = {3311484},
       DOI = {10.1007/s10955-014-1145-1},
       URL = {http://dx.doi.org/10.1007/s10955-014-1145-1},
}
@article{hormander1967hypoelliptic,
  title={Hypoelliptic second order differential equations},
  author={H{\"o}rmander, Lars},
  journal={Acta Mathematica},
  volume={119},
  pages={147--171},
  year={1967},
  publisher={Institut Mittag-Leffler}
}

@article{keyfitz2006fichera,
  title={The Fichera function and nonlinear equations},
  author={Keyfitz, BL},
  journal={Rendiconti Accademia Nazionale delle Scienze detta dei XL. Memorie di Matematica e Applicazioni},
  volume={30},
  number={1},
  pages={83--94},
  year={2006}
 }

@article{paronetto2020further,
  title={Further existence results for elliptic--parabolic and forward--backward parabolic equations},
  author={Paronetto, Fabio},
  journal={Calculus of Variations and Partial Differential Equations},
  volume={59},
  number={4},
  pages={1--30},
  year={2020},
  publisher={Springer}
}
@article{paronetto2004existence,
  title={Existence results for a class of evolution equations of mixed type},
  author={Paronetto, Fabio},
  journal={Journal of Functional Analysis},
  volume={212},
  number={2},
  pages={324--356},
  year={2004},
  publisher={Elsevier}
}
@article{paronetto2017local,
  title={Local boundedness for forward--backward parabolic {D}e {G}iorgi classes with coefficients depending on time},
  author={Paronetto, Fabio},
  journal={Nonlinear Analysis},
  volume={158},
  pages={168--198},
  year={2017},
  publisher={Elsevier}
}
@article{paronetto2016harnack,
  title={A {H}arnack's inequality for mixed type evolution equations},
  author={Paronetto, Fabio},
  journal={Journal of Differential Equations},
  volume={260},
  number={6},
  pages={5259--5355},
  year={2016},
  publisher={Elsevier}
}

%
@article{armstrong2019variational,
      title={Variational methods for the kinetic {F}okker-{P}lanck equation}, 
      author={D. Albritton and S. Armstrong and J. -C. Mourrat and M. Novack},
      	journal={arXiv preprint arXiv:1902.04037},
      year={2021},
}
@book{hypobookinvite,
author={Bramanti,Marco},
title={An Invitation to Hypoelliptic Operators and Hörmander's Vector Fields},
publisher={Springer International Publishing},
address={Cham},
edition={2014},
keywords={Analysis; Mathematics; Mathematics and Statistics; Operator Theory; Partial Differential Equations},
isbn={2191-8198},
language={English},

}
@article{hypohairer2011malliavin,
  title={On Malliavin's proof of H\"{o}rmander's theorem},
  author={Hairer, Martin},
  journal={Bulletin des sciences mathematiques},
  volume={135},
  number={6-7},
  pages={650--666},
  year={2011},
  publisher={Elsevier}
}
@book{nier2005hypoelliptic,
  title={Hypoelliptic estimates and spectral theory for Fokker-Planck operators and Witten Laplacians},
  author={Nier, Francis and Helffer, Bernard},
  year={2005},
  publisher={Springer Science \& Business Media}
}

@article {JGLiuZZ21,
    AUTHOR = {Liu, Jian-Guo and Wang, Ziheng and Xie, Yantong and Zhang,
              Yuan and Zhou, Zhennan},
     TITLE = {Investigating the integrate and fire model as the limit of a
              random discharge model: a stochastic analysis perspective},
   JOURNAL = {Math. Neurosci. Appl.},
  FJOURNAL = {Mathematical Neuroscience and Applications},
    VOLUME = {1},
      YEAR = {2021},
     PAGES = {Art. No. 2, 36},
   MRCLASS = {92C20 (60H10 60H30 92B20)},
  MRNUMBER = {4665139},
}

@misc{dou2022exponential,
      title={Exponential convergence to equilibrium for a two-speed model with variant drift fields via the resolvent estimate}, 
      author={Xu'an Dou and Zhennan Zhou},
      year={2022},
      note="arXiv,2201.12494"
}
@article{CDZ2022,
  title={A simplified voltage-conductance kinetic model for interacting neurons and its asymptotic limit},
  author={Carrillo, Jos{\'e} A and Dou, Xu'an and Zhou, Zhennan},
  journal={arXiv preprint arXiv:2203.02746},
  year={2022}
}
@article{hill1970sharp,
  AUTHOR = {Hill, C. Denson},
     TITLE = {A sharp maximum principle for degenerate elliptic-parabolic
              equations},
   JOURNAL = {Indiana Univ. Math. J.},
  FJOURNAL = {Indiana University Mathematics Journal},
    VOLUME = {20},
      YEAR = {1970/71},
     PAGES = {213--229},
      ISSN = {0022-2518},
   MRCLASS = {35.70},
  MRNUMBER = {287175},
MRREVIEWER = {J. K. Oddson},
       DOI = {10.1512/iumj.1970.20.20020},
       URL = {https://doi-org.accesdistant.sorbonne-universite.fr/10.1512/iumj.1970.20.20020},
}

@article{arnold2014sharp,
  title={Sharp entropy decay for hypocoercive and non-symmetric Fokker-Planck equations with linear drift},
  author={Arnold, Anton and Erb, Jan},
  journal={arXiv preprint arXiv:1409.5425},
  year={2014}
}
@inproceedings{bony1969principe,
  title={Principe du maximum, in{\'e}galit{\'e} de Harnack et unicit{\'e} du probleme de Cauchy pour les op{\'e}rateurs elliptiques d{\'e}g{\'e}n{\'e}r{\'e}s},
  author={Bony, Jean-Michel},
  booktitle={Annales de l'institut Fourier},
  volume={19},
  number={1},
  pages={277--304},
  year={1969}
}

@article{kovavcivc2009fokker,
  title={Fokker-Planck description of conductance-based integrate-and-fire neuronal networks},
  author={Kova{\v{c}}i{\v{c}}, Gregor and Tao, Louis and Rangan, Aaditya V and Cai, David},
  journal={Physical Review E},
  volume={80},
  number={2},
  pages={021904},
  year={2009},
  publisher={APS}
}
@incollection {FicheraMR0111931,
    AUTHOR = {Fichera, Gaetano},
     TITLE = {On a unified theory of boundary value problems for
              elliptic-parabolic equations of second order},
 BOOKTITLE = {Boundary problems in differential equations},
     PAGES = {97--120},
 PUBLISHER = {Univ. Wisconsin Press, Madison, Wis.},
      YEAR = {1960},
   MRCLASS = {35.00},
  MRNUMBER = {0111931},
MRREVIEWER = {A. N. Milgram},
}
@article {DerridjMR601055,
    AUTHOR = {Derridj, Makhlouf},
     TITLE = {Un probl\`eme aux limites pour une classe d'op\'{e}rateurs du second
              ordre hypoelliptiques},
   JOURNAL = {Ann. Inst. Fourier (Grenoble)},
  FJOURNAL = {Universit\'{e} de Grenoble. Annales de l'Institut Fourier},
    VOLUME = {21},
      YEAR = {1971},
    NUMBER = {4},
     PAGES = {99--148},
      ISSN = {0373-0956},
   MRCLASS = {35H05},
  MRNUMBER = {601055},
       URL = {http://www.numdam.org/item?id=AIF_1971__21_4_99_0},
}

@article {CattiauxMR1182641,
    AUTHOR = {Cattiaux, Patrick},
     TITLE = {Stochastic calculus and degenerate boundary value problems},
   JOURNAL = {Ann. Inst. Fourier (Grenoble)},
  FJOURNAL = {Universit\'{e} de Grenoble. Annales de l'Institut Fourier},
    VOLUME = {42},
      YEAR = {1992},
    NUMBER = {3},
     PAGES = {541--624},
      ISSN = {0373-0956},
   MRCLASS = {60H30 (35R60 60J35 60J60)},
  MRNUMBER = {1182641},
MRREVIEWER = {Kazuaki Taira},
       URL = {http://www.numdam.org/item?id=AIF_1992__42_3_541_0},
}

@article {MR1077278,
    AUTHOR = {Lin, Chang Shou and Tso, Kaising},
     TITLE = {On regular solutions of second order degenerate
              elliptic-parabolic equations},
   JOURNAL = {Comm. Partial Differential Equations},
  FJOURNAL = {Communications in Partial Differential Equations},
    VOLUME = {15},
      YEAR = {1990},
    NUMBER = {9},
     PAGES = {1329--1360},
      ISSN = {0360-5302},
   MRCLASS = {35M10 (35B65)},
  MRNUMBER = {1077278},
MRREVIEWER = {Min You Qi},
       DOI = {10.1080/03605309908820727},
       URL = {https://doi.org/10.1080/03605309908820727},
}
@article {eigenMR912210,
    AUTHOR = {Okumura, Hirozo},
     TITLE = {Boundary value problems and eigenvalue problems for second
              order equations with nonnegative characteristic form},
   JOURNAL = {Comm. Partial Differential Equations},
  FJOURNAL = {Communications in Partial Differential Equations},
    VOLUME = {12},
      YEAR = {1987},
    NUMBER = {12},
     PAGES = {1365--1387},
      ISSN = {0360-5302},
   MRCLASS = {35J25 (35K20)},
  MRNUMBER = {912210},
MRREVIEWER = {G. F. Roach},
       DOI = {10.1080/03605308708820533},
       URL = {https://doi.org/10.1080/03605308708820533},
}

@book {Oleinik,
    AUTHOR = {Oleĭnik, O. A. and Radkevič, E. V.},
     TITLE = {Second order equations with nonnegative characteristic form},
      NOTE = {Translated from the Russian by Paul C. Fife},
 PUBLISHER = {Plenum Press, New York-London},
      YEAR = {1973},
     PAGES = {vii+259},
      ISBN = {0-306-30751-0},
   MRCLASS = {35-02 (35G05 35JXX)},
  MRNUMBER = {0457908},
}
%
@article{radkevich2009equations,
  title={Equations with nonnegative characteristic form. I},
  author={Radkevich, E. V.},
  journal={Journal of Mathematical Sciences},
  volume={158},
  number={3},
  pages={297--452},
  year={2009},
  publisher={Springer}
}


%
@article{fu2021fokker,
  title={Fokker--Plank system for movement of micro-organism population in confined environment},
  author={Fu, Jingyi and Perthame, Benoit and Tang, Min},
  journal={Journal of Statistical Physics},
  volume={184},
  number={1},
  pages={1--25},
  year={2021},
  publisher={Springer}
}


@article{zbMATH07654553,
 author = {Ca{\~n}izo, Jos{\'e} A. and Mischler, St{\'e}phane},
 title = {Harris-type results on geometric and subgeometric convergence to equilibrium for stochastic semigroups},
 fjournal = {Journal of Functional Analysis},
 journal = {J. Funct. Anal.},
 issn = {0022-1236},
 volume = {284},
 number = {7},
 pages = {46},
 note = {Id/No 109830},
 year = {2023},
 language = {English},
 doi = {10.1016/j.jfa.2022.109830},
 keywords = {47D07,60J05,35B40},
 zbMATH = {7654553},
 Zbl = {1523.47052}
}

@article{Muckenhoupt1972,
author = {Muckenhoupt, Benjamin},
journal = {Studia Mathematica},
language = {eng},
number = {1},
pages = {31-38},
title = {Hardy's inequality with weights},
url = {http://eudml.org/doc/217718},
volume = {44},
year = {1972},
}

@article{zbMATH07903423,
 author = {Salort, Delphine and Smets, Didier},
 title = {Convergence towards equilibrium for a model with partial diffusion},
 fjournal = {Communications in Partial Differential Equations},
 journal = {Commun. Partial Differ. Equations},
 issn = {0360-5302},
 volume = {49},
 number = {5-6},
 pages = {410--427},
 year = {2024},
 language = {English},
 doi = {10.1080/03605302.2024.2344806},
 keywords = {35Q92,35Q84,92B20,35B40,82C32},
 zbMATH = {7903423},
 Zbl = {1546.35241}
}

@article{Ambrogi_2026,
doi = {10.1088/1361-6544/ae4cec},
url = {https://doi.org/10.1088/1361-6544/ae4cec},
year = {2026},
month = {mar},
publisher = {IOP Publishing},
volume = {39},
number = {3},
pages = {035013},
author = {Ambrogi, Elena and He, Qingyou and Salort, Delphine},
title = {Nonlinear stability for a two-dimensional Fokker-Planck equation with partial diffusion in neuroscience},
journal = {Nonlinearity},
}

@article{zbMATH07047481,
 author = {C{\'a}ceres, Mar{\'{\i}}a and Schneider, Ricarda},
 title = {Analysis and numerical solver for excitatory-inhibitory networks with delay and refractory periods},
 fjournal = {European Series in Applied and Industrial Mathematics (ESAIM): Mathematical Modelling and Numerical Analysis},
 journal = {ESAIM, Math. Model. Numer. Anal.},
 issn = {0764-583X},
 volume = {52},
 number = {5},
 pages = {1733--1761},
 year = {2018},
 language = {English},
 doi = {10.1051/m2an/2018014},
 keywords = {35K60},
 zbMATH = {7047481},
 Zbl = {1411.35178}
}


\begin{thebibliography}{10}

\bibitem{Ambrogi_2026}
{\sc E.~Ambrogi, Q.~He, and D.~Salort}, {\em Nonlinear stability for a
  two-dimensional fokker-planck equation with partial diffusion in
  neuroscience}, Nonlinearity, 39 (2026), p.~035013.

\bibitem{BrHa}
{\sc N.~Brunel and V.~Hakim}, {\em Fast global oscillations in networks of
  integrate-and-fire neurons with long firing rates}, Neural Computation, 11
  (1999), pp.~1621--1671.

\bibitem{caceres2018global}
{\sc M.~C{\'a}ceres, P.~Roux, D.~Salort, and R.~Schneider}, {\em Global-in-time
  classical solutions and qualitative properties for the nnlif neuron model
  with synaptic delay}, arXiv preprint arXiv:1806.01934,  (2018).

\bibitem{zbMATH07047481}
{\sc M.~C{\'a}ceres and R.~Schneider}, {\em Analysis and numerical solver for
  excitatory-inhibitory networks with delay and refractory periods}, ESAIM,
  Math. Model. Numer. Anal., 52 (2018), pp.~1733--1761.

\bibitem{zbMATH08030978}
{\sc M.~J. C{\'a}ceres, J.~A. Ca{\~n}izo, and A.~Ramos-Lora}, {\em On the
  asymptotic behavior of the {NNLIF} neuron model for general connectivity
  strength}, Commun. Math. Phys., 406 (2025), p.~55.
\newblock Id/No 115.

\bibitem{CCP}
{\sc M.~J. C\'aceres, J.~A. Carrillo, and B.~Perthame}, {\em Analysis of
  nonlinear noisy integrate \& fire neuron models: blow-up and steady states},
  J. Math. Neurosci., 1 (2011), pp.~Art. 7, 33.

\bibitem{CaPe}
{\sc M.~J. C{\'a}ceres and B.~Perthame}, {\em Beyond blow-up in excitatory
  integrate and fire neuronal networks: refractory period and spontaneous
  activity}, J. Theoret. Biol., 350 (2014), pp.~81--89.

\bibitem{CGGS}
{\sc J.~A. Carrillo, M.~D.~M. Gonz{\'a}lez, M.~P. Gualdani, and M.~E.
  Schonbek}, {\em Classical solutions for a nonlinear {Fokker}-{Planck}
  equation arising in computational neuroscience}, Commun. Partial Differ.
  Equations, 38 (2013), pp.~385--409.

\bibitem{CPSS}
{\sc J.~A. Carrillo, B.~Perthame, D.~Salort, and D.~Smets}, {\em Qualitative
  properties of solutions for the noisy integrate and fire model in
  computational neuroscience}, Nonlinearity, 28 (2015), pp.~3365--3388.

\bibitem{RouxCar}
{\sc J.~A. Carrillo and P.~Roux}, {\em Nonlinear partial differential equations
  in neuroscience: From modeling to mathematical theory}, Mathematical Models
  and Methods in Applied Sciences, 35 (2025), pp.~403--584.

\bibitem{CaceresCanizo2024}
{\sc M.~J. Cáceres, J.~A. Cañizo, and A.~Ramos-Lora}, {\em Sequence of
  pseudoequilibria describes the long-time behavior of the nonlinear noisy
  leaky integrate-and-fire model with large delay}, 2024.
\newblock ar{X}iv/2403.00971.

\bibitem{IRSS2022}
{\sc K.~Ikeda, P.~Roux, D.~Salort, and D.~Smets}, {\em Theoretical study of the
  emergence of periodic solutions for the inhibitory {NNLIF} neuron model with
  synaptic delay}, Math. Neurosci. Appl., 2 (2022), pp.~Art. No. 4, 37.

\bibitem{zbMATH01061253}
{\sc G.~M. Lieberman}, {\em Second order parabolic differential equations},
  Singapore: World Scientific, 1996.

\bibitem{zbMATH08113917}
{\sc B.~Perthame, C.~Rieutord, and D.~Salort}, {\em Strongly nonlinear
  age-structured equation, time-elapsed model and large delays}, J. Math.
  Biol., 91 (2025), p.~29.
\newblock Id/No 65.

\bibitem{arXiv:2601.19282}
\leavevmode\vrule height 2pt depth -1.6pt width 23pt, {\em A {Fokker}-{Planck}
  equation with superlinear drift at infinity for {Integrate}-and-{Fire}
  model}.
\newblock Preprint, {arXiv}:2601.19282 [math.{AP}] (2026), 2026.

\bibitem{zbMATH07903423}
{\sc D.~Salort and D.~Smets}, {\em Convergence towards equilibrium for a model
  with partial diffusion}, Commun. Partial Differ. Equations, 49 (2024),
  pp.~410--427.

\end{thebibliography}

\end{document}